\documentclass[11pt, reqno]{amsart}
\title[Shadows of 2-knots and complexity]
{Shadows of 2-knots and complexity}

\author{Hironobu Naoe}
\address{
Department of Mathematics, Chuo University, 1-13-27 Kasuga Bunkyo-ku, Tokyo, 112-8551, Japan}
\email{naoe@math.chuo-u.ac.jp}

\usepackage{amsfonts,amsmath,amssymb,amscd}
\usepackage{amsthm}
\usepackage{latexsym}
\usepackage{bm}
\usepackage[dvipdfmx]{graphicx}
\usepackage{psfrag}
\usepackage{color}
\usepackage{url}
\usepackage{ulem}

\theoremstyle{plain}
\newtheorem{theorem}{Theorem}[section]
\newtheorem{lemma}[theorem]{Lemma}

\newtheorem{proposition}[theorem]{Proposition}

\theoremstyle{definition}
\newtheorem{definition}[theorem]{Definition}

\newtheorem{remark}[theorem]{Remark}

\newtheorem{step}{Step}
\newtheorem{case}{Case}

\newtheoremstyle{mycitation}%
    {}%
    {}%
    {\it}%
    {}%
    {\bf}%
    {}%
    {5pt}%
    {\thmname{#1} \thmnumber{#2}\thmnote{#3}.}%
\theoremstyle{mycitation}
\newtheorem*{citingtheorem}{Theorem}

\newtheoremstyle{myclaim}%
    {}%
    {}%
    {\it}%
    {}%
    {\it}%
    {}%
    {5pt}%
    {\thmname{#1} \thmnumber{#2}\thmnote{#3}.}%
\theoremstyle{myclaim}
\newtheorem{claim}{Claim}

\newcommand{\K}{\mathcal{K}}

\newcommand{\Z}{\mathbb{Z}}

\newcommand{\CP}{\mathbb{CP}^2}

\newcommand{\Int}{\mathrm{Int}}

\newcommand{\Nbd}{\mathrm{Nbd}}
\newcommand{\pr}{\mathrm{pr}}

\newcommand{\Cl}{\mathrm{Cl}}

\newcommand{\gl}{\mathfrak{gl}}

\newcommand{\shco}{\mathrm{sc}}
\newcommand{\spshco}{\mathrm{sc}^{\mathrm{sp}}}
\newcommand{\vt}[1]{(\mathrm{#1})}

\newcommand{\lrangle}[2]{\left\langle #1 \mathrel{}\middle|\mathrel{} #2 \right\rangle}

\definecolor{darkred}{rgb}{.80,.0,.0}

\definecolor{bl}{gray}{0.7}

\newlength{\myheight}
\newlength{\myheighta}
\makeatletter

\long\def\@makecaption#1#2{
  \small
  \vskip\abovecaptionskip
  \sbox\@tempboxa{#1. #2}
  \ifdim \wd\@tempboxa >\hsize
    #1. #2\par
  \else
    \global \@minipagefalse
    \hb@xt@\hsize{\hfil\box\@tempboxa\hfil}
  \fi
  \vskip\belowcaptionskip}

\makeatother
\begin{document}
\begin{abstract}
We introduce a new invariant for a $2$-knot in $S^4$, called the shadow-complexity, 
based on the theory of Turaev shadows, 
and we give a characterization of $2$-knots with shadow-complexity at most $1$. 
Specifically, we show that the unknot is the only $2$-knot with shadow-complexity $0$ 
and that there exist infinitely many $2$-knots with shadow-complexity $1$. 
\end{abstract}

\maketitle
\setcounter{tocdepth}{1}
\tableofcontents
\section{Introduction}

A $2$-{\it knot} is a smoothly embedded $2$-sphere in $S^4$. 
The first example of a non-trivial $2$-knot was given by Artin in \cite{Art25}, 
which is so-called a spun knot. 
The non-triviality is of fundamental interest in knot theory. 
In the classical case, namely the $1$-knot case, 
the trefoil knot can be said to be the simplest non-trivial knot 
in some terms of numerical invariants: 
the crossing number, the bridge number, the tunnel number, and so forth. 
Then one naturally wonder which $2$-knot is the simplest non-trivial one. 
The answer to this question should be 
based on as many criteria as possible that measures, in some sense, 
how different a given $2$-knot is from the trivial one. 

There are several studies on numerical invariants for $2$-knots. 
For examples, 
we refer the reader 
to \cite{Sat00, SS04, Yaj64} for the triple point number and 
to \cite{SS05, Sat09} for the sheet number. 
These two invariants are defined with broken surface diagrams \cite{CRS97, Ros98}, 
which is a natural analogue of classical knot diagrams. 
Yoshikawa introduced the ch-index by using ch-diagrams (also called marked graph diagrams), 
and he gave the table of $2$-knots with ch-index up to 10 in \cite{Yos94}. 
These invariants measure how complicated the descriptions of a given $2$-knot must be, 
like the crossing numbers for $1$-knots. 
On the one hand, for examples, the bridge number, the tunnel number for classical knots seem to be the complexity of the embedding or the complement. 
Recently, bridge positions of knotted surfaces, called bridge trisection, was introduced by Meier and Zupan in \cite{MZ17, MZ18} 
using a trisection (see \cite{GK16} for the details) of the ambient $4$-manifold, 
and the bridge numbers for knotted surfaces were defined by using this notion. 
Incidentally, Kirby and Thompson introduced an integer-valued invariant, 
called the $\mathcal{L}$-invariant or the Kirby-Thompson length, for closed $4$-manifolds, 
and this notion was adapted to the knotted surfaces in \cite{BCTT22}. 

In the present paper, we propose to study $2$-knots using a Turaev's {\it shadow}. 
The notion of shadows was introduced by Turaev \cite{Tur94} in the 1990s. 
A shadow is a simple polyhedron embedded in a closed $4$-manifold as a $2$-skeleton, 
in other words, a simple polyhedron such that the complement of its neighborhood is diffeomorphic to a $4$-dimensional $1$-handlebody. 
Turaev showed that regions of a shadow are naturally equipped with half-integers such as ``relative self-intersection numbers'', 
which has a sufficient information to reconstruct the ambient $4$-manifold. 
It is known as Turaev's reconstruction theorem. 
Thus, a shadow can be treated as a description of a $4$-manifold, which brings about several studies: 
Stein structures, spin${}^c$ structures and complex structures, 
stable maps and hyperbolic structures on $3$-manifolds, Lefschetz fibrations, and so on, see \cite{Cos06,Cos08,CT08,IK17,IN20} for examples. 
Moreover, as another worth of the theory of shadows, 
we can define a complexity for $4$-manifolds, called the {\it shadow-complexity}. 
Costantino introduced this notion in \cite{Cos06b}. 
The shadow-complexity of a $4$-manifold is defined as the minimum number of specific points called true vertices contained in a shadow of the $4$-manifold. 
He also gave the classification of closed $4$-manifold with complexity up to $1$ in the special case. 
Martelli gave a characterization of all the closed $4$-manifolds with shadow-complexity $0$ in \cite{Mar11}. 
He also showed in the paper that a closed simply-connected $4$-manifold with shadow-complexity $0$ is diffeomorphic to 
$S^4$ or the connected sum of some copies of $\CP$, $\overline{\CP}$ and $S^2\times S^2$. 
This implies that the shadow-complexity is a diffeomorphism invariant. 
Note that exotic smooth structures on $\CP\# k\overline{\CP}$ have been found for some $k$. 
In \cite{KMN18}, Koda, Martelli and the present author introduced a new complexity, called the connected shadow-complexity, 
and they gave a characterization of all the closed $4$-manifolds with connected shadow-complexity at most $1$. 

If a surface $F$ is embedded in a shadow of a $4$-manifold, 
then such $F$ in the $4$-manifold is smoothly embedded and generally knotted. 
In view of this situation, we define a shadow of a $2$-knot $K$ 
as a shadow $X$ of the ambient $4$-manifold $S^4$ such that $K$ is embedded in $X$. 
Of course, we can define a shadow for a knotted surface as well (c.f. Remark~\ref{rmk:non-trivial_Kirby_diag}), 
but the focus in this paper is $2$-knots. 

Every closed $4$-manifold admits a shadow, which follows from the existence of a handle decomposition. 
We can also show the following. 
\begin{citingtheorem}
[\ref{thm:2-knot_shadow}]
Every $2$-knot admits a shadow. 
\end{citingtheorem}
The above theorem allows us to define a complexity for $2$-knots
\[
\{2\text{-knots in }S^4\}
\to\Z_{\geq0} 
\]
as the minimum number of true vertices in a shadow of $K$. 
We call this number the {\it shadow-complexity} of $K$ and write it as $\shco(K)$. 

The aim of this paper is to give the classification of $2$-knots with shadow-complexity at most $1$. 
Firstly, one might expect that the unknot is the only $2$-knot with shadow-complexity $0$ 
as well as some other numerical invariants for $1$- or $2$-knots, which is indeed true. 
\begin{citingtheorem}
[\ref{thm:complexity0}]
A $2$-knot has shadow-complexity $0$ if and only if it is unknotted. 
\end{citingtheorem}
The unknotted $2$-knot admits a special shadow with no true vertices, 
which implies that the unknot is also the only $2$-knot with special shadow-complexity $0$. 
The condition for a shadow to be special is fairly strong, 
and any special polyhedron with one true vertex can not be a shadow of any $2$-knot except for the unknot. 
\begin{citingtheorem}
[\ref{thm:no_spsc=1}]
There are no $2$-knots with special shadow-complexity $1$.
\end{citingtheorem}
It is noteworthy that the special shadow-complexity for closed $4$-manifolds is a finite-to-one invariant \cite[Corollary 2.7]{Mar05}.
However, that for $2$-knots is not finite-to-one as noted in Remark~\ref{rmk:not_finite-to-one}, 
where we find infinitely many $2$-knots with special shadow-complexity at most $5$. 
We actually have not determined the special shadow-complexity of any non-trivial $2$-knot. 
Note that all the special polyhedra with complexity up to $2$ are listed in \cite[Appendix B]{KN20}, 
so we already have possible candidates of shadows of $2$-knots with special shadow-complexity $2$ if such a $2$-knot exists. 

Before stating the theorem on the complexity $1$ case, we introduce a series of $2$-knots. 
For $n\in\Z$, 
let $K_n$ be a $2$-knot presented by a banded unlink diagram shown in Figure~\ref{fig:Kn}. 
See \cite{HKM20} and Section~\ref{sec:Banded unlink diagrams} for the definition and the details of banded unlink diagrams. 
Note that $K_0$ is trivial. 
The knot $K_{-1}$ is the spun trefoil, and $K_{1}$ is the knot $9_1$ in Yoshikawa's table \cite{Yos94}. 
For any $n\in\Z$, the knot $K_n$ is a ribbon $2$-knots. 
\begin{figure}[tbp]
\includegraphics[width=45mm]{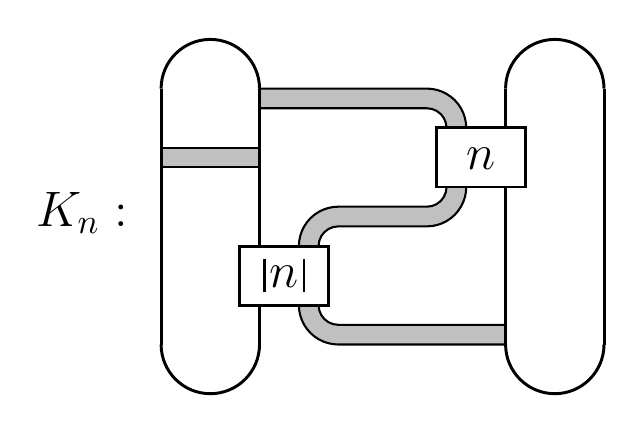}
\caption{Banded unlink diagram of $K_n$ for $n\in\Z$. 
Note that the number written in the lower-left box is the absolute value $|n|$ of $n$. }
\label{fig:Kn}
\end{figure}
As stated in Proposition~\ref{prop:Kn_diffeo}, two $2$-knots $K_n$ and $K_{n'}$ are not equivalent unless $n=n'$, 
which can be distinguished by their Alexander polynomials. 

The following is the main theorem for $2$-knots with shadow-complexity $1$. 
\begin{citingtheorem}
[\ref{thm:sc=1}]
A $2$-knot $K$ with $G(K)\not\cong\Z$ has shadow-complexity $1$ if and only if 
$K$ is diffeomorphic to $K_n$ for some non-zero integer $n$. 
\end{citingtheorem}
The unknotting conjecture says that a $2$-knot is unknotted if its knot group is infinite cyclic, 
which is known to be true in the topological category in \cite{FQ90} 
and is also studied in the smooth category in \cite{Kaw21} (see also \cite{Kaw22}). 
Supposing the unknotting conjecture is true in the smooth category, 
we obtain the complete classification of all $2$-knots with shadow-complexity at most $1$. 

One interpretation of the shadow-complexity for $2$-knots is an analogue to the tunnel number for $1$-knots. 
We recall that the tunnel number of a $1$-knot is the minimum number of $1$-cells 
such that the complement of the neighborhood of the union of the $1$-knot and the $1$-cells is a $3$-dimensional $1$-handlebody. 
The procedure to make the complement trivial is similar to a construction of a shadow from a $2$-knot. 
By definition, a shadow of a $2$-knot $K$ can be obtained from $K$ by attaching $1$-cells and $2$-cells 
so that the complement of the neighborhood is diffeomorphic to a $4$-dimensional $1$-handlebody. 
In this sense, the shadow-complexity can seem to measure the non-triviality of the complement of a given $2$-knot. 

\subsection*{Organization}
In Section~\ref{sec:Preliminaries}, 
we give a review of the theory of Turaev shadows and encoding graphs. 
In Section~\ref{sec:Shadows of 2-knots and knot groups}, 
we define a shadow of a $2$-knot, 
and give a presentation of the knot group using a shadow. 
In Section~\ref{sec:Banded unlink diagrams}, 
we introduce a banded unlink diagram, from which we construct a shadow of a $2$-knot. 
In Section~\ref{sec:Modifications and fundamental groups}, 
we introduce two important modifications: compressing disk addition and connected-sum and reduction. 
In Section~\ref{sec:2-knots with complexity zero}, 
we give the proof for the complexity zero case. 
A large part of Section~\ref{sec:Existence of 2-knots with complexity one} is devoted to 
compute the knot groups of shadows having one true vertex. 
In Section~\ref{sec:Classification of 2-knots with complexity one}, 
we give the proofs for the complexity one case with describing the $2$-knots in banded unlink diagrams. 

\subsection*{Acknowledgement}
The author would like to express his gratitude to Masaharu Ishikawa, Seiichi Kamada and 
Yuya Koda for many valuable suggestions and their kindness, 
and also to Mizuki Fukuda for useful discussion. 
This work was supported by JSPS KAKENHI Grant Number 20K14316. 
\section{Preliminaries}
\label{sec:Preliminaries}
For polyhedral spaces $A\subset B$, 
let $\Nbd(A;B)$ denote a regular neighborhood of $A$ in $B$. 
If $B$ collapses onto $A$, we use the notation $B\searrow A$. 

For integers $0\leq k\leq n$, 
an $n$-{\it dimensional} $k$-{\it handlebody} is an $n$-dimensional manifold 
admitting a handle decomposition consisting of handles whose indices are at most $k$. 

Throughout this paper, we assume that any manifold is compact, connected, oriented and smooth unless otherwise noted. 

\subsection{Simple polyhedra and shadows of $4$-manifolds with boundary}
A {\it simple polyhedron} is a compact connected space whose every point has 
a regular neighborhood homeomorphic to one of (i)-(v) in Figure~\ref{fig:local_model}. 
\begin{figure}[tbp]
\includegraphics[width=1\hsize]{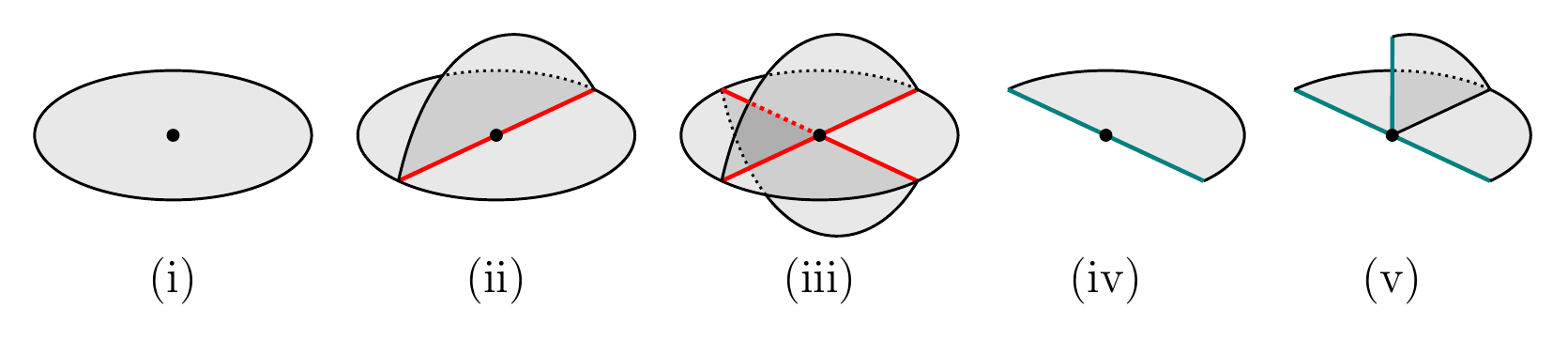}
\caption{Local models of simple polyhedra.}
\label{fig:local_model}
\end{figure}
Let $X$ be a simple polyhedron. 
The {\it singular set} $S(X)$ of $X$ is the set of points 
of type~(ii), (iii) or (v) in Figure~\ref{fig:local_model}. 
The point of type~(iii) is called a {\it true vertex}. 
Let $c(X)$ denote the number of true vertices contained in $X$, and this number is called the {\it complexity} of $X$. 
A connected component consisting of points of type~(ii) is called a {\it triple line}. 
The {\it boundary} $\partial X$ of $X$ is the set of points of type~(iv) or (v). 
It is clear that the boundary of a simple polyhedron is a 
(possibly non-connected) trivalent graph. 
If $\partial X=\emptyset$, we say that $X$ is {\it closed}. 
A {\it region} of $X$ is a connected component of $X\setminus S(X)$. 
A region is called a {\it boundary region} 
if it contains a point of $\partial X$ (or equivalently, a point of type~(iv)), 
otherwise it is called an {\it internal region}.
If every region is simply-connected, then $X$ is called a {\it special polyhedron}. 

We then define a shadow of a $4$-manifold with boundary. 
\begin{definition}
Let $M$ be a $4$-manifold with boundary. 
A simple polyhedron $X$ properly embedded in $M$ is called a {\it shadow} of $M$ 
if $X$ is locally flat in $M$ and $M\searrow X$. 
\end{definition}
Note that $X$ is {\it locally flat} in $M$ if for any $x\in X$, 
there exists a local chart $(U_x,\phi_x)$ around $x$ in $M$ such that $X\cap U_x$ is contained in a smooth $3$-ball in $U_x$. 

The notion of a shadow was introduced by Turaev \cite{Tur94}, and he proved the following. 
\begin{theorem}
[Turaev, \cite{Tur94}]
\label{thm:Tuarev}
Any $4$-dimensional $2$-handlebody admits a shadow. 
\end{theorem}

We next define {\it gleams} of regions of a shadow. 
Let $R$ be an internal region of $X$, and set $X_S=\Nbd(S(X);X)$. 
Then there exists a (possibly non-orientable) $3$-dimensional $1$-handlebody $H_S$ in $M$ 
such that $X_S$ is properly embedded in $H_S$ and $H_S \searrow X_S$. 
Set $\bar{R}=R\setminus\Int H_S$. 
Its boundary $\partial\bar{R} = R\cap\partial H_S$ 
forms a disjoint union of circles in the surface $\partial H_S$. 
Set $B=\Nbd(\bar{R};\partial H_S)$, 
which is a disjoint union of some annuli or M\"obius bands. 
Let $\bar{R}'$ be a small perturbation of $\bar{R}$ in $M$ 
such that $\partial \bar{R}'\subset B$ and $|\bar{R}\cap\bar{R}'|<\infty$. 
Then the {\it gleam} $\gl(R)$ is defined as 
\begin{align}
\label{align:gleam}
 \gl(R)=\# (\Int\bar{R}\cap\Int\bar{R}') + \frac12 \# (\partial \bar{R}\cap\partial \bar{R}'), 
\end{align}
where $\#$ is the algebraic intersection number. 

The number of the M\"obius bands in $B$ is actually determined only by the 
combinatorial structure of $X$, that is, it does not depends on the embedding $X\subset M$. 
If it is even, the region $R$ is said to be {\it even}, otherwise {\it odd}. 
The gleam $\gl(R)$ is an integer if and only if $R$ is even. 

\subsection{Shadowed polyhedra and shadows of closed $4$-manifolds}
For a simple polyhedron $X$, 
we assign a half integer to each internal region $R$ of $X$ 
so that it is an integer if and only if $R$ is even. 
Such a polyhedron is called a {\it shadowed polyhedron}. 

The following theorem is known as Turaev's reconstruction theorem. 

\begin{theorem}[Turaev, \cite{Tur94}]
\label{thm:Turaev_reconst}
There is a canonical way to construct a $4$-manifold $M_X$ with boundary 
from a shadowed polyhedron $X$ such that $X$ is a shadow of $M_X$. 
Moreover, the gleam of an internal region of $X$ coincides with that coming from 
the formula~(\ref{align:gleam}). 
\end{theorem}


The gleam is a kind of ``local self-intersection number'' 
as one can see in the formula~(\ref{align:gleam}). 
Indeed, the intersection form for the $4$-manifold $M_X$, reconstructed from a shadowed polyhedron $X$, 
can be calculated with the gleam (see \cite{Tur94} and the next subsection). 
Especially, if a closed surface $F$ is embedded in a shadowed polyhedron $X$, 
the sum of all the gleams of regions contained in $F$ coincides with the self-intersection number of $F$ in $M_X$. 

We then define a shadow for a closed $4$-manifold. 
\begin{definition}
Let $W$ be a closed $4$-manifold. 
A simple polyhedron $X$ embedded in $W$ is called a {\it shadow} of $W$ 
if it is local flatly in $W$ and $W\setminus \Int\Nbd(X;W)$ is diffeomorphic to a $4$-dimensional $1$-handlebody. 
\end{definition}
By definition, $\partial \Nbd(X;W)$ must be diffeomorphic to the connected-sum of some copies of $S^1\times S^2$. 
Thus, if $X$ is a shadowed polyhedron and $\partial M_X$ is diffeomorphic to 
the connected-sum of some copies of $S^1\times S^2$, 
then $X$ is a shadow of some closed $4$-manifold by \cite{LP72}. 

We then define the complexities of $4$-manifold that was introduced by Costantino in \cite{Cos06b}. 
\begin{definition}
For a $4$-manifold $W$, 
The {\it shadow-complexity} $\shco(W)$ and the {\it special shadow-complexity} $\spshco(W)$ of $W$
are the minimum number of true vertices of all shadows of $W$ and 
that of all special shadows of $W$, respectively. 
\end{definition}
\subsection{Intersection forms}
%
Let $X$ be a shadow of a closed $4$-manifold $W$. 
Since $X$ can be considered as a $2$-skeleton of $W$, 
the inclusion $\iota:X\hookrightarrow W$ induces an epimorphism $\iota_*:H_2(X)\to H_2(W)$. 

We equip each orientable region $R_i$ of $X$ with an orientation arbitrarily. 
Then any element in $H_2(X)$ is represented by a sum $\sum a_i R_i$ 
of oriented regions $R_1,\ldots,R_n$ with integer coefficients $a_1\ldots,a_n\in\Z$. 
Defining the intersection form $Q_X$ on $H_2(X)$ as
\[
Q_X\left(\sum a_i R_i, \sum b_i R_i\right) = \sum a_i b_i \gl(R_i), 
\]
we can calculate the intersection form $Q_W$ on $H_2(W)$ as 
\[
Q_M(\alpha,\beta)=Q_X\left(\sum a_i R_i, \sum b_i R_i\right), 
\]
where $\alpha$ and $\beta$ are the images of $\sum a_i R_i$ and $\sum b_i R_i$ by $\iota_*$, respectively.
See \cite{Tur94} for the details. 

\subsection{Topological types of neighborhoods of singular sets with $c\leq1$}
\label{subsec:Top types of nbd of S(X)}
Let $X$ be a simple polyhedron and $S$ be a connected component of $S(X)$. 
Here we give a review on the topological types of $\Nbd(S;X)$ in the cases $c(\Nbd(S;X))=0$ and $1$. 
Note that $\Nbd(S;X)$ itself is a simple polyhedron. 

First suppose $c(\Nbd(S;X))=0$, that is, $S$ is a circle. 
There are three possibilities $Y_{3},\ Y_{12},\ Y_{111}$ for topological types of $\Nbd(S;X)$. 
These simple polyhedra are interpreted as follows. 
Let $\pi:S^1\times B^3\to S^1$ be the canonical projection, and 
let $L_{3},\ L_{12}$ and $L_{111}$ be the links in $S^1\times B^3$ given in Figure~\ref{fig:0-vertex_link}. 
Then $Y_{3},\ Y_{12}$ and $Y_{111}$ are the mapping cylinders of $\pi$ restricted to $L_{3},\ L_{12}$ and $L_{111}$, respectively. 
\begin{figure}[tbp]
\centering
\includegraphics[width=60mm]{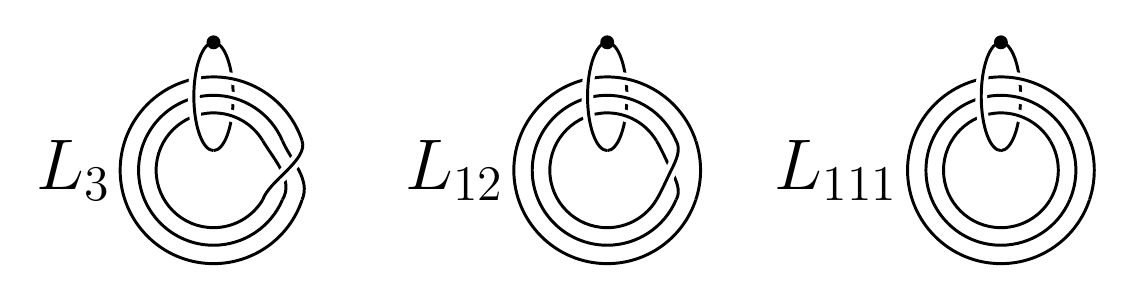}
\caption{Links in $S^1\times B^3$. 
The polyhedra $Y_{3},\ Y_{12},\ Y_{111}$ can be constructed from thes links. }
\label{fig:0-vertex_link}
\end{figure}

Next we suppose $c(\Nbd(S;X))=1$. 
Then $S$ is an $8$-shaped graph, that is, the wedge sum $S^1\vee S^1$ of two circles. 
In this case, there are 11 possible topological types $X_1,\ldots,X_{11}$ of $\Nbd(S_1;X)$, 
which are explained as follows. 
\begin{figure}[tbp]
\includegraphics[width=115mm]{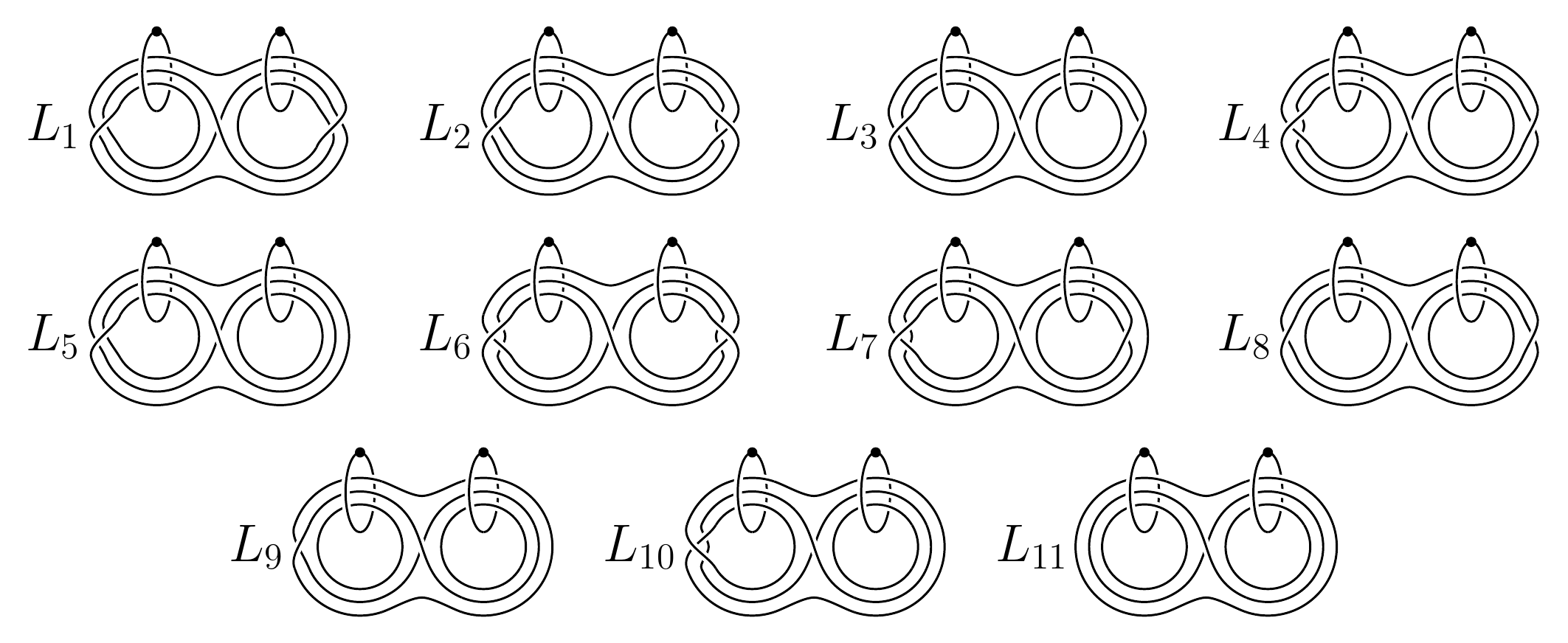}
\caption{Links in $(S^1\times B^3)\natural(S^1\times B^3)$. 
The polyhedra $X_{1},\ldots,X_{11}$ can be constructed from these links. }
\label{fig:1-vertex_link}
\end{figure}
Let $\pi$ be a natural projection from $(S^1\times B^3)\natural(S^1\times B^3)$ to $S^1\vee S^1$, and 
let $L_{i}$ be the link in $(S^1\times B^3)\natural(S^1\times B^3)$ given in Figure~\ref{fig:1-vertex_link} for $i\in\{1,\ldots,11\}$. 
Then $X_{i}$ is the mapping cylinder of $\pi$ restricted to $L_{i}$. 


Note that the over/under information of the links in Figures~\ref{fig:0-vertex_link} and \ref{fig:1-vertex_link} 
does not matter for defining the polyhedra $Y_{3},\ Y_{12},\ Y_{111},\ X_1,\ldots,X_{11}$ 
since they are links in $S^1\times B^3$ or $(S^1\times B^3)\natural(S^1\times B^3)$. 

\subsection{Encoding graphs}
Here we explain a presentation of a simple polyhedron $X$ 
using a graph consisting of some special vertices. 
This notion was introduced by Martelli in \cite{Mar11} 
for the case $c(X)=0$ and generalized in \cite{KMN18} to the case where each connected component of $S(X)$ contains at most one true vertex. 

Let $X$ be a simple polyhedron whose boundary consists of circles and suppose that $c(X)\leq1$. 
We first give a decomposition of $X$ into some fundamental portions. 
Each connected component of $S(X)$ is homeomorphic to $S^1$ or $S^1\vee S^1$. 
As reviewed in the previous subsection, 
a connected component of $\Nbd(S(X);X)$ is homeomorphic to one of 
$Y_{3},\ Y_{12},\ Y_{111},\ X_1,\ldots,X_{11}$. 
Then each component of $X\setminus\Int\Nbd(S(X);X)$ 
is a compact surface corresponding to a region of $X$. 
Note that such a surface is possibly non-orientable, 
and hence it is decomposed into some disks, pair of pants and M\"obius bands. 
Thus, we conclude that $X$ is decomposed (along circles contained in regions) 
into some copies of $D,\ P,\ Y_2,\ Y_{3},\ Y_{12},\ Y_{111},\ X_1,\ldots,X_{11}$, 
where $D$ is a $2$-disk, $P$ is a pair of pants, and $Y_2$ is a M\"obius band. 

The above decomposition induces an {\it encoding graph $G$ of $X$} that has one vertex 
for each portion $D,\ P,\ Y_2,\ Y_{3},\ Y_{12},\ Y_{111},\ X_1,\ldots,X_{11}$ 
or boundary component of $X$. 
Two vertices are connected by an edge if the corresponding portions in $X$ are adjacent. 
Hence each edge $e$ of $G$ corresponds to a simple closed curve 
contained in a region of $X$, which is determined up to isotopy. 
This simple closed curve is called a {\it lift} of $e$. 
\begin{figure}[tbp]
\includegraphics[width=100mm]{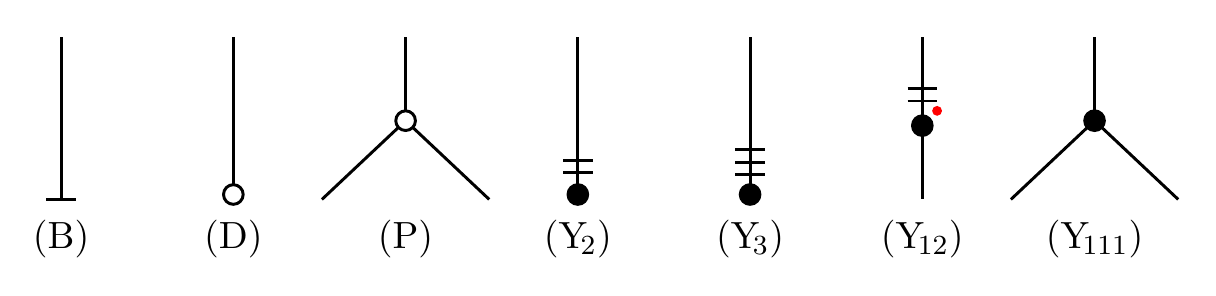}
\caption{The vertices of types 
$(\mathrm B),\ (\mathrm D),\ (\mathrm P),\ (\mathrm Y\!{}_2),\ (\mathrm Y\!{}_{3}),\ (\mathrm Y\!{}_{12}),\ (\mathrm Y\!{}_{111})$, 
which correspond to boundary components of $X$ and portions 
$D,\ P,\ Y_2,\ Y_{3},\ Y_{12}$ and $Y_{111}$, respectively. 
}
\label{fig:encoding_graph}
\includegraphics[width=100mm]{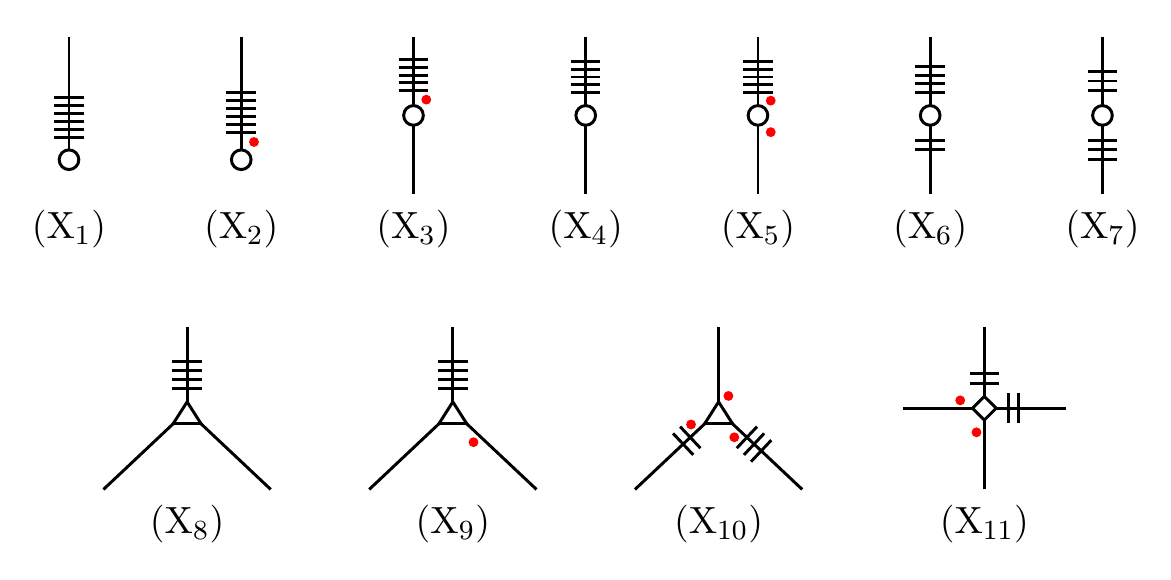}
\caption{The vertices of types $(\mathrm X_1),\ldots, (\mathrm X_{11})$, which correspond to portions $X_1,\ldots,X_{11}$, respectively. }
\label{fig:encoding_graph_Xi}
\end{figure}
The vertices corresponding to $D,\ P,\ Y_2,\ Y_{3},\ Y_{12},\ Y_{111},\ X_1,\ldots,X_{11}$ 
are said to be of types~$\vt{D}$, $\vt{P}$, $\vt{Y\!{}_2}$, $\vt{Y\!{}_{3}}$, $ \vt{Y\!{}_{12}}$, $ \vt{Y\!{}_{111}}$, $\vt{X_1},\ldots,\vt{X_{11}}$, respectively. 
The notations of them are shown in 
Figure~\ref{fig:encoding_graph} and \ref{fig:encoding_graph_Xi}, 
where the vertex of type~$\vt{B}$ indicates a boundary component of $X$. 
See also \cite{KMN18,Mar11}. 

If an encoding graph $G$ is a tree, 
the polyhedron $X$ is uniquely reconstructed from $G$ up to homeomorphism. 
In such a case, we say that $G$ {\it encodes} $X$.
Hence, in this case, the fundamental group of $X$ can be computed from $G$ by using van Kampen's theorem. 
\begin{table}[t]
\begin{tabular}{c||cccc}
\hline
\begin{minipage}[c]{13mm}\centering portion\end{minipage} & 
\begin{minipage}[c]{28mm}\centering graph\end{minipage} & 
\begin{minipage}[c]{22mm}\centering $\pi_1$\end{minipage} & 
\begin{minipage}[c]{40mm}\centering \strut boundary\\ \strut classes in $\pi_1$\end{minipage} \\
\hline
$D$ & 
\begin{minipage}[c]{25mm}
\includegraphics[width=25mm]{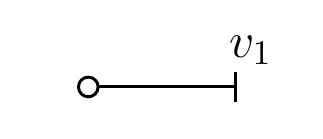}
\end{minipage}& 
$\{1\}$& 
$\gamma_1=1$ \\
\hline
$P$ &
\begin{minipage}[c]{25mm}
\includegraphics[width=25mm]{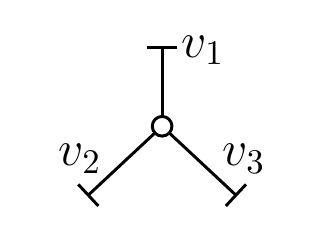}
\end{minipage}& 
$\langle x,y,z\mid xyz\rangle$& 
$\begin{array}{l}
\gamma_1=x,\\ \gamma_2=y,\\ \gamma_3=z
\end{array}$\\
\hline
$Y_2$ &
\begin{minipage}[c]{25mm}
\includegraphics[width=25mm]{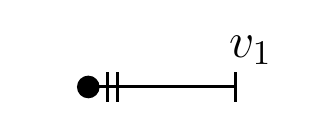}
\end{minipage}& 
$\langle x\rangle$& 
$\gamma_1=x^2$ \\
\hline
$Y_3$ &
\begin{minipage}[c]{25mm}
\includegraphics[width=25mm]{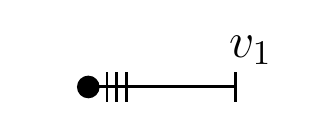}
\end{minipage}& 
$\langle x\rangle$&
$\gamma_1=x^3$ \\
\hline
$Y_{12}$ &
\begin{minipage}[c]{25mm}
\includegraphics[width=25mm]{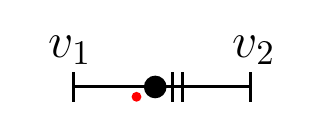}
\end{minipage}& 
$\langle x\rangle$&
$\begin{array}{l}
\gamma_1=x,\\ \gamma_2=x^2
\end{array}$ \\
\hline
$Y_{111}$ & 
\begin{minipage}[c]{25mm}
\includegraphics[width=25mm]{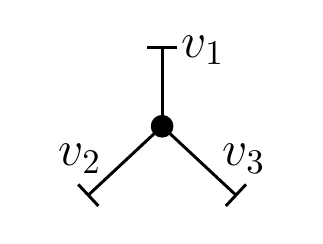}
\end{minipage}& 
$\langle x\rangle$&
$\gamma_1=\gamma_2=\gamma_3=x$\\
\hline
\end{tabular}
\caption{Encoding graphs of $D,\ P,\ Y_2,\ Y_{3},\ Y_{12},\ Y_{111}$, their fundamental groups, and the homotopy classes of the boundary components.}
\label{table:vertex}
\end{table}
\begin{table}[t]
\begin{tabular}{c||cccl}
\hline
\begin{minipage}[c]{13mm}\centering portion\end{minipage} & 
\begin{minipage}[c]{28mm}\centering graph\end{minipage} & 
\begin{minipage}[c]{22mm}\centering $\pi_1$\end{minipage} & 
\begin{minipage}[c]{40mm}\centering \strut boundary\\ \strut classes in $\pi_1$\end{minipage} \\
\hline
$X_1$ & 
\begin{minipage}[c]{25mm}
\includegraphics[width=25mm]{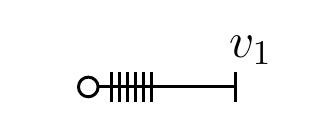}
\end{minipage}& 
$\langle x,y\rangle$&
$\gamma_1=xyx^{-2}y^{-2}$ \\
\hline
$X_2$ &
\begin{minipage}[c]{25mm}
\includegraphics[width=25mm]{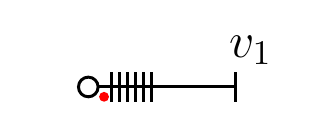}
\end{minipage}&
$\langle x,y\rangle$&
$\gamma_1=xyx^{2}y^{-2}$\\
\hline
$X_3$ &
\begin{minipage}[c]{25mm}
\includegraphics[width=25mm]{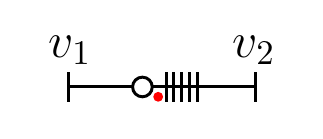}
\end{minipage}&
$\langle x,y\rangle$&
$\begin{array}{l}
\gamma_1=y,\\ 
\gamma_2=xyx^{-2}y^{-1}
\end{array}$\\
\hline
$X_4$ &
\begin{minipage}[c]{25mm}
\includegraphics[width=25mm]{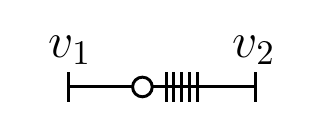}
\end{minipage}&
$\langle x,y\rangle$&
$\begin{array}{l}
\gamma_1=y,\\ 
\gamma_2=xyx^{-2}y
\end{array}$\\
\hline
$X_5$ &
\begin{minipage}[c]{25mm}
\includegraphics[width=25mm]{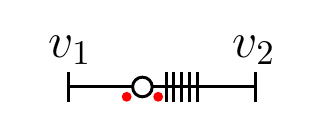}
\end{minipage}&
$\langle x,y\rangle$&
$\begin{array}{l}
\gamma_1=y,\\ 
\gamma_2=xyx^2y^{-1}
\end{array}$\\
\hline
$X_6$ &
\begin{minipage}[c]{25mm}
\includegraphics[width=25mm]{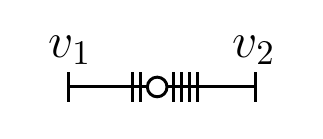}
\end{minipage}&
$\langle x,y\rangle$&
$\begin{array}{l}
\gamma_1=xy,\\ 
\gamma_2=x^2y^{-2}
\end{array}$\\
\hline
$X_7$ &
\begin{minipage}[c]{25mm}
\includegraphics[width=25mm]{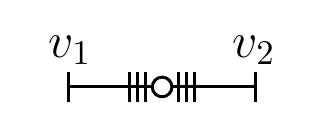}
\end{minipage}&
$\langle x,y\rangle$&
$\begin{array}{l}
\gamma_1=xy^2,\\ 
\gamma_2=x^2y^{-1}
\end{array}$\\
\hline
$X_8$ &
\begin{minipage}[c]{25mm}
\includegraphics[width=25mm]{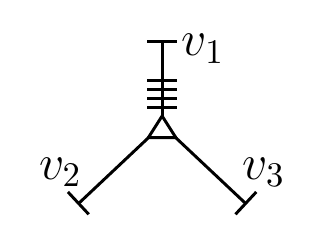}
\end{minipage}&
$\langle x,y\rangle$&
$\begin{array}{l}
\gamma_1=xyx^{-1}y^{-1},\\ 
\gamma_2=x,\\
\gamma_3=y
\end{array}$\\
\hline
$X_9$ &
\begin{minipage}[c]{25mm}
\includegraphics[width=25mm]{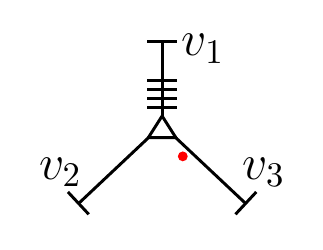}
\end{minipage}&
$\langle x,y\rangle$&
$\begin{array}{l}
\gamma_1=xyxy^{-1},\\ 
\gamma_2=x,\\
\gamma_3=y
\end{array}$\\
\hline
$X_{10}$ &
\begin{minipage}[c]{25mm}
\includegraphics[width=25mm]{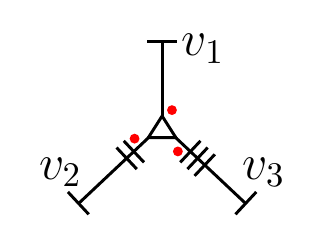}
\end{minipage}&
$\langle x,y\rangle$&
$\begin{array}{l}
\gamma_1=y,\\ 
\gamma_2=xy,\\
\gamma_3=x^2y^{-1}
\end{array}$\\
\hline
$X_{11}$ &
\begin{minipage}[c]{25mm}
\includegraphics[width=25mm]{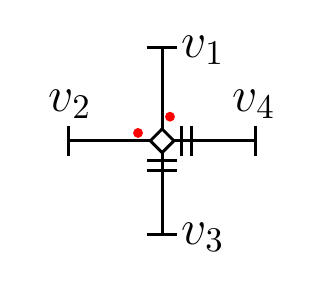}
\end{minipage}&
$\langle x,y\rangle$&
$\begin{array}{l}
\gamma_1=x,\\ 
\gamma_2=y,\\
\gamma_3=xy,\\ 
\gamma_4=xy^{-1}
\end{array}$\\
\hline
\end{tabular}
\caption{Encoding graphs of $X_1,\ldots,X_{11}$, their fundamental groups, and the homotopy classes of the boundary components.}
\label{table:vertex_Xi}
\end{table}
The necessary information is summarized 
in Tables~\ref{table:vertex} and \ref{table:vertex_Xi}, 
which exhibit encoding graphs of the portions 
$D$, $P$, $Y_2$, $Y_{3}$, $Y_{12}$, $Y_{111}$, $X_1,\ldots,X_{11}$, 
presentations of their fundamental groups 
and the homotopy classes of their boundaries. 
Here each vertex of type~$\vt{B}$ is denoted by $v_i$ for some $i\in\{1,\ldots,4\}$, 
and $\gamma_i$ is the corresponding component of the boundary of a portion. 

Each boundary component $\gamma$ of $Y_{3}$, $Y_{12}$, $Y_{111}$, $X_1,\ldots,X_{11}$ 
is represented by a word in $\langle x\rangle$ or $\langle x,y\rangle$ 
as in Tables~\ref{table:vertex} and \ref{table:vertex_Xi}, 
and its length coincides with the number of triple lines along which $\gamma$ goes, counted with multiplicity. 
This number is called the {\it length} of $\gamma$ as well. 

Several edges of $G$ are decorated with some dashes and red dots near the vertices, see Figure~\ref{fig:encoding_graph} and \ref{fig:encoding_graph_Xi}. 
The number of dashes indicates the length of the corresponding boundary, 
and a red dot indicates the parity of the corresponding region of $X$. 
Note that the length for a M\"obius band $Y_2$ has not defined, 
but the incident edge of a vertex of type $\vt{{Y}\!{}_{12}}$ 
is also decorated with two dashes for consistency with the other notations. 
We also note that a red dot of a vertex of type $\vt{Y\!{}_{12}}$ is sometimes omitted 
if no confusion arises. 


Notice that an encoding graph is not uniquely determined from $X$. 
Two moves on encoding graphs are shown in Figure~\ref{fig:YV-IH}: 
the left is a {\it YV-move} and the right is an {\it IH-move}. 
These moves correspond to giving another decomposition of a region, so they do not change the homeomorphism types of the corresponding polyhedra. 
\begin{figure}[tbp]
\includegraphics[width=90mm]{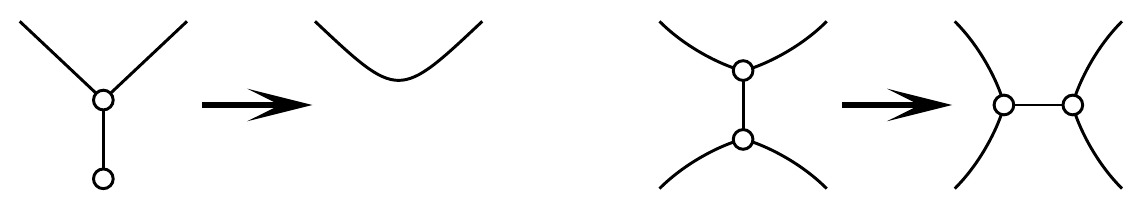}
\caption{YV-move and IH-move.}
\label{fig:YV-IH}
\end{figure}
Suppose that $G$ is a tree and let $G'\subsetneq G$ be a subgraph. 
Let $N(G')$ denote the neighborhood of $G'$, that is, 
$N(G')\subset G$ is obtained from $G'$ by adding all the vertices 
adjacent to vertices of $G'$ and all the edges between them. 
Then we replace each vertex in $N(G')\setminus G'$ with a vertex of type~$\vt{D}$, 
and the obtained graph is called the {\it $\vt{D}$-closure} of $G'$, denoted by $\widehat{G'}$. 
See Figure~\ref{fig:(D)-closure} for an example. 
\begin{figure}[tbp]
\includegraphics[width=70mm]{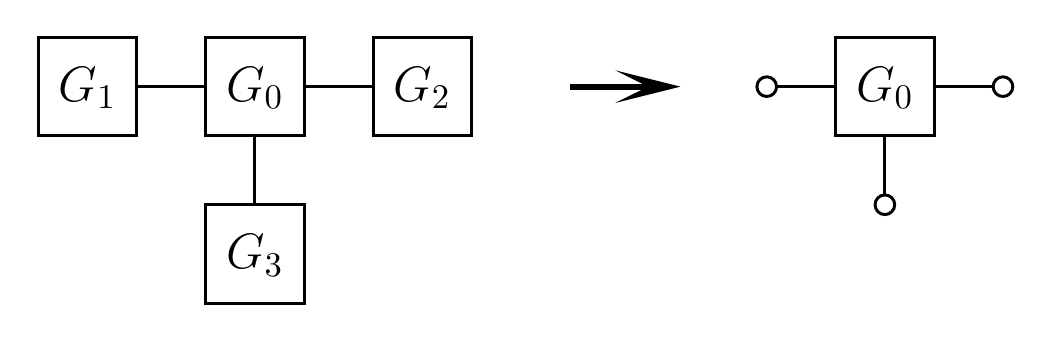}
\caption{The $\vt{D}$-closure of a subgraph $G_0$. }
\label{fig:(D)-closure}
\end{figure}
The left of the figure shows an encoding graph $G$ and subgraphs $G_0,\ldots,G_4$ of $G$, 
and the right one shows the $\vt{D}$-closure $\widehat{G_0}$ of $G_0$. 


\section{Shadows of $2$-knots and knot groups}
\label{sec:Shadows of 2-knots and knot groups}
\subsection{Shadows of $2$-knots}
A smoothly embedded surface $K$ in a $4$-manifold $W$ is called a {\it knotted surface}. 
If $K$ and $W$ are diffeomorphic to $S^2$ and $S^4$, respectively, 
$K$ is especially called a {\it $2$-knot}. 
A $2$-knot is said to be \textit{unknotted} (or \textit{trivial}) 
if it bounds a smooth $3$-ball in $S^4$. 

We now define a shadow of a $2$-knot as follows. 

\begin{definition}
\label{def:shadow_2-knot}
Let $K$ be a $2$-knot. 
A shadow $X$ of $S^4$ is called a {\it shadow} of $K$ 
if $K$ is embedded in $X$. 
\end{definition}
We can define a shadow also for a knotted surface in a similar manner, 
but it is not our focus in this paper. 

Note that an unknotted $2$-knot is a shadow of itself with gleam $0$. 
In general, a $2$-knot is unknotted if and only if it 
admits a shadow without true vertices, 
which will be shown in Theorem~\ref{thm:complexity0}. 

\begin{theorem}
\label{thm:2-knot_shadow}
Every $2$-knot admits a shadow. 
\end{theorem}
By considering the handle decomposition relative to $\Nbd(K;S^4)$, 
we can prove the above from Theorem~\ref{thm:Tuarev}. 
In Section~\ref{sec:Banded unlink diagrams}, 
we will give a recipe for making a shadow of a $2$-knot from a \textit{banded unlink diagram}, 
which gives an alternative proof of Theorem~\ref{thm:2-knot_shadow}. 

Then we define complexities for $2$-knots as well as for $4$-manifolds. 
\begin{definition}
For a $2$-knot $K$, 
the {\it shadow-complexity} $\shco(K)$ and the {\it special shadow-complexity} $\spshco(K)$ of $K$ 
are defined as the minimum number of true vertices of all shadows of $K$ and 
that of all special shadows of $K$, respectively. 
\end{definition}
\subsection{Knot groups}
Let $K$ be a $2$-knot. 
The {\it knot group} $G(K)$ of $K$ is the fundamental group of 
the complement of $K$. 
We will give a presentation of $G(K)$ in Proposition~\ref{prop:knot_group}. 
To state this proposition, we first give some notations. 

Let $X$ be a shadow of $K$ 
and $X_K$ be a regular neighborhood $\Nbd(K;X)$ of $K$ in $X$. 
Choose a regular neighborhood $\Nbd(K;S^4)$ so that 
$X_K$ is proper in $\Nbd(K;S^4)$, 
and let it be denoted by $M_K$. 
Set
\begin{align*}
&X'=X\setminus \Int X_K , \\
&T=\partial X'\cap\partial X_K\subset\partial X' \text{, and}\\
&M_{X'}=\Nbd(X'; S^4\setminus \Int M_{K}). 
\end{align*}
We assume that $X'$ and $T$ are connected for simplicity. 
Note that this can always be assumed 
by applying some $(0\to2)$-moves (c.f. \cite{Cos05,Tur94}) in advance, 
and see also Remark~\ref{rmk:knot_group_general}. 
The gluing map $\partial M_{X'}\cap M_K\to \partial M_{K}$ will be written as $f$. 
Note that $T$ is a graph and the valency of each vertex of $T$ is 3. 
By the definition of shadows of $2$-knots, 
the knot complement $S^4\setminus K$ admits a decomposition
\[
(M_{K}\setminus K)\cup_f M_{X'}\cup (\text{3- and 4-handles}). 
\]
We can easily see that $M_{K}\setminus K\cong (S^2\times D^2)\setminus (S^2\times\{0\})$ and 
it retracts onto $\partial M_{K}$ (${}\cong S^2\times S^1$). 
We also see that $M_{X'}$ retracts onto $X'$ with keeping $T$. 
Thus, the knot group $G(K)$ can be computed from
$\partial M_{K}\cup_{f|_T} X'$. 

Choose a base point $t\in T$ and 
a presentation 
\[
\pi_1(X',t)=\langle S \mid R\rangle. 
\]
Since $T$ is a graph, its fundamental group is freely generated by some loops 
\[
w_1,\ldots,w_m \in \pi_1(X',t). 
\]
We give an orientation to $K$ arbitrarily. 
Then the fundamental group of $\partial M_{K}$ has a presentation 
\[
\pi_1(\partial M_{K},t)=\langle \mu \rangle, 
\]
where $\mu$ is the meridian of $K$ whose orientation agrees with those 
of $K$ and $S^4$. 

For each $i\in\{1,\ldots,m\}$, 
there is a $2$-chain $D_i=\sum a_{i,j}R_j$ in $K$ 
with $\partial D_i=[w_i]$ in $H_1(\Cl(X_K\setminus K))\cong H_1(S(X)\cap K)$, 
where $R_j$ is a region contained in $K$ with an orientation induced from that of $K$. 
Set 
\begin{align}\label{align:gleam_of_chain}
\mathrm{gl}(w_i)=\sum a_{i,j}\gl(R_j). 
\end{align} 
This number is equal to the algebraic intersection number of $\tilde{D}_i$ and $K$ in $M_{K}$, 
where $\tilde{D}_i$ is a smooth oriented surface bounded by $f(w_i)$ in $M_{K}$. 
Hence we have $f(w_i)=\mu^{\mathrm{gl}(w_i)}$ in $\pi_1(\partial M_{K},\ast)$. 
Therefore, by van Kampen's theorem, we obtain the following. 
\begin{proposition}
\label{prop:knot_group}
Under the above settings, the following holds:
\[
G(K)\cong 
\langle S,\mu \mid R,w_1\mu^{-\mathrm{gl}(w_1)},\ldots,w_m\mu^{-\mathrm{gl}(w_m)}\rangle. 
\]
\end{proposition}
\begin{remark}
\label{rmk:knot_group_general}
The assumption that $X'$ and $T$ are connected is not essential. 
The case where $X'$ and $T$ are not connected is as follows. 
Let $X'_1,\ldots,X'_k$ and $T_1,\ldots,T_k$ denote the connected components of $X'$ and $T$, respectively. 
Then the knot group $G(K)$ is presented as
\[
G(K)\cong\langle S_1,\ldots,S_k,\mu \mid 
R_1,\ldots,R_k, w_1\mu^{-\mathrm{gl}(w_1)},\ldots,w_{m'}\mu^{-\mathrm{gl}(w_{m'})}\rangle,  
\]
where $\pi_1(X'_j)\cong\langle S_j\mid R_j\rangle$ for $j\in\{1,\ldots,k\}$ and 
$w_1,\ldots,w_{m'}$ are loops in $\pi_1(X'_1)\ast\cdots\ast\pi_1(X'_k)$ 
such that they generates $\pi_1(T_1)\ast\cdots\ast\pi_1(T_k)$. 
\end{remark}
\section{Banded unlink diagrams}
\label{sec:Banded unlink diagrams}
In this section, we give a review of a description called a {\it banded unlink diagram} for a knotted surface. 
See \cite{HKM20} for the details. 
We start with the definition of banded links in a $3$-manifold. 
\subsection{Banded links}
\label{subsec:Banded links}
For a link $L$ in a $3$-manifold $N$, 
the image $b=\mathrm{Im}\beta$ of an embedding $\beta:[0,1]\times[0,1]\to N$ 
with $L\cap b=\beta(\{0,1\}\times[0,1])$ is called a {\it band}. 
The {\it core} of $b$ is defined as $\beta([0,1]\times\{1/2\})$. 
The pair $(L,\bm{b})$ of $L$ and mutually disjoint bands 
$\bm{b}=b_1\cup\cdots\cup b_n$ is called a {\it banded link}. 
The {\it negative resolution} and the {\it positive resolution} of 
$(L,\bm{b})$, respectively, are defined as the links $L$ and $L_{\bm{b}}$, where 
\[
L_{\bm{b}}=\left(L \setminus \left(\bigcup_{i=1}^n \beta_i\big(\{0,1\}\times[0,1]\big)\right)\right)\cup
\left(\bigcup_{i=1}^n \beta_i\big([0,1]\times\{0,1\}\big)\right). 
\]

\subsection{Banded unlink diagrams}
\label{subsec:Banded unlink diagrams}
Let $W$ be a closed $4$-manifold, and 
fix a handle decomposition of $W$ having a single $0$-handle. 
For $i\in\{0,\ldots,4\}$, let $W_i$ denote the handlebody consisting of all the handles with indices at most $i$. 
Clearly, $W_0\cong B^4$ and $W_4=W$. 
Suppose that this handle decomposition is described by a Kirby diagram $\K=\mathcal{L}_1\sqcup \mathcal{L}_2\subset S^3$, 
where $\mathcal{L}_1$ is a dotted unlink indicating the $1$-handles and $\mathcal{L}_2$ is a framed link indicating the $2$-handles. 
Note that $S^3$ in which $\K$ is drawn is considered as the boundary $\partial W_0$ of the $0$-handle, 
and we can identify the complement $S^3\setminus \nu \K$ of a tubular neighborhood $\nu \K$ of $\K$ with a subset in $\partial W_2$. 

Let $(L,\bm{b})$ be a banded link in $S^3\setminus\nu \K$. 
Note that $(L,\bm{b})$ is considered as a banded link in $\partial W_2$ and also in $\partial W_1$. 
Suppose that the negative resolution $L$ and the positive resolution $L_{\bm{b}}$ 
are unlinks in $\partial W_1$ and $\partial W_2$, respectively. 
Then we call the triple $(\K,L,\bm{b})$ a {\it banded unlink diagram} in $K$. 

A banded unlink diagram $(\K,L,\bm{b})$ can be interpreted as a presentation 
of a knotted surface in the $4$-manifold $W$ in the following way. 
By the definition of a banded unlink diagram, 
there exist collections $\mathcal{D}_1\subset \partial W_1$ and $\mathcal{D}_2\subset \partial W_2$ of $2$-disks 
bounded by $L$ and $L_{\bm{b}}$, respectively. 
We push the interiors of $\mathcal{D}_1$ and $\mathcal{D}_2$ into $\Int W_1$ and $W\setminus W_2$, respectively, with keeping the boundaries. 
Then set $K=\mathcal{D}_1\cup\bm{b}\cup\mathcal{D}_2$. 
It forms a knotted surface in $W$, and $(\K,L,\bm{b})$ is also said to be a {\it banded unlink diagram} for $K$. 

If a Kirby diagram for $S^4$ consists of no dotted circles and no framed knots, 
we will say that such a diagram is the {\it trivial Kirby diagram}. 
By Kawauchi, Shibuya and Suzuki in \cite{KSS82} and also by Lomanaco \cite{Lom81}, 
it was shown that any $2$-knot admits a banded unlink diagram in the trivial Kirby diagram. 
In the general case, 
Hughes, Kim and Miller showed in \cite{HKM20} that any knotted surface in any closed $4$-manifold admits a banded unlink diagram. 


\begin{figure}[tbp]
\centering
\includegraphics[width=1\hsize]{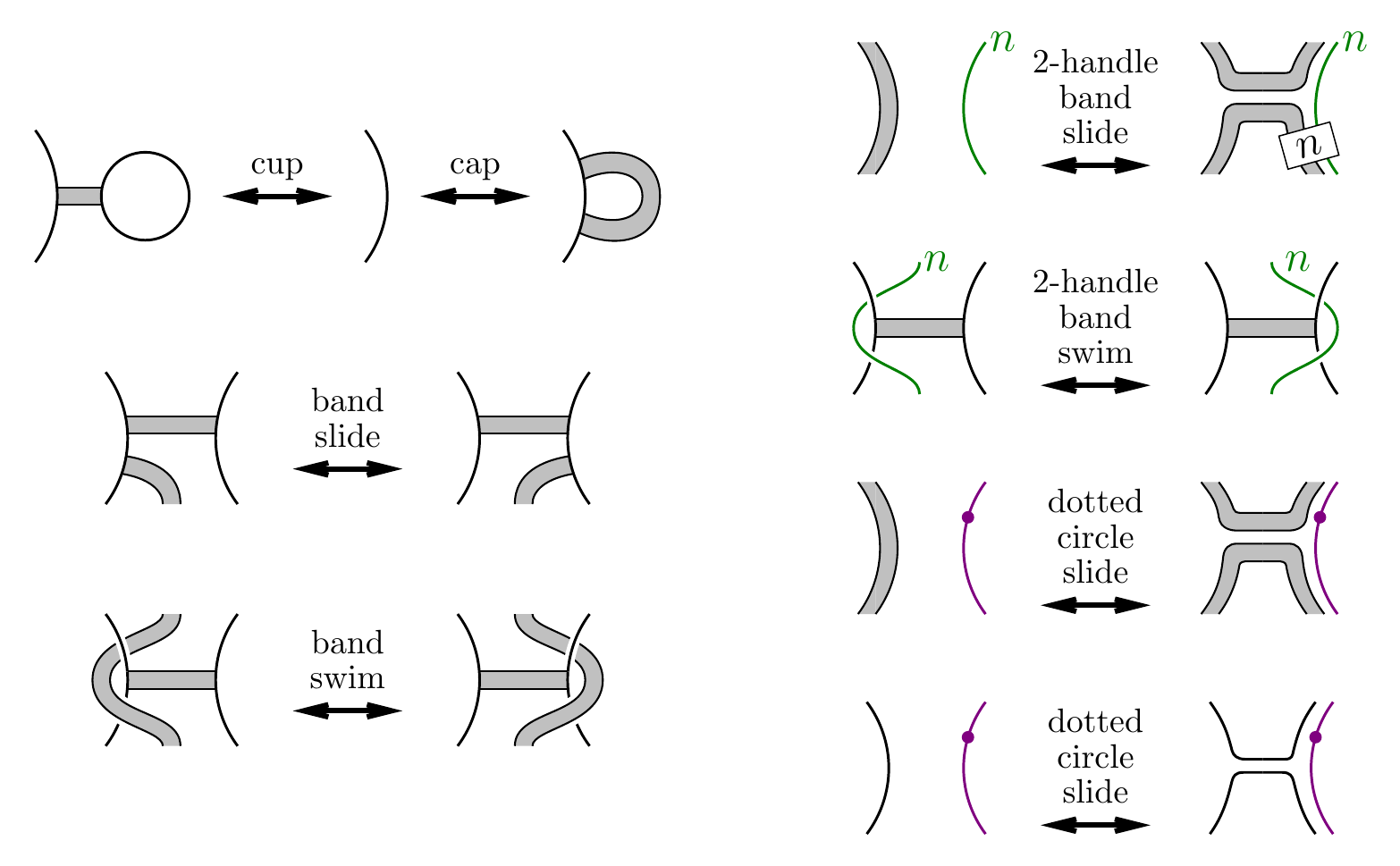}
\caption{Band moves of banded unlink diagrams. }
\label{fig:band_moves}
\end{figure} 
See Figure~\ref{fig:band_moves}. 
The three moves shown in the left of the figure (in the trivial Kirby diagram) 
were introduced by Yoshikawa in \cite{Yos94}, 
and it was shown that these moves are sufficient to relate any two banded unlink diagrams describing the same knotted surface 
by Swenton \cite{Swe01} and also by Kearton and Kurlin\cite{KK08}. 
The other moves in Figure~\ref{fig:band_moves} were introduced by Hughes, Kim and Miller \cite{HKM20} for the general case. 
The seven kinds of moves exhibited in the figure are called {\it band moves}. 
\begin{theorem}
[Hughes, Kim and Miller \cite{HKM20}]
Two banded unlink diagrams $(\K,L,\bm{b})$ and $(\K,L',\bm{b}')$ representing the same knotted surface are related by a finite sequence of band moves. 
\end{theorem}

\subsection{Shadows from banded unlink diagrams}
We again focus on the case of $2$-knots. 
In this subsection, we give a construction of a shadow of a $2$-knot from a banded unlink diagram. 

Let $K$ be a $2$-knot and $(\K,L,\bm{b})$ a banded unlink diagram for $K$, 
where $L=L_1\sqcup\cdots\sqcup L_m$ and $\bm{b}=b_1\sqcup\cdots\sqcup b_n$. 
For simplicity, we suppose $\K$ is the trivial Kirby diagram, 
and we will write $(L,\bm{b})$ instead of $(\K,L,\bm{b})$. 
In this case, the ambient $4$-manifold $S^4$ is decomposed into one $0$-handle $W_0$ and one $4$-handle, 
and $(L,\bm{b})$ is assumed to be in $\partial W_0$. 
As explained in Subsection~\ref{subsec:Banded unlink diagrams}, 
the $2$-knot $K$ lies in $S^4$ so that $K\cap\partial W_0=L\cup\bm{b}$, 
and we recall the notations 
$\mathcal{D}_1=(K\cap \Int W_0) \cup L$ and 
$\mathcal{D}_2=(K\cap S^4\setminus W_0) \cup L$. 
%
%

\begin{step}
Let $\Gamma$ be the union of $L$ and the cores of $b_1,\ldots,b_n$, 
which is a trivalent graph in $S^3=\partial B^4$. 
Let $\pi$ be a regular projection from $\Gamma$ to a $2$-disk $D_0$ 
such that the images of $L_1,\ldots,L_m$ bound mutually disjoint 
$2$-disks $D_1,\ldots,D_m\subset D_0$, respectively. 
Then consider (abstractly) the mapping cylinder 
\[
\tilde X_{(L,\bm{b},\pr)}=\big(D_0\cup (\Gamma\times[0,1])\big)/\!\sim 
\]
of $\pi$, where $\sim$ is defined by $\pi(x)\sim (x,0)$ for $x\in\Gamma$. 
Since $\pi$ is chosen so that the diagram is regular, $\tilde X_{(L,\bm{b},\pr)}$ is a simple polyhedron. 
This polyhedron can be embedded in the $4$-ball $W_0$ as a shadow 
since $\tilde X_{(L,\bm{b},\pr)}$ can collapses onto the disk $D_0$. 
Actually, there is a natural gleam on $\tilde X_{(L,\bm{b},\pr)}$ determined from the diagram of $\Gamma$ 
such that it corresponds to a $4$-ball in which $\tilde X_{(L,\bm{b},\pr)}$ is embedded as a shadow. 
See Remark~\ref{rmk:BL_gleam}. 
Then $\tilde X_{(L,\bm{b},\pr)}$ will be considered as a shadow of $W_0$, and we can naturally identify 
\begin{itemize}
 \item
$\partial \tilde X_{(L,\bm{b},\pr)}\setminus\partial D_0$ 
with $\Gamma\subset \partial W_0$, and 
\item
$\displaystyle \bigcup_{i=1}^m \big( D_i\cup\big(L_i\times[0,1]\big)\big)$ 
with $\mathcal{D}_1\subset W_0$. 
\end{itemize}
See Figure~\ref{fig:banded_link_example} for an example. 
\begin{figure}[tbp]
\centering
\includegraphics[width=110mm]{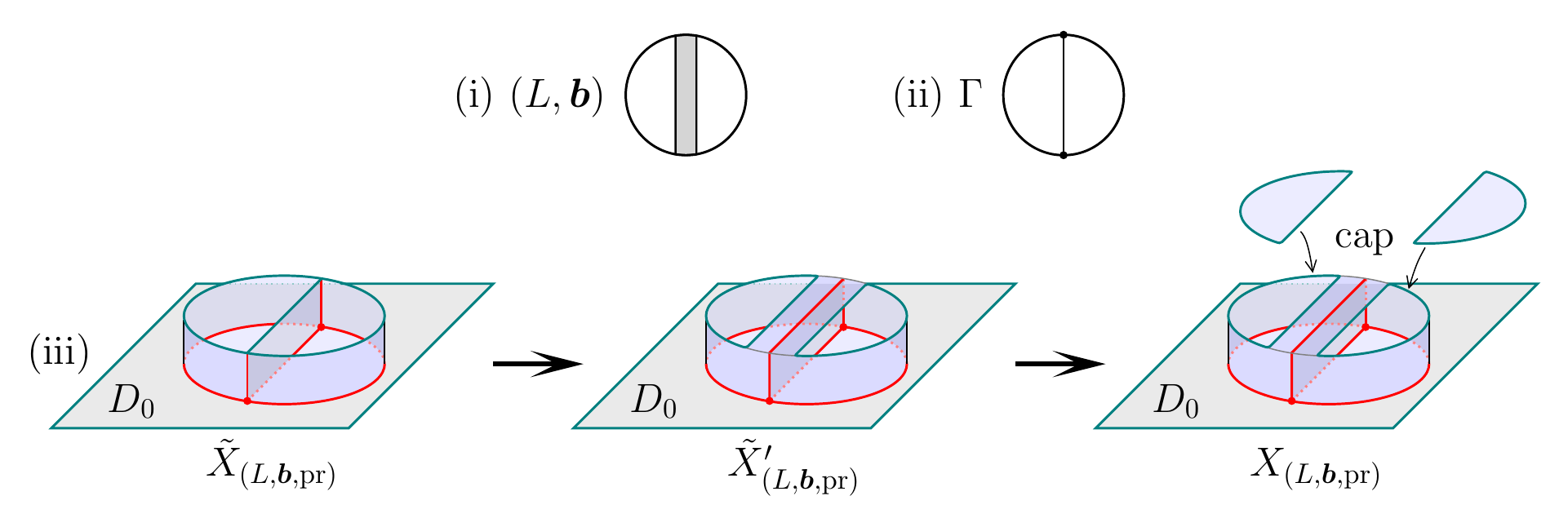}
\caption{An example of a banded unlink diagram $(L,\bm{b})$, $\Gamma$ and 
polyhedra $\tilde X_{(L,\bm{b},\pr)}$, $\tilde X'_{(L,\bm{b},\pr)}$ and $X_{(L,\bm{b},\pr)}$. }
\label{fig:banded_link_example}
\end{figure} 
The banded link $(L,\bm{b})$ shown in Figure~\mbox{\ref{fig:banded_link_example}-(i)} consists of one unknot $L$ and one band $\bm{b}=b$, 
and it is a presentation of the trivial $2$-knot. 
Figure~\mbox{\ref{fig:banded_link_example}-(ii)} shows a diagram of the graph $\Gamma=L\cup(\text{core of }b)$. 
Then $\tilde X_{(L,\bm{b},\pr)}$ is a polyhedron as shown in the leftmost of Figure~\mbox{\ref{fig:banded_link_example}-(iii)}. 
\end{step}
\begin{step}
The graph $\Gamma$ lies in $\partial W_0$ as the boundary of $\tilde X_{(L,\bm{b},\pr)}$, 
and the whole banded link $(L,\bm{b})$ is also embedded in $\partial W_0$. 
Set 
\[
\tilde X'_{(L,\bm{b},\pr)} = \tilde X_{(L,\bm{b},\pr)}\cup\bm{b}
\]
and then push a neighborhood $\Nbd\left(\bm{b};\tilde X'_{(L,\bm{b},\pr)}\right)$ 
inside $W_0$ so that $\tilde X'_{(L,\bm{b},\pr)}$ is properly embedded in $W_0$. 
See the center of Figure~\mbox{\ref{fig:banded_link_example}-(iii)}. 
Note that $\tilde X'_{(L,\bm{b},\pr)}$ is also a shadow of $W_0$. 
\end{step}

\begin{step}
The boundary $\partial \tilde X'_{(L,\bm{b},\pr)}\subset \partial W_0$
is the positive resolution $L_{\bm{b}}$ of $(L,\bm{b})$, 
which is the $m'$-component unlink, where $m'=2+n-m$.  
We attach $m'$ $2$-handles to $W_0$ along $L_{\bm{b}}$ with $0$-framing, 
and we set 
\[
X_{(L,\bm{b},\pr)}=\tilde X'_{(L,\bm{b},\pr)}\cup D'_1\cup\cdots\cup D'_{m'}, 
\]
where $D'_1,\ldots,D'_{m'}$ are the core disks of the $2$-handles. 
See the right of Figure~\mbox{\ref{fig:banded_link_example}-(iii)}. 
These disks $D'_1,\ldots,D'_{m'}$ correspond to $\mathcal{D}_2$. 
The $4$-manifold $W_0$ with the $2$-handles attached 
is diffeomorphic to $\natural_{m'}(S^2\times D^2)$, 
and $X_{(L,\bm{b},\pr)}$ is its shadow. 
We can obtain $S^4$ from this $4$-manifold 
by attaching $m'$ $3$-handles and a $4$-handle in a canonical way \cite{LP72}. 
Hence $X_{(L,\bm{b},\pr)}$ is a shadow of $S^4$, and 
the $2$-knot $K$ is realized in $S^4$ as 
\[
\bigg(\bigcup_{i=1}^m \big(D_i\cup\big(L_i\times[0,1]\big)\big)\bigg)
\cup\bm{b}\cup 
 \bigg(\bigcup_{j=1}^{m'} D'_j\bigg) 
\]
in $X_{(L,\bm{b},\pr)}$. Thus we have the following. 
\end{step}
\begin{proposition}
The polyhedron $X_{(L,\bm{b},\pr)}$ is a shadow of $K$.  
\end{proposition}
\begin{remark}
\label{rmk:BL_gleam}
The gleams of regions of $X_{(L,\bm{b},\pr)}$ can be easily calculated.  
\begin{figure}[tbp]
\centering
\includegraphics[height=24mm]{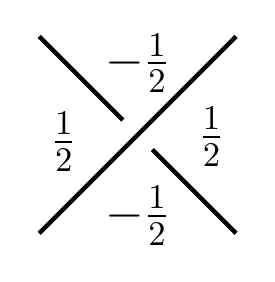}
\caption{The local contribution to the gleams. }
\label{fig:contribution}
\end{figure} 
For regions on the disk $D_0$, we can use the rule shown in Figure~\ref{fig:contribution}: 
the gleam of an internal region contained in $D_0$ is given as the sum of 
the local contribution shown in the figure at each crossing of 
the diagram of $\Gamma$ adjacent to the region \cite{Tur94,CT08}. 
The gleam of the region forming $(\text{core of }b_i)\times [0,1]$ is given as 
the number of times $b_i$ twists with respect to $D_0$ on the diagram. 
Each of the remaining regions is a part of $K$ and contains a core disk $D'_j$ of a $2$-handle. 
The gleams of them are the minus of the writhe number of $L'_j$ on $D_0$, where $L'_j$ is the component of $L_{\bm{b}}$ to which $D'_j$ is attached. 
\end{remark}
\begin{remark}
\label{rmk:BL_true_vertices}
All the true vertices of $X_{(L,\bm{b},\pr)}$ lie on $D_0$.  
Each of them derives from a crossing of the diagram of $\Gamma$ 
or a trivalent vertex of $\Gamma$. 
Thus, we can estimate the shadow-complexity of $K$ from the diagram of $\Gamma$. 
Examples will be studied in the next subsection. 
\end{remark}
\begin{remark}
\label{rmk:non-trivial_Kirby_diag}
Even if $\K$ is not trivial, 
we also can construct a shadow of $K$ 
by considering a shadow of $W_2$ instead of the disk $D_0$. 
\end{remark}

\subsection{Examples}
Let $T(2,2n+1)^k$ denote the $k$-twist spun of the classical torus knot $T(2,2n+1)$. 
Figure~\ref{fig:spun_torus} shows a banded unlink diagram $(L,\bm{b})=(L_1\sqcup L_2,b_1\cup b_2)$ 
for $T(2,2n+1)^k$ that was given in \cite{Jab16}. 
\begin{figure}[tbp]
\centering
\begin{tabular}{ccc}
\begin{minipage}[t]{60mm}
\includegraphics[height=60mm]{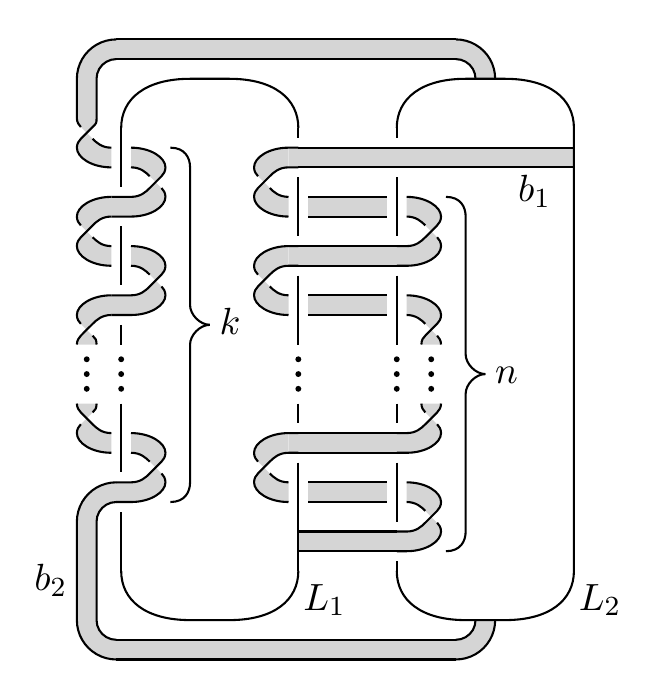}
\caption{A banded unlink diagram of the $k$-twist spun knot of $T(2,2n+1)$. }
\label{fig:spun_torus} 
\end{minipage}
\hspace{5mm}
\begin{minipage}[t]{55mm}
\centering
\includegraphics[height=60mm]{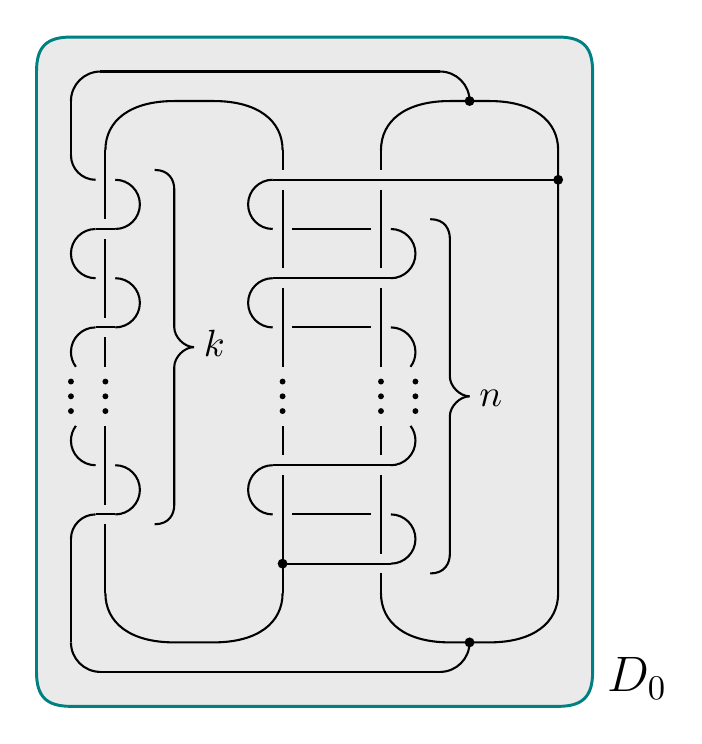}
\caption{A diagram of the graph $\Gamma$ that is the union of $L_1,L_2$ and the cores of $b_1, b_2$.}
\label{fig:spun_torus_diag} 
\end{minipage}
\end{tabular}
\end{figure} 
Considering a natural projection $\pr:\Gamma\to D_0$ 
from the graph $\displaystyle \Gamma=L_1\cup L_2\cup\left(\text{cores of }b_1,b_2\right)$ to a $2$-disk $D_0$, 
we draw a regular diagram of $\Gamma$ as in Figure~\ref{fig:spun_torus_diag}. 
This diagram has $4n+2k$ crossings, and $\Gamma$ has $4$ trivalent vertices. 
Hence $X_{(L,\bm{b},\pr)}$, a shadow of $T(2,2n+1)^k$, has $4n+2k+4$ true vertices. 
The polyhedron $X_{(L,\bm{b},\pr)}$ has a single boundary region, 
which is adjacent to $2k+3$ true vertices. 
These true vertices are eliminated by collapsing, and the resulting polyhedron is still a shadow of $T(2,2n+1)^k$. 
Therefore, we obtain an upper bound of the shadow-complexity 
of the twist spun knot $T(2,2n+1)^k$. 
\begin{proposition}
\label{prop:sc_twist_torus_knot}
$\shco\left(T(2,2n+1)^k\right)\leq 4n+1$. 
\end{proposition}
\begin{remark}
The $1$-twist spun of any $1$-knot is trivial \cite{Zee65}. 
As we will see in Theorem~\ref{thm:complexity0}, 
the shadow-complexity of an unknotted $2$-knot is $0$. 
The $2$-knot $T(2,2n+1)^k$ with $n=1$ and $k=0$ is the spun trefoil, 
which is $K_{-1}$ in our notation, see Figure~\ref{fig:Kn}. 
We will show that its shadow-complexity is $1$ in Theorem~\ref{thm:sc=1}. 
\end{remark}
\section{Modifications and fundamental groups}
\label{sec:Modifications and fundamental groups}
A subspace $Y$ of a simple polyhedron $X$ is called a {\it subpolyhedron} if 
there exist simple closed curves $\gamma_1\sqcup\cdots\sqcup\gamma_n$ in $X\setminus (S(X)\cup\partial X)$ 
such that $Y$ is the closure of a connected component of $X\setminus(\gamma_1\sqcup\cdots\sqcup\gamma_n)$. 
It is obvious that $Y$ itself is a simple polyhedron. 
Each simple closed curve $\gamma_i$ is a boundary component of $Y$, 
and it is called a {\it cut end} of $Y$ in $X$. 
In other words, a boundary component of $Y$ but not of $X$ is called a cut end. 
If $X$ is a shadowed polyhedron, 
$Y$ can also be assigned with gleams canonically: 
an internal region $R$ of $Y$ might be formed by some internal regions of $X$, 
and the sum of their gleams is the gleam of $R$, see \cite{Nao17}. 

Henceforth, we fix the following notations: 
\begin{itemize}
\item
$K$ is a $2$-knot, and 
\item
$X$ is a shadow of $K$.  
\end{itemize}
Note that $X$ is simply-connected since it is a shadow of $S^4$. 

\subsection{Compressing disk addition}
Let $\gamma$ be a simple closed curve contained in a region of $X$. 
Let $\pi:M_X\to X$ be the projection, 
where $M_X$ is the $4$-manifold obtained from $X$ by Turaev's reconstruction. 
Since $X$ is also a shadow of $S^4$, we have $\partial M_X\cong \#_h(S^1\times S^2)$ for some $h\in\Z_{\geq0}$. 
Hence $(\pi|_{\partial M_X})^{-1}(\gamma)$ is an embedded torus in $\#_h(S^1\times S^2)$. 
Every embedded torus in $\#_h(S^1\times S^2)$ has a compressing disk by Dehn's lemma, 
so let $D_\gamma$ be such a disk for the torus $(\pi|_{\partial M_X})^{-1}(\gamma)$. 
We consider a $2$-disk $D'_\gamma$ enlarged from $D_\gamma$ such that 
$D'_\gamma\setminus D_\gamma\subset \pi^{-1}(\gamma)$ and $\partial D'_\gamma\subset X$, 
and then modify the disk $D'_\gamma$ near its boundary 
so that $\partial D'_\gamma$ is a generically immersed curve in $\Nbd(\gamma;X)$ 
by a small perturbation. 
This can be done without creating self-intersections of $\Int D'_\gamma$ . 
We thus obtain a new simple polyhedron $X\cup D'_\gamma$ embedded in $S^4$. 
\begin{proposition}
[{\cite{KMN18}}]
Under the above settings, $X\cup D'_\gamma$ is a shadow of $S^4$ and also of $K$. 
\end{proposition}
The disk $D'_\gamma$ is called a {\it compressing disk} of $\gamma$. 
The addition of $D'_\gamma$ corresponds to attaching a $2$-handle that is cancelled with a $3$-handle. 

Note that the image of $\partial D_\gamma$ by $\pi$ is contained in $\gamma$. 
Then we can define a map $\rho:\partial D_\gamma\to\gamma$ by $\rho(x)=\pi(x)$. 
There are two important cases: (i) $\mathrm{deg}(\rho)=0$ and (ii) $|\mathrm{deg}(\rho)|=1$. 
In other words, 
\begin{itemize}
 \item[(i)]
$\partial D'$ is null-homotopic in $\Nbd(\gamma;X)$, and 
 \item[(ii)]
$\partial D'$ is homotopic to $\gamma$ in $\Nbd(\gamma;X)$. 
\end{itemize}
\begin{figure}[tbp]
\includegraphics[width=100mm]{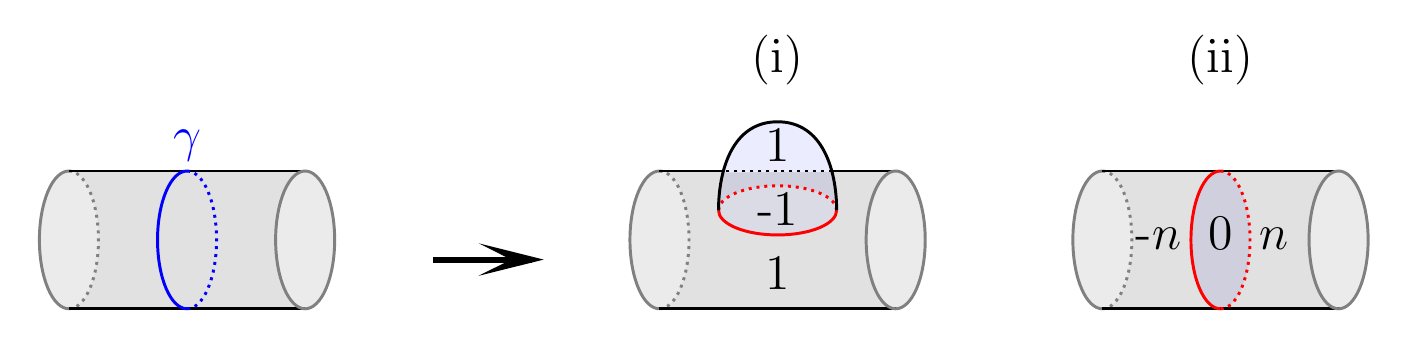}
\caption{(i) a vertical compressing disk and (ii) a horizontal compressing disk. }
\label{fig:compressing_disk}
\end{figure}
The disk $D_\gamma$ is said to be {\it vertical} if (i), and {\it horizontal} if (ii). 
Figure~\ref{fig:compressing_disk} shows the modification of $X$ to $X\cup D'_\gamma$ 
in the cases (i) and (ii). 
\begin{remark}
If $\gamma$ is a small circle bounding a disk in a region and has a vertical compressing disk, 
then we often say that the region has a vertical compressing disk. 
Clearly, any such $\gamma$ has a horizontal compressing disk. 
\end{remark}
\subsection{Connected-sum and reduction}

Suppose that $X$ has a disk region $D$ such that 
$\Nbd(\partial D;X)$ is homeomorphic to $Y_{111}$ and $\gl(D)=0$. 
The neighborhood $\Nbd(D;X)$ is shown in the left of Figure~\ref{fig:conn_sum}. 
Note that there exists a smooth $3$-ball $B_D$ in $S^4$ such that $\Nbd(D;X)\subset B_D$ since $\gl(D)=0$. 
We consider a modification as shown in Figure~\ref{fig:conn_sum}. 
Precisely, we first remove $\Int\Nbd(D;X)$ from $X$ and then 
cap each of the resulting boundary circles off a $2$-disk. 
This modification can be performed locally in $B_D$ as in the figure. 
Since $X$ is simply-connected, 
this modification produces two new simple polyhedra $X'_1$ and $X'_2$. 
We suppose that $X'_1$ contains the whole $K$, and 
we say that $X'_1$ is obtained from $X$ by the {\it connected-sum reduction} along $D$. 

\begin{proposition}
\label{prop:connected_sum}
Suppose $c(X)\leq1$ and that $X'_1$ is obtained from $X$ by the {\it connected-sum reduction} along a disk region $D$. 
Then $X'_1$ is a shadow of $K$. 
\end{proposition}
\begin{figure}[tbp]
\includegraphics[width=80mm]{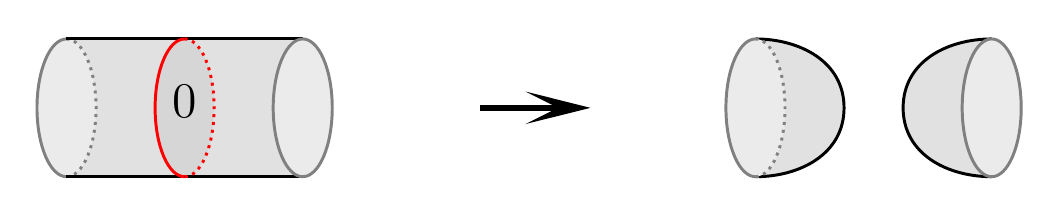}
\caption{connected sum.}
\label{fig:conn_sum}
\end{figure}

\begin{proof}
Let $W$ denote the $4$-sphere in which $K$ and $X$ are embedded. 
By \cite[Proposition 4.1]{Mar11}, 
the $4$-sphere $W$ can be decomposed as $W_1\# W_2$, 
where $W_1$ and $W_2$ are closed $4$-manifolds admitting shadows $X'_1$ and $X'_2$, respectively. 
Since either $X'_1$ or $X'_2$ has no true vertices by $c(X)\leq1$, 
the corresponding $4$-manifold, namely $W_1$ or $W_2$, is diffeomorphic to $S^4$ by \cite[Corollary 1.8]{Mar11}. 
Thus, both $W_1$ and $W_2$ are diffeomorphic to $S^4$. 
Then $W\setminus \Int\Nbd(X'_1;W)$ is diffeomorphic to a $4$-dimensional $1$-handlebody, 
and hence $X'_1$ is a shadow of $K$. 
\end{proof}

\subsection{Lemmas on encoding graphs} 
Here we prepare some lemmas 
about the shape and the types of the vertices of an encoding graph of $X$. 

\begin{lemma}[{\cite[Claim 1]{Mar11}}]
\label{lem:loop_shrinking_lemma}
Suppose that a loop $\gamma\subset X$ separates 
$X$ into two connected components $V_1$ and $V_2$. 
Then both $V'_1$ and $V'_2$ are simply-connected, 
where $V'_1=V_1\cup_{\gamma}D^2$ and $V'_2=V_2\cup_{\gamma}D^2$. 
\end{lemma}
\begin{proof}
The quotient space $X/\gamma$ is homeomorphic to the wedge sum $V'_1\vee V'_2$. 
Then we have a surjection from $\pi_1(X)=\{1\}$ onto $\pi_1(V'_1)\ast\pi_1(V'_2)$. 
\end{proof}

We rephrase the above lemma in terms of encoding graphs as below. 
\begin{lemma}
\label{lem:loop_shrinking_lemma_for_graph}
Let $G$ be a tree encoding $X$ and $G'\subsetneq G$ be a subgraph. 
Then a simple polyhedron encoded by the $\vt{D}$-closure $\widehat{G'}$ is also simply-connected.
\end{lemma}
\begin{lemma}
\label{lem:graph_lemma}
Suppose $c(X)\leq1$. 
Let $G$ be a graph encoding $X$. 
\begin{enumerate}
 \item
$G$ is a tree. 
 \item
$G$ does not have a vertex of type~$\vt{Y\!{}_{2}}$, $\vt{Y\!{}_{3}}$, 
$\vt{X_1}$, $\vt{X_{2}}$, $\vt{X_{5}}$, $\vt{X_{6}}$ nor $\vt{X_{7}}$. 
\item
$G$ does not contain a subgraph shown in Figure~\ref{fig:graph_lemma}-(i) nor -(ii). 
\end{enumerate}
\end{lemma}
\begin{figure}[tbp]
\includegraphics[width=75mm]{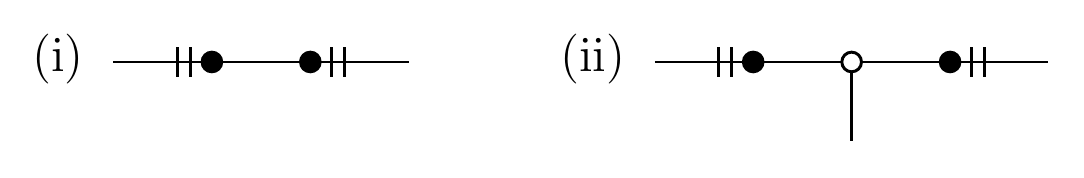}
\caption{Graphs those are not contained in a graph encoding a simply-connected polyhedron.}
\label{fig:graph_lemma}
\end{figure}
\begin{proof}
\begin{enumerate}
 \item
We can embed $G$ in $X$ as a retract, and hence there is a surjection $\pi_1(X)\to\pi_1(G)$. 
\item
Assume that $G$ has a vertex $v$ of type~$\vt{Y\!{}_{2}}$, $\vt{Y\!{}_{3}}$, 
$\vt{X_1}$, $\vt{X_{2}}$, $\vt{X_{5}}$, $\vt{X_{6}}$ or $\vt{X_{7}}$. 
Then the polyhedron encoded by the $\vt{D}$-closure $\widehat{v}$ is not simply-connected, 
which is a contradiction to Lemma~\ref{lem:loop_shrinking_lemma_for_graph}. 
\item
The proof is similar to that of (2). 
\end{enumerate}
\end{proof}
\begin{remark}
\begin{enumerate}
 \item
The fundamental groups of the special polyhedra encoded by the $\vt{D}$-closures 
of vertices of types~$\vt{Y\!{}_{2}}$, $\vt{Y\!{}_{3}}$, 
$\vt{X_1}$, $\vt{X_{2}}$, $\vt{X_{5}}$, $\vt{X_{6}}$ and $\vt{X_{7}}$ are isomorphic to 
$\Z/2\Z$, $\Z/3\Z$, 
$\langle x,y \mid xyx^{-2}y^{-2}\rangle$, 
$\langle x,y \mid xyx^{2}y^{-2}\rangle$, 
$\Z/3\Z$, $\Z/4\Z$ and $\Z/5\Z$, respectively. 
\item
The two polyhedra encoded by the $\vt{D}$-closures of the subgraphs shown in Figure~\ref{fig:graph_lemma}-(i) and -(ii) are homeomorphic, 
and their fundamental groups are isomorphic to  $\Z/2\Z$. 
\end{enumerate}
\end{remark}

\subsection{Lemmas on the fundamental groups of subpolyhedra} 
In this subsection, we present 4 lemmas on the fundamental group of subpolyhedron in $X$. 
Lemmas~\ref{lem:pi_1_of_subpolyhedron1} and \ref{lem:pi_1_of_subpolyhedron1-2} treat a subpolyhedron with one cut end, and  
Lemmas~\ref{lem:pi_1_of_subpolyhedron2} and \ref{lem:pi_1_of_subpolyhedron2-2} treat a subpolyhedron with two cut ends. 
Note that we will assume that $X$ is closed in these lemmas, 
which actually does not matter for our main theorems due to Lemma~\ref{lem:can_make_closed}. 

\begin{definition}
Let $G$ be a tree graph encoding a simple polyhedron.
Let $v$ and $v'$ be vertices of $G$ and $v'$ is of type~$\vt{Y{}\!_{12}}$. 
If the edge incident to $v'$ with no dashes is contained in the shortest path from $v$ to $v'$, 
then $v'$ is said to be {\it one-sided to $v$}. 
Otherwise we say that $v'$ is {\it two-sided to $v$}. 
\end{definition}
\begin{lemma}
\label{lem:pi_1_of_subpolyhedron1}
Suppose $X$ is closed, and 
let $V$ be a subpolyhedron of $X$ with a single cut end $\gamma$ and $c(V)=0$. 
Then there exists a non-negative integer $k$ such that 
$\pi_1(V)\cong\langle \gamma\mid \gamma^{2^k} \rangle$. 
\end{lemma}
\begin{proof}
Let $G$ be an encoding graph of $V$. 
From Lemma~\ref{lem:loop_shrinking_lemma}, the polyhedron $V\cup_{\gamma}D^2$ is simply-connected, 
and $G$ is a tree from Lemma~\ref{lem:graph_lemma}~(1). 
This graph $G$ has exactly one vertex of type~$\vt{B}$ corresponding to $\gamma$, 
and let $v_\gamma$ denote it. 
Other vertices in $G$ are of types~$\vt{D}$, $\vt{Y{}\!_{12}}$, $\vt{P}$ or $\vt{Y{}\!_{111}}$ 
by Lemma~\ref{lem:graph_lemma}~(2). 
Note that the valencies of types~$\vt{D}$, $\vt{Y{}\!_{12}}$, $\vt{P}$ and $\vt{Y{}\!_{111}}$ are $1, 2, 3$ and $3$, respectively. 

\begin{figure}[tbp]
\includegraphics[width=90mm]{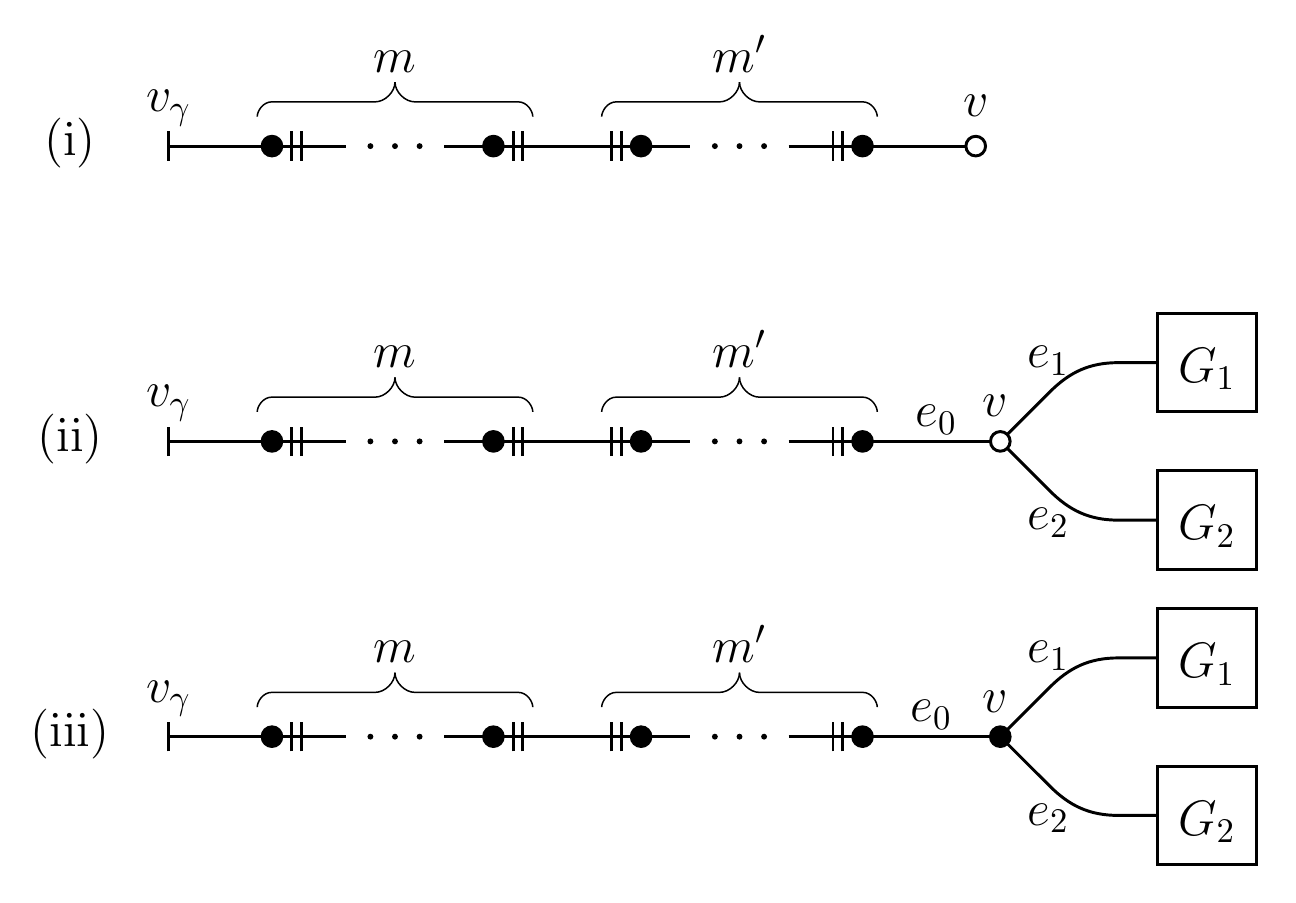}
\caption{The possible cases of a graph $G$ encoding $V$. }
\label{fig:graph_lemma6}
\end{figure}
We give an orientation to $\gamma$ arbitrarily. 
Let $v$ be the nearest vertex from $v_\gamma$ among those of types~$\vt{D}$, $\vt{P}$ and $\vt{Y{}\!_{111}}$. 
By Lemma~\ref{lem:graph_lemma}-(3), 
the possible cases are shown in Figure~\ref{fig:graph_lemma6}, where $m,m'\in\Z_{\geq0}$. 
We prove the argument by induction on the number of vertices of 
types~$\vt{D}$, $\vt{P}$ and $\vt{Y{}\!_{111}}$ in $G$, 
so it is enough to consider the following (i)-(iii).  
\begin{itemize}
 \item[(i)]
Suppose that $G$ is as shown in Figure~\ref{fig:graph_lemma6}-(i). 
Then we easily obtain $\pi_1(V)\cong\langle \gamma\mid \gamma^{2^{m}}\rangle$. 
 \item[(ii)]
Suppose that $G$ is as shown in Figure~\ref{fig:graph_lemma6}-(ii) 
and that $\pi_1(V_i)\cong\langle \gamma_i\mid \gamma_i^{2^{k_i}}\rangle$ for $i\in\{1,2\}$, 
where $V_i$ is the polyhedron encoded by the subgraph $G_i$ and $\gamma_i=\partial V_i$. 
Let $\gamma_0$ be a lift of $e_0$. 
Note that $\gamma_i$ is a lift of $e_i$ for $i\in\{1,2\}$. 
By Lemma~\ref{lem:loop_shrinking_lemma_for_graph} applied to the subgraph $v\cup G_1\cup G_2$, the group  
$\langle \gamma_1,\gamma_2\mid \gamma_1^{2^{k_1}},\ \gamma_2^{2^{k_2}},\ \gamma_1\gamma_2\rangle$
must be trivial, and hence we have $k_1=0$ or $k_2=0$. 
Suppose $k_2=0$. 
Then we have 
\begin{align*}
\pi_1(V)
&\cong\langle 
\gamma, \gamma_0,\gamma_1,\gamma_2
\mid 
\gamma^{2^{m}}\gamma_0^{2^{m'}},\ \gamma_1^{2^{k_1}},\ \gamma_2^{2^{k_2}},\ \gamma_0\gamma_1\gamma_2
\rangle. \\
&\cong\langle 
\gamma, \gamma_0,\gamma_1
\mid 
\gamma^{2^{m}}\gamma_0^{2^{m'}},\ \gamma_1^{2^{k_1}},\ \gamma_0\gamma_1
\rangle. \\
&\cong\langle 
\gamma, \gamma_0
\mid 
\gamma^{2^{m}}\gamma_0^{2^{m'}},\ \gamma_0^{2^{k_1}}
\rangle. 
\end{align*}
It becomes trivial by adding a relation $\gamma=1$ since $\pi_1(V\cup_{\gamma}D^2)=\{1\}$. 
Hence $m'=0$ or $k_1=0$, and the lemma follows in either case. 
 \item[(iii)]
Suppose that $G$ is as shown in Figure~\ref{fig:graph_lemma6}-(iii) 
and that $\pi_1(V_i)\cong\langle \gamma_i\mid \gamma_i^{2^{k_i}}\rangle$ 
for $i\in\{1,2\}$, 
where $V_i$ is the polyhedron encoded by the subgraph $G_i$ and $\gamma_i=\partial V_i$. 
Let $\gamma_0$ be a lift of $e_0$. 
We have 
\begin{align*}
\pi_1(V)
&\cong\langle 
\gamma, \gamma_0,\gamma_1,\gamma_2
\mid 
\gamma^{2^{m}}\gamma_0^{2^{m'}},\ \gamma_1^{2^{k_1}},\ \gamma_2^{2^{k_2}},\ \gamma_0=\gamma_1=\gamma_2
\rangle. \\
&\cong\langle 
\gamma, \gamma_0
\mid 
\gamma^{2^{m}}\gamma_0^{2^{m'}},\ \gamma_0^{2^{k_3}} 
\rangle, \\
\end{align*}
where $k_3=\min\{k_1,k_2\}$. 
It becomes trivial by adding a relation $\gamma=1$ since $\pi_1(V\cup_{\gamma}D^2)=\{1\}$. 
Hence $m'=0$ or $k_3=0$, and the lemma follows in either case. 
\end{itemize}
\end{proof}

For each vertex $v$ of type~$\vt{D}$,  
let $k(v)$ denote the number of vertices of type~$\vt{Y{}\!_{12}}$ one-sided to $v_\gamma$ contained in the geodesic from $v_\gamma$ to $v$. 
By the proof of Lemma~\ref{lem:pi_1_of_subpolyhedron1}, 
the integer $k$ stated in the lemma is given as 
the minimum of $k(v)$ for any vertex $v$ of type~$\vt{D}$. 
Thus we also have proved the following: 
\begin{lemma}
\label{lem:pi_1_of_subpolyhedron1-2}
Under the same notation as in Lemma~\ref{lem:pi_1_of_subpolyhedron1} and its proof, 
if $\pi_1(V)=\{1\}$, then there exists a leaf 
such that all the vertices of type~$\vt{Y{}\!_{12}}$ contained 
in the geodesic (in $G$) from $v_\gamma$ to the leaf are two-sided to $v_\gamma$. 
\end{lemma}


\begin{lemma}
\label{lem:pi_1_of_subpolyhedron2}
Suppose $X$ is closed, and 
let $U$ be a subpolyhedron of $X$ with exactly two cut ends $\gamma_1\sqcup \gamma_2$ and 
$c(U)=0$. 
Suppose that $[\gamma_1]$ is not a torsion element in $H_1(U)$. 
\begin{enumerate}
\item
$[\gamma_2]$ is also not a torsion element in $H_1(U)$. 
\item
$\pi_1(U)\cong\langle \gamma_1,\gamma_2 \mid (\gamma_1^{2^m} \gamma_2^{2^l})^{2^k}\rangle$ for some $k,l,m\in\Z_{\geq0}$. 
\end{enumerate}
\end{lemma}
\begin{proof}
(1)
We have $c(U\cup_{\gamma_2}D^2)=0$. 
Hence the fundamental group of $U\cup_{\gamma_2}D^2$ is a finite cyclic group generated by $\gamma_1$ by Lemma~\ref{lem:pi_1_of_subpolyhedron1}, 
and $[\gamma_1]$ is a torsion element in $H_1(U\cup_{\gamma_2}D^2)$. 
Suppose, to derive a contradiction, that $[\gamma_2]$ is a torsion element in $H_1(U)$. 
Then $[\gamma_1]$ and $[\gamma_2]$ are linearly independent in $H_1(U)$. 
Hence $[\gamma_1]$ is not a torsion element also in $H_1(U\cup_{\gamma_2}D^2)$, 
which is a contradiction. \\[2mm] 
(2)
Let $G$ be a tree encoding $U$, 
and let $v_1$ and $v_2$ be the vertices of type~$\vt{B}$ corresponding to $\gamma_1$ and $\gamma_2$, respectively. 
Let $\ell$ be the geodesic from $v_1$ to $v_2$. 
The vertices in $G$ other than $v_1$ or $v_2$ are of types~$\vt{D}$, $\vt{Y{}\!_{12}}$, $\vt{P}$ or $\vt{Y{}\!_{111}}$ 

We assume that $\ell$ contains a vertex $v$ of type~$\vt{Y{}\!_{111}}$ and lead to contradiction. 
Recall that there are three edges incident to $v$. See Figure~\ref{fig:graph_lemma7}-(i-1).
\begin{figure}[tbp]
\includegraphics[width=110mm]{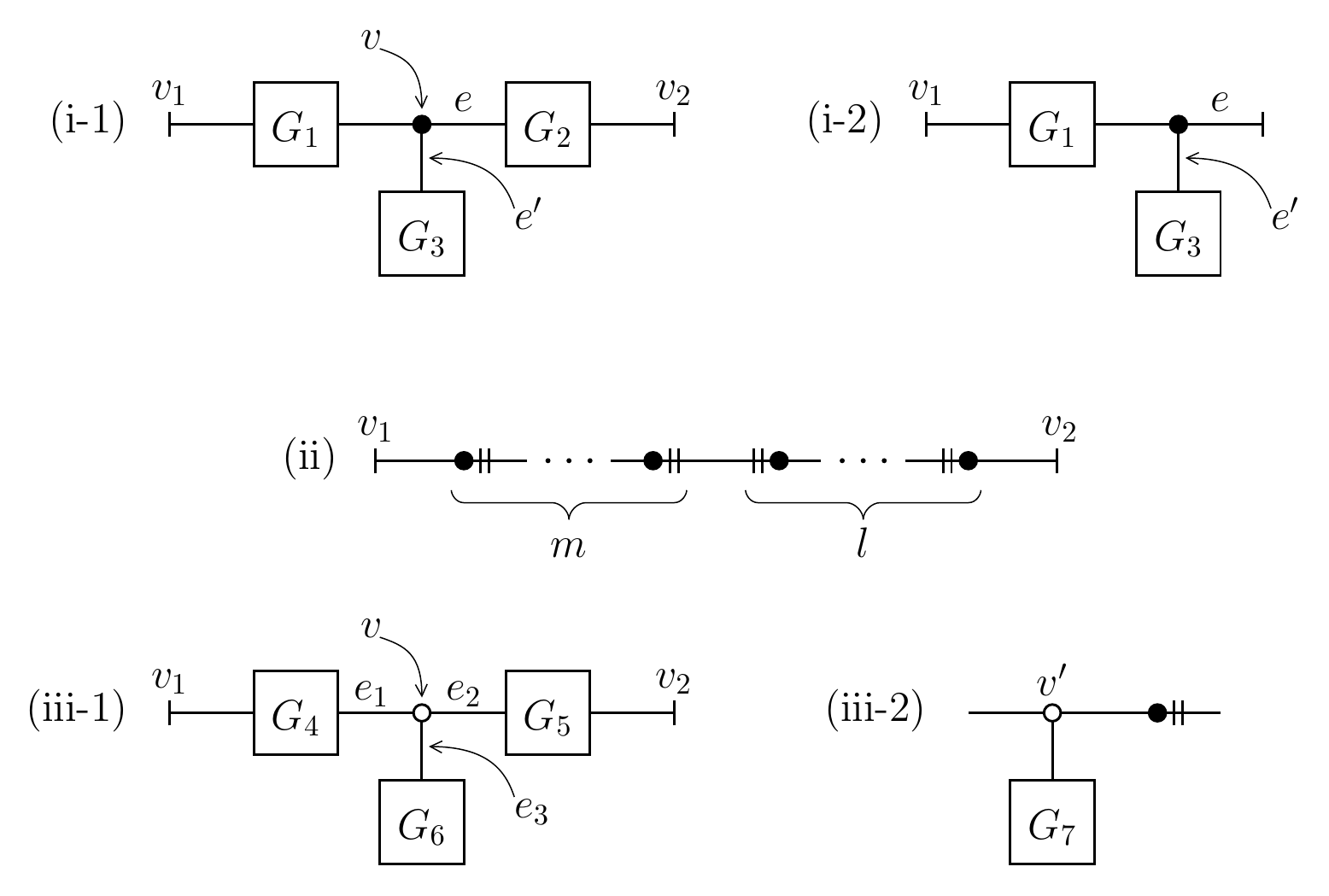}
\caption{Encoding graphs of $U$ used in the proof of Lemma~\ref{lem:pi_1_of_subpolyhedron2}. }
\label{fig:graph_lemma7}
\end{figure}
Let $e$ be the edge such that it is incident to $v$ and between $v$ and $v_2$, 
and let $e'$ be the edge incident to $v$ and not on $\ell$. 
Let $G_1, G_2$ and $G_3$ be subgraphs of $G$ as indicated in Figure~\ref{fig:graph_lemma7}-(i-1).
Now let $U'$ be one of the components containing $\gamma_1$ 
obtained by cutting $U$ along a lift of $e$, 
which is encoded in Figure~\ref{fig:graph_lemma7}-(i-2). 
This subpolyhedron $U'$ has two boundary components: 
one is $\gamma_1$ and the other, namely a lift of $e$, will be denoted by $\gamma'_2$. 
Since $U'$ itself satisfies the assumption of the lemma, the cycle $[\gamma'_2]$ is not a torsion element in $H_1(U')$ by (1). 
It is homologous to a cycle represented by a lift of $e'$, 
which is a torsion element in the subpolyhedron encoded by $G_3$ by Lemma~\ref{lem:pi_1_of_subpolyhedron1}. 
It is a contradiction. 
Therefore, the vertices between $v_1$ and $v_2$ are of types~$\vt{Y{}\!_{12}}$ or $\vt{P}$. 

If $\ell$ does not contain vertices of type~$\vt{P}$, 
the graph $G$ is as shown in Figure~\ref{fig:graph_lemma7}-(ii) by Lemma~\ref{lem:graph_lemma}~(3). 
We then have $\pi_1(U)\cong\langle \gamma_1,\gamma_2 \mid \gamma_1^{2^m} \gamma_2^{2^l}\rangle$. 

We then suppose that $G$ has a vertex $v$ of type~$\vt{P}$, and $G$ is as shown in Figure~\ref{fig:graph_lemma7}-(iii-1). 
We can assume that each vertex of type~$\vt{Y{}\!_{12}}$ in $\ell$ is two-sided to $v$ by Lemma~\ref{lem:graph_lemma}~(3). 
Suppose that there is a subgraph of $G_4$ as shown in Figure~\ref{fig:graph_lemma7}-(iii-2), 
where $v'$ is a vertex of type~$\vt{P}$ contained in $\ell$. 
The fundamental group of the subpolyhedron encoded by $G_7$ is 
isomorphic to $\langle \gamma'\mid \gamma'^{2^{k'}} \rangle$ for some $k'\in\Z_{\geq0}$ by Lemma~\ref{lem:pi_1_of_subpolyhedron1}, 
where $\gamma'$ is the boundary of the subpolyhedron. 
The simple polyhedron encoded by the $\vt{D}$-closure of the graph shown in Figure~\ref{fig:graph_lemma7}-(iii-2) 
is simply-connected by Lemma~\ref{lem:loop_shrinking_lemma_for_graph}. 
Hence the subpolyhedron encoded by $G_7$ must be simply-connected as well. 
Therefore, the graph $G_4$ in Figure~\ref{fig:graph_lemma7}-(iii-1) encodes a polyhedron whose fundamental group is 
presented by $\langle \gamma_1, u_1\mid u_1\gamma_1^{2^m}\rangle$, 
where $u_1$ is a lift of $e_1$ and $m$ is the number of vertices of type~$\vt{Y{}\!_{12}}$ in $G_4\cap\ell$. 
Similarly, the graph $G_5$ in Figure~\ref{fig:graph_lemma7}-(iii-1) encodes a polyhedron whose fundamental group is 
presented by $\langle \gamma_2, u_2\mid u_2\gamma_2^{2^l}\rangle$, 
where $u_2$ is a lift of $e_2$ and $l$ is the number of vertices of type~$\vt{Y{}\!_{12}}$ in $G_5\cap\ell$. 
By Lemma~\ref{lem:pi_1_of_subpolyhedron1}, 
the graph $G_6$ in Figure~\ref{fig:graph_lemma7}-(iii-1) encodes a polyhedron whose fundamental group is 
presented by $\langle u_3\mid u_3^{2^k}\rangle$ for some $k\in\in\Z_{\geq0}$, 
where $u_3$ is a lift of $e_3$. 
Thus, we obtain a presentation
\begin{align*}
\pi_1(U)
&\cong
\langle 
\gamma_1,\gamma_2,u_1,u_2,u_3
\mid
u_1\gamma_1^{2^m},\ u_2\gamma_2^{2^l} ,\ u_3^{2^k},\ u_1u_2u_3
\rangle\\
&\cong
\langle 
\gamma_1,\gamma_2
\mid
(\gamma_1^{2^m}\gamma_2^{2^l})^{2^k}
\rangle. 
\end{align*}
\end{proof}
From the above proof, we immediately obtain the following lemma, 
which will be used in the proof of Theorem~\ref{thm:sc=1_onlyifpart}. 
\begin{lemma}
\label{lem:pi_1_of_subpolyhedron2-2}
Under the same notation as in Lemma~\ref{lem:pi_1_of_subpolyhedron2} and its proof, 
if $\pi_1(U)\cong\langle \gamma_1,\gamma_2\mid \gamma_1^{2^m}\gamma_2\rangle$, 
the following holds:
\begin{itemize}
\item
any vertex lying in $\ell$ is of type~$\vt{Y{}\!_{12}}$ or $\vt{P}$, 
\item
$m$ is the number of the vertices of type~$\vt{Y{}\!_{12}}$ lying in $\ell$, 
all of which are one-sided to $v_1$, and 
\item
the subpolyhedron corresponding to each connected component of 
$G\setminus\ell$ is simply-connected. 
\end{itemize}
\end{lemma}

\section{$2$-knots with complexity zero}
\label{sec:2-knots with complexity zero}
From now on, we discuss the classification of $2$-knots according to the shadow-complexity, 
and we give the proof of the theorem for the case $\shco(K)=0$ in this section. 

Let us start with the following lemma. 
It is an analogue of \cite[Lemma~1.3]{KN20}, 
and the original statement in the paper is for shadows of ``4-manifolds''. 
Roughly speaking, the shadow-complexity of any $2$-knot is always attained by a closed shadow. 
\begin{lemma}
\label{lem:can_make_closed} 
If $\shco(K)=n$, 
then $K$ admits a closed shadow with complexity exactly $n$. 
\end{lemma}
\begin{proof}
The proof is almost the same as that of \cite[Lemma~1.3]{KN20}, 
so we only sketch the proof. 

Let $X$ be a shadow of $K$ with $c(X)=n$ and $\partial X\ne\emptyset$. 
Then $X$ collapses onto an almost-simple polyhedron $Y$ 
(see \cite{Mat03} for the definition and the details) 
that is minimum with respect to collapsing and has at most $n$ true vertices. 
Note that the collapsing is done in a regular neighborhood $\Nbd(X;S^4)$, 
which is also a regular neighborhood of $Y$ in $S^4$. 
Since $K$ is a $2$-sphere embedded in $X$, 
it is kept by collapsing, that is, $K$ is also embedded in $Y$. 
The polyhedron $Y$ is either a closed simple polyhedron 
or the union of a closed simple polyhedron and a graph. 
If the former, this $Y$ is what we required. 
If the latter, as in the proof of \cite[Lemma~1.3]{KN20}, 
$Y$ can be modified to a closed simple polyhedron $Y'$ such that 
$Y$ and $Y'$ have the same regular neighborhood in $S^4$ 
and also that no new true vertices are created by the modification. 
\end{proof}

As well as in Lemma~\ref{lem:pi_1_of_subpolyhedron1}, 
we consider a subpolyhedron having one cut end in the next lemma. 
However, unlike in Lemma~\ref{lem:pi_1_of_subpolyhedron1}, 
Lemma~\ref{lem:H_1_of_subpolyhedron} gives a homological condition, 
and a subpolyhedron can have true vertices 
and boundary components other than the cut end. 
Recall the notation $\mathrm{gl}(\gamma)$ defined in the formula (\ref{align:gleam_of_chain}). 

\begin{lemma}
\label{lem:H_1_of_subpolyhedron}
Set $X_K=\Nbd(K;X)$, and 
suppose that $\partial X_K$ has a circle component $\gamma$. 
Let $X'$ be a connected component of $X\setminus \Int X_K$ with 
$X_K\cap\partial X'=\gamma$. 
Give orientations to $K$ and $\gamma$ arbitrarily. 
Then at least one of the following holds:
\begin{itemize}
 \item
$H_1(X')$ is an infinite cyclic group generated by $\gamma$, or  
\item
$\mathrm{gl}(\gamma)=0$. 
\end{itemize}
\end{lemma}
\begin{proof}
We have $H_1(\gamma)\cong\Z\langle\gamma\rangle$ and 
$H_1(X_K)\cong H_1(X)\cong0$. 
From the Mayer-Vietoris sequence
\[
H_1(\gamma)\to H_1(X_K)\oplus H_1(X')\to H_1(X), 
\]
$H_1(X')$ admits a surjection from $\Z$ generated by $[\gamma]$. 

If $H_1(X')=\Z/k\Z\langle\gamma\rangle$ for some $k\in\Z_{>0}$, 
there is a $2$-chain $C$ in $X'$ such that $\partial C=k[\gamma]$. 
Let $D=\sum a_j R_j$ be a $2$-chain in $X_K$ with $\partial D=[\gamma]$, 
where $R_j$ is a region contained in $K$. 
Then $C-kD$ is a homology cycle in $X$, and we have 
\[
Q([K],C-kD)=-k\sum a_j\gl(R_j)=-k\mathrm{gl}(\gamma), 
\]
where $Q$ is the intersection form on $H_2(X)$. 
Since the intersection form of $S^4$ is trivial, 
we have $\mathrm{gl}(\gamma)=0$. 
\end{proof}
\begin{lemma}
\label{lem:disk_region_on_K}
Suppose $X$ is closed and 
that $S(X)\cap K$ has a circle component $\gamma$ bounding a disk region $D$ on $K$. 
Let $X'$ be a connected component of $X\setminus K$ with $\partial X'=\gamma$. 
If $X'$ does not contain true vertices, then $X\setminus X'$ is a shadow of $K$. 
\end{lemma}
\begin{proof}
From Lemma~\ref{lem:pi_1_of_subpolyhedron1}, 
$H_1(X')\cong \Z/2^k\Z$ for some $k\in\Z_{\geq0}$. 
Then $\gl(D)=0$ by Lemma~\ref{lem:H_1_of_subpolyhedron}, 
and $X\setminus X'$ is a shadow of $K$ by Lemma~\ref{prop:connected_sum}. 
\end{proof}
\begin{theorem}
\label{thm:complexity0}
A $2$-knot $K$ is unknotted if and only if $\shco(K)=0$. 
\end{theorem}
\begin{proof}
Let $X$ be a shadow of a $2$-knot $K$ with $c(X)=0$. 
By Lemma~\ref{lem:can_make_closed}, we can assume $\partial X=\emptyset$. 
Since $c(X)=0$, 
each component of $S(X)\cap K$ is a circle, 
and each component of $X\setminus K$ does not contain true vertices. 
By iterating Lemma~\ref{lem:disk_region_on_K}, $K$ admits itself as a shadow. 
Hence $K$ is unknotted. 
The converse is obvious. 
\end{proof}
\section{Existence of $2$-knots with complexity one}
\label{sec:Existence of 2-knots with complexity one}

From here, we always assume the following. 
\begin{itemize}
 \item
$X$ is a closed shadow of a $2$-knot $K$, and 
\item
$c(X)=1$. 
\end{itemize}
Note that $K\cap S(X)$ is not empty. 
Then there are two cases; 
(i) the true vertex is contained in a component of $K\cap S(X)$, which is an 8-shaped graph, and 
(ii) the true vertex does not lie on $K$ and every component of $K\cap S(X)$ is a circle. 
Therefore, 
we can also assume the following by Lemma~\ref{lem:disk_region_on_K}. 
\begin{itemize}
\item
$K\cap S(X)$ is connected, and it is a circle or an 8-shaped graph.  
\end{itemize}
Moreover, we fix the notation 
\begin{itemize}
 \item
$X'=X\setminus \Int\Nbd(K;X)$, 
\end{itemize}
which is also connected by the above assumption. 

\subsection{True vertex lies on $K$} 
Suppose that the true vertex of $X$ lies on $K$. 
Let us first consider the case of special shadow-complexity 1. 
We need the following lemma. 
\begin{lemma}
\label{lem:handle_decomp_2-knot}
Let $K$ be a $2$-knot. 
Suppose that $S^4$ admits a decomposition consisting of 
$\Nbd(K;S^4)$, one $2$-handle, two $3$-handles and one $4$-handle. 
Then $K$ is unknotted. 
\end{lemma}
\begin{proof}
We regard $\Nbd(K;S^4)$ as the union of a $0$-handle $h_0$ and a $2$-handle $h_2$. 
This $2$-handle $h_2$ is attached along the $0$-framed unknot $L$ lying on the boundary of $h_0$ 
since $\Nbd(K;S^4)\cong S^2\times D^2$. 
Let $h'_2$ be the other $2$-handle in the decomposition, 
and let $L'$ be its attaching circle. 
This famed knot $L'$ is contained in $\partial h_0\setminus \nu L$, 
and hence $L\sqcup L'$ forms a $2$-component link in $\partial h_0\cong S^3$. 
By \cite[Proposition 3.2]{GST10}, we can modify the link $L\sqcup L'$ into the unlink 
only by using handle-slides of $L'$ over $L$. 
Then $L'$ is contained in a small $3$-ball in the boundary of $h_0\cup h_2$ ($=\Nbd(\K;S^4)$), 
and the $2$-sphere $K$ can be pushed to the boundary of $h_0\cup h_2\cup h_2'$ by isotopy. 
The resulting $2$-sphere plays a role of the attaching sphere of a $3$-handle by \cite{LP72}. 
Hence $K$ bounds a $3$-ball in $S^4$, which is the definition of $K$ to be unknotted. 
\end{proof}
\begin{theorem}
\label{thm:no_spsc=1}
There are no $2$-knots with special shadow-complexity $1$.
\end{theorem}
\begin{proof}
Let us suppose that $X$ is a special shadow $X$ of $K$ with $c(X)=1$. 
Then $S(X)$ is connected and $S(X)\subset K$. 
Moreover, it is homeomorphic to an 8-shaped graph, and $\Nbd(S;K)$ is homeomorphic to a pair of pants. 
Hence $\Nbd(S;X)$ is homeomorphic to $X_{11}$ (see Subsection~\ref{subsec:Top types of nbd of S(X)}), 
and $X'$ is a $2$-disk. 
Thus, $S^4$ is decomposed as in Lemma~\ref{lem:handle_decomp_2-knot}, 
and $K$ is unknotted and $\spshco(K)=0$. 
\end{proof}

We next consider the non-special case. 
\begin{proposition}
\label{prop:one_true_vertex_on_K}
If the true vertex of $X$ lies on $K$, 
then $G(K)$ is an infinite cyclic group. 
\end{proposition}
\begin{proof}
Let $S$ be the component of $K\cap S(X)$ containing the true vertex. 
As well as in the proof of Theorem~\ref{thm:no_spsc=1}, 
$\Nbd(S;X)$ is homeomorphic to $X_{11}$, 
and $\partial X'=X'\cap \Nbd(K;X)$ is a circle. 
By Lemma~\ref{lem:pi_1_of_subpolyhedron1},
we have $\pi_1(X')\cong \langle \gamma\mid \gamma^{2^k} \rangle$ for some $k\in\Z_{\geq0}$, 
where $\gamma=\partial X'$. 
Then we have $G(K)\cong \langle \gamma,\mu\mid \gamma^{2^k}, \gamma\mu^{-\mathrm{gl}(\gamma)}$ by Proposition~\ref{prop:knot_group}. 
Since $\mathrm{gl}(\gamma)=0$ by Lemma~\ref{lem:H_1_of_subpolyhedron}, 
$G(K)$ is an infinite cyclic group. 
\end{proof}

\subsection{True vertex does not lie on $K$} 
Hereafter we suppose that the true vertex of $X$ does not lie on $K$, 
that is, it is contained in $X'$. 
In this subsection, we investigate the knot group of $K$. 

The part $S(X)\cap K$ of the singular set separates $K$ into two disk regions, 
and their gleams are $g$ and $-g$ for some $g\in\Z_{\geq0}$ since the self-intersection number of $K$ is $0$. 

If $g=0$, $K$ is a shadow of itself by Proposition~\ref{prop:connected_sum}. 
Then $K$ is unknotted. 

We henceforth suppose $g>0$. 
Let us orient $K$ arbitrarily, and 
then an oriented meridian $\mu$ of $K$ is defined. 
Set $\gamma=\partial X'$, 
and let $G$ be a graph encoding $X'$. 
This graph $G$ has exactly one vetex of type~$\vt{B}$, 
which corresponds to $\gamma$ and will be denoted by $v$. 
By Lemma~\mbox{\ref{lem:graph_lemma}-(2)}, 
$G$ have exactly one vertex $v_0$ of type~$\vt{X_3}$, $\vt{X_4}$, 
$\vt{X_8}$, $\vt{X_9}$, $\vt{X_{10}}$ or $\vt{X_{11}}$, 
and hence $G$ is one of those shown in Figures~\ref{fig:c=1cases1}, \ref{fig:c=1cases2} and \ref{fig:c=1cases3}. 
Note that 
$X_8$ and $X_{11}$ have symmetries that interchange 
two boundary components with the same length. 
Let $X_{v_0}$ be the connected component of $\Nbd(S(X);X)$ corresponding to $v_0$, 
which is homeomorphic to one of $X_3$, $X_4$, $X_8$, $X_9$, $X_{10}$ or $X_{11}$. 

Let $\mathfrak{a}$ denote the abelianization map of a group. 
In the following, we discuss what kind of $2$-knot admits a shadow encoded by a graph in Figures~\ref{fig:c=1cases1}, \ref{fig:c=1cases2} and \ref{fig:c=1cases3}. 
\begin{figure}[tbp]
\includegraphics[width=.85\hsize]{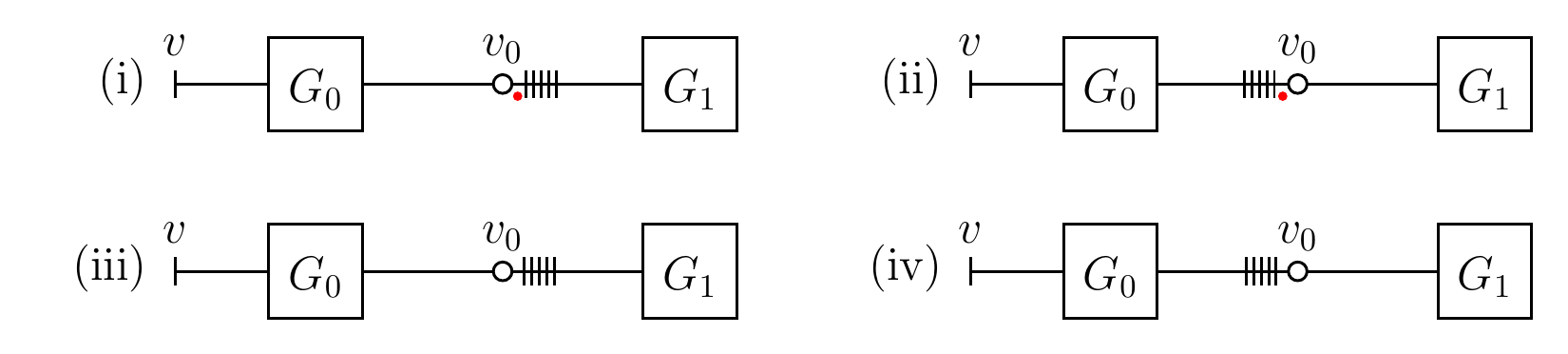}
\caption{
The possible cases of encoding graph $G$ of $X$ 
such that $v_0$ is of type $\vt{X_3}$ or $\vt{X_{4}}$. }
\label{fig:c=1cases1}
\end{figure}
\begin{case}
[$v_0$ is of type $\vt{X_3}$ or $\vt{X_4}$]
The graph $G$ is one of those shown in Figure~\ref{fig:c=1cases1}. 
Let $G_0$ and $G_1$ be subgraphs of $G$ as indicated in the figure. 
The subpolyhedron $X'$ is decomposed into $U$, $X_{v_0}$ and $V$, 
where $U$ and $V$ are the subpolyhedron corresponding to $G_0$ and $G_1$, respectively. 
Note that $U$ has two cut ends $\gamma\sqcup\gamma_0$, 
and also note that $V$ has one cut end $\gamma_1$. 
Then we have $\pi_1(V)\cong\langle \gamma_1\mid\gamma_1^{2^{k_1}} \rangle$ 
by Lemma~\ref{lem:pi_1_of_subpolyhedron1} for some $k_1\in\Z_{\geq0}$. 
We also have $H_1(X')\cong\Z\langle\gamma\rangle$ by Lemma~\ref{lem:H_1_of_subpolyhedron}. 
Then we can apply Lemma~\ref{lem:pi_1_of_subpolyhedron2} to $U$, 
and we have $\pi_1(U)\cong\langle \gamma,\gamma_0\mid (\gamma^{2^m}\gamma_0^{2^l})^{2^{k_0}} \rangle$ 
for some $k_0,l,m\in\Z_{\geq0}$. 
\begin{lemma}
\label{lem:X34}
The following hold. 
\begin{enumerate}
\item
If $G$ is shown in Figure~\ref{fig:c=1cases1}-(i), 
then $\pi_1(U)\cong\langle \gamma,\gamma_0\mid \gamma^{2^m}\gamma_0\rangle$ and 
$V$ is simply-connected. 
Moreover, $G(K)\cong\langle x,\mu\mid x^2\mu^nx^{-1}\mu^{-n} \rangle$, where $n=2^mg$. 
\item
If $G$ is shown in Figure~\ref{fig:c=1cases1}-(ii), 
then $G(K)$ is an infinite cyclic group. 
\item
If $G$ is shown in Figure~\ref{fig:c=1cases1}-(iii), 
then $\pi_1(U)\cong\langle \gamma,\gamma_0\mid \gamma^{2^m}\gamma_0\rangle$ and 
$V$ is simply-connected. 
Moreover, $G(K)\cong\langle x,\mu\mid x^2\mu^nx^{-1}\mu^n \rangle$, where $n=2^mg$. 
\item
If $G$ is shown in Figure~\ref{fig:c=1cases1}-(iv), 
then $G(K)$ is an infinite cyclic group. 
\end{enumerate}
\end{lemma}

\begin{proof}
(1) 
The polyhedron $X'$ is decomposed as $U\cup X_3\cup V$, 
and $U$ and $V$ are glued with $X_3$ along the boundary components of $X_3$ with length 1 and 5, respectively. 
Then the fundamental group of $X'$ and its abelianization are obtained as follows: 
\begin{align*}
\pi_1(X')
&\cong
\langle
x,y,\gamma,\gamma_0,\gamma_1 
\mid
(\gamma^{2^m}\gamma_0^{2^l})^{2^{k_0}}, \gamma_1^{2^{k_1}}, \gamma_0=y, \gamma_1=xyx^{-2}y^{-1}
\rangle\\
&\cong
\langle
x,y,\gamma
\mid
(\gamma^{2^m}y^{2^l})^{2^{k_0}}, (xyx^{-2}y^{-1})^{2^{k_1}}
\rangle\\
&\overset{\mathfrak{a}}{\longrightarrow}
\Z\langle x\rangle\oplus\Z\langle y\rangle\oplus\Z\langle \gamma \rangle/
\Z\langle 2^{m+k_0}\gamma+2^{l+k_0}y\rangle\oplus\Z\langle 2^{k_1}x\rangle\\
&\cong H_1(X'). 
\end{align*}
It must be an infinite cyclic group generated by $[\gamma]$. 
Hence $k_0=k_1=l=0$. 
Thus we have the following: 
\begin{align*}
G(K)
&\cong
\langle
x,y,\gamma,\mu
\mid
(\gamma^{2^m}y^{2^l})^{2^{k_0}}, (xyx^{-2}y^{-1})^{2^{k_1}}, \gamma\mu^{-g}
\rangle\\
&\cong
\langle x,\mu\mid x^2\mu^nx^{-1}\mu^{-n} \rangle,
\end{align*}
where $n=2^mg$. 

\vspace{2mm}
\noindent (2)
The polyhedron $X'$ is decomposed as $U\cup X_3\cup V$, 
and $U$ and $V$ are glued with $X_3$ along the boundary components of $X_3$ with length 5 and 1, respectively. 
Then the fundamental group of $X'$ and its abelianization are obtained as follows: 
\begin{align*}
\pi_1(X')
&\cong
\langle
x,y,\gamma,\gamma_0,\gamma_1
\mid
y=\gamma_1, xyx^{-2}y^{-1}=\gamma_0, (\gamma^{2^m}\gamma_0^{2^l})^{2^{k_0}}, \gamma_1^{2^{k_1}}
\rangle\\
&\cong
\langle
x,y,\gamma
\mid
(\gamma^{2^m}(xyx^{-2}y^{-1})^{2^l})^{2^{k_0}}, y^{2^{k_1}}
\rangle\\
&\overset{\mathfrak{a}}{\longrightarrow}
\Z\langle x\rangle\oplus\Z\langle y\rangle\oplus\Z\langle \gamma \rangle/
\Z\langle 2^{m+k_0}\gamma+2^{l+k_0}x\rangle\oplus\Z\langle 2^{k_1}y\rangle\\
&\cong H_1(X'). 
\end{align*}
It must be an infinite cyclic group generated by $[\gamma]$. Hence $k_0=k_1=l=0$. 
Thus we have the following: 
\begin{align*}
G(K)
&\cong
\langle
x,y,\gamma,\gamma_0,\gamma_1, \mu
\mid
y=\gamma_1, xyx^{-2}y^{-1}=\gamma_0, (\gamma^{2^m}\gamma_0^{2^l})^{2^{k_0}}, \gamma_1^{2^{k_1}}, \gamma\mu^{-g}
\rangle\\
&\cong
\langle \mu \rangle.
\end{align*}

\vspace{2mm}
\noindent (3)
The polyhedron $X'$ is decomposed as $U\cup X_4\cup V$, 
and $U$ and $V$ are glued with $X_4$ along the boundary components of $X_4$ with length 1 and 5, respectively. 
Then the fundamental group of $X'$ and its abelianization are obtained as follows: 
\begin{align*}
\pi_1(X')
&\cong
\langle
x,y,\gamma,\gamma_0,\gamma_1
\mid
y=\gamma_0, xyx^{-2}y=\gamma_1, (\gamma^{2^m}\gamma_0^{2^l})^{2^{k_0}}, \gamma_1^{2^{k_1}}
\rangle\\
&\cong
\langle
x,y,\gamma
\mid
(\gamma^{2^m}y^{2^l})^{2^{k_0}}, (xyx^{-2}y)^{2^{k_1}}
\rangle\\
&\overset{\mathfrak{a}}{\longrightarrow}
\Z\langle x\rangle\oplus\Z\langle y\rangle\oplus\Z\langle \gamma \rangle/
\Z\langle 2^{m+k_0}\gamma+2^{l+k_0}y\rangle\oplus\Z\langle 2^{k_1}(2y-x)\rangle\\
&\cong H_1(X'). 
\end{align*}
It must be an infinite cyclic group generated by $[\gamma]$. 
Hence $k_0=k_1=l=0$. 
Thus we have the following: 
\begin{align*}
G(K)
&\cong
\langle
x,y,\gamma,\gamma_0,\gamma_1, \mu
\mid
y=\gamma_0, xyx^{-2}y=\gamma_1, (\gamma^{2^m}\gamma_0^{2^l})^{2^{k_0}}, \gamma_1^{2^{k_1}}, \gamma\mu^{-g}
\rangle\\
&\cong
\langle x,y\mid x^2y^n x^{-1}y^{n} \rangle, 
\end{align*}
where $n=2^mg$. 

\vspace{2mm}
\noindent (4)
The polyhedron $X'$ is decomposed as $U\cup X_4\cup V$, 
and $U$ and $V$ are glued with $X_4$ along the boundary components of $X_4$ with length 5 and 1, respectively. 
Then the fundamental group of $X'$ and its abelianization are obtained as follows: 
\begin{align*}
\pi_1(X')
&\cong
\langle
x,y,\gamma,\gamma_0,\gamma_1
\mid
y=\gamma_1, xyx^{-2}y=\gamma_0, (\gamma^{2^m}\gamma_0^{2^l})^{2^{k_0}}, \gamma_1^{2^{k_1}}
\rangle\\
&\cong
\langle
x,y,\gamma
\mid
(\gamma^{2^m}(xyx^{-2}y)^{2^l})^{2^{k_0}}, y^{2^{k_1}}
\rangle\\
&\overset{\mathfrak{a}}{\longrightarrow}
\Z\langle x\rangle\oplus\Z\langle y\rangle\oplus\Z\langle \gamma \rangle/
\Z\langle 2^{m+k_0}\gamma+2^{l+k_0}(2y-x)\rangle\oplus\Z\langle 2^{k_1}y\rangle\\
&\cong H_1(X'). 
\end{align*}
It must be an infinite cyclic group generated by $[\gamma]$. 
Hence $k_0=k_1=l=0$. 
Thus we have the following: 
\begin{align*}
G(K)
&\cong
\langle
x,y,\gamma,\mu
\mid
(\gamma^{2^m}(xyx^{-2}y)^{2^l})^{2^{k_0}}, y^{2^{k_1}}, \gamma\mu^{-g}
\rangle\cong\langle \mu \rangle.
\end{align*}
\end{proof}
\end{case}
\begin{figure}[tbp]
\includegraphics[width=.85\hsize]{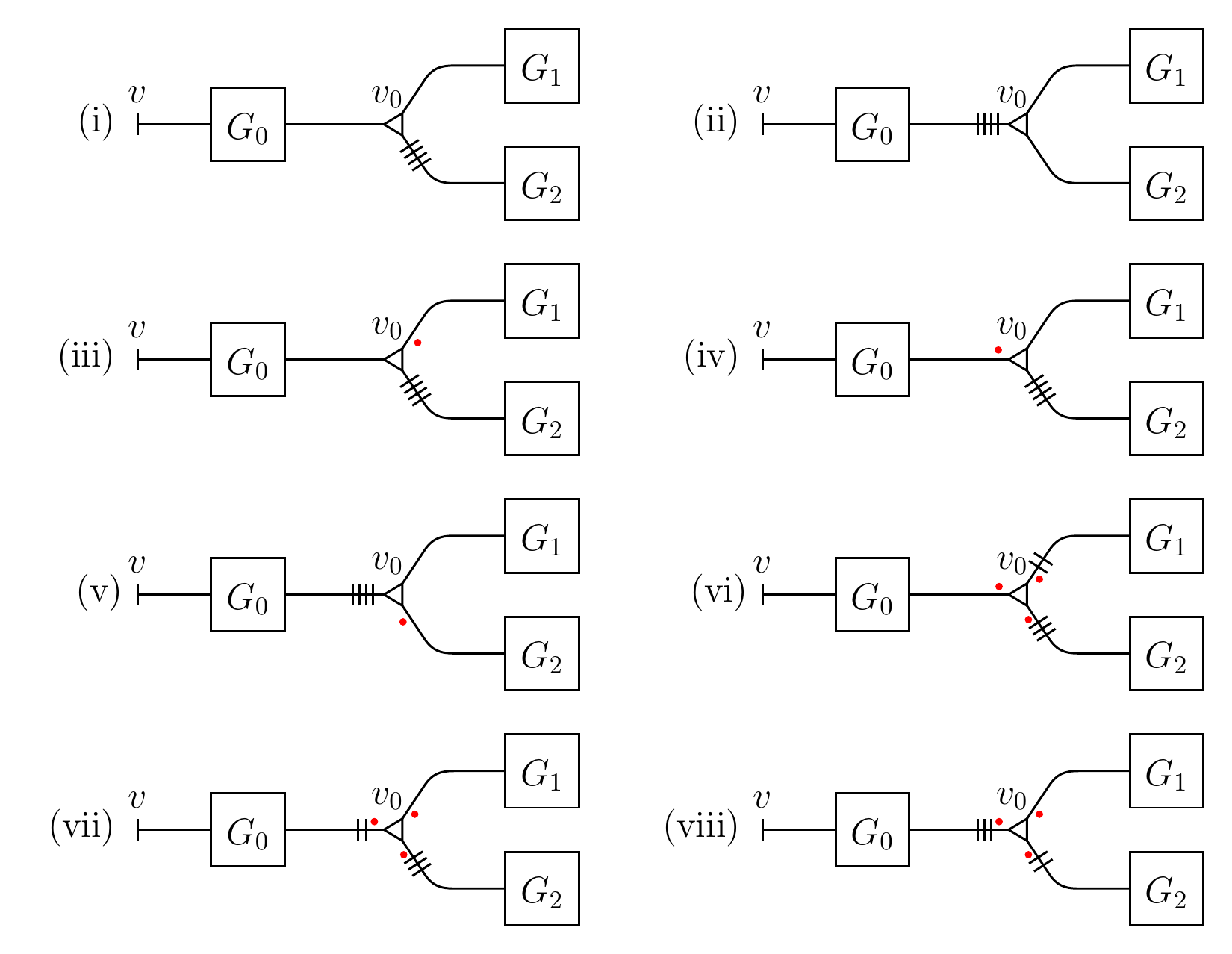}
\caption{
The possible cases of encoding graph $G$ of $X$ 
such that $v_0$ is of type $\vt{X_8}$, $\vt{X_9}$ or $\vt{X_{10}}$. }
\label{fig:c=1cases2}
\end{figure}
\begin{case}
[$v_0$ is of type $\vt{X_8}, \vt{X_9}$ or $\vt{X_{10}}$]
The graph $G$ is one of those shown in Figure~\ref{fig:c=1cases2}. 
Let $G_0$, $G_1$ and $G_2$ be subgraphs of $G$ as indicated in the figure. 
The subpolyhedron $X'$ is decomposed into $U$, $X_{v_0}$, $V_1$ and $V_2$, 
where $U$, $V_1$ and $V_2$ are the subpolyhedron corresponding to $G_0$, $G_1$ and $G_2$, respectively. 
Note that $U$ has two cut ends $\gamma\sqcup\gamma_0$, 
and also note that $V_i$ has one cut end $\gamma_i$ for $i\in\{1,2\}$. 
Then we have $\pi_1(V_i)\cong\langle \gamma_i\mid\gamma_i^{2^{k_i}} \rangle$ 
by Lemma~\ref{lem:pi_1_of_subpolyhedron1} for some $k_i\in\Z_{\geq0}$. 
We also have $H_1(X')\cong\Z\langle\gamma\rangle$ by Lemma~\ref{lem:H_1_of_subpolyhedron}. 
Then we can apply Lemma~\ref{lem:pi_1_of_subpolyhedron2} to $U$, 
and we have $\pi_1(U)\cong\langle \gamma,\gamma_0\mid (\gamma^{2^m}\gamma_0^{2^l})^{2^{k_0}} \rangle$ 
for some $k_0,l,m\in\Z_{\geq0}$. 
\begin{lemma}
\label{lem:X8910}
The following hold. 
\begin{enumerate}
\item
If $G$ is shown in Figure~\ref{fig:c=1cases2}-(i), 
then $G(K)$ is an infinite cyclic group. 
\item
Figure~\ref{fig:c=1cases2}-(ii) does not encodes a shadow of any $2$-knot. 
\item
Figure~\ref{fig:c=1cases2}-(iii) does not encodes a shadow of any $2$-knot. 
\item
If $G$ is shown in Figure~\ref{fig:c=1cases2}-(iv), 
then $G(K)$ is an infinite cyclic group. 
\item
Figure~\ref{fig:c=1cases2}-(v) does not encodes a shadow of any $2$-knot. 
\item
Figure~\ref{fig:c=1cases2}-(vi) does not encodes a shadow of any $2$-knot. 
\item
Figure~\ref{fig:c=1cases2}-(vii) does not encodes a shadow of any $2$-knot. 
\item
Figure~\ref{fig:c=1cases2}-(viii) does not encodes a shadow of any $2$-knot. 
\end{enumerate}
\end{lemma}
\begin{proof}

(1) 
The polyhedron $X'$ is decomposed as $U\cup X_8\cup V_1\cup V_2$, 
and $U$, $V_1$ and $V_2$ are glued with $X_8$ along the boundary components of $X_8$ with length 1, 1 and 4, respectively. 
The fundamental group of $X'$ and its abelianization are obtained as follows: 
\begin{align*}
\pi_1(X')
&\cong
\lrangle
{x,y,\gamma,\gamma_0,\gamma_1, \gamma_2}
{\begin{array}{c}
(\gamma^{2^m}\gamma_0^{2^l})^{2^{k_0}}, \gamma_1^{2^{k_1}}, \gamma_2^{2^{k_2}},\\
\gamma_0=x, \gamma_1=y, \gamma_2=xyx^{-1}y^{-1}
 \end{array}}\\
&\cong
\langle
x,y,\gamma
\mid
(\gamma^{2^m}x^{2^l})^{2^{k_0}}, y^{2^{k_1}}, (xyx^{-1}y^{-1})^{2^{k_2}}
\rangle\\
&\overset{\mathfrak{a}}{\longrightarrow}
\Z\langle x\rangle\oplus\Z\langle y\rangle\oplus\Z\langle \gamma \rangle/
\Z\langle 2^{m+k_0}\gamma+2^{l+k_0}x\rangle\oplus\Z\langle 2^{k_1}y\rangle\\
&\cong H_1(X'). 
\end{align*}
It must be an infinite cyclic group generated by $[\gamma]$. 
Hence $k_0=k_1=l=0$. 
Thus we have the following: 
\begin{align*}
G(K)
&\cong
\langle
x,y,\gamma,\mu
\mid
(\gamma^{2^m}x^{2^l})^{2^{k_0}}, y^{2^{k_1}}, (xyx^{-1}y^{-1})^{2^{k_2}}, \gamma\mu^{-g}
\rangle\cong
\langle
\mu
\rangle. 
\end{align*}

\vspace{2mm}
\noindent (2)
Suppose $G$ is as shown in Figure~\ref{fig:c=1cases2}-(ii). 
The polyhedron $X'$ is decomposed as $U\cup X_8\cup V_1\cup V_2$, 
and $U$, $V_1$ and $V_2$ are glued with $X_8$ along the boundary components of $X_8$ with length 4, 1 and 1, respectively. 
Then the fundamental group of $X'$ and its abelianization are obtained as follows: 
\begin{align*}
\pi_1(X')
&\cong
\lrangle
{x,y,\gamma,\gamma_0,\gamma_1, \gamma_2}
{\begin{array}{c}
(\gamma^{2^m}\gamma_0^{2^l})^{2^{k_0}}, \gamma_1^{2^{k_1}}, \gamma_2^{2^{k_2}},\\
\gamma_0=xyx^{-1}y^{-1}, \gamma_1=x, \gamma_2=y
 \end{array}}\\
&\cong
\langle
x,y,\gamma
\mid
(\gamma^{2^m}(xyx^{-1}y^{-1})^{2^l})^{2^{k_0}}, x^{2^{k_1}}, y^{2^{k_2}}
\rangle\\
&\overset{\mathfrak{a}}{\longrightarrow}
\Z\langle x\rangle\oplus\Z\langle y\rangle\oplus\Z\langle \gamma \rangle/
\Z\langle 2^{m+k_0}\gamma\rangle\oplus\Z\langle 2^{k_1}x\rangle\oplus\Z\langle 2^{k_2}y\rangle\\
&\cong H_1(X'). 
\end{align*}
It must be an infinite cyclic group generated by $[\gamma]$, 
which is impossible. 

\vspace{2mm}
\noindent (3)
Suppose $G$ is as shown in Figure~\ref{fig:c=1cases2}-(iii). 
The polyhedron $X'$ is decomposed as $U\cup X_9\cup V_1\cup V_2$. 
One of the boundary components of $X_9$ has length 4 
and the other two have 1. 
Note that, however, $X_9$ does not have a symmetry such as $X_8$. 
The boundary components of $X_9$ are represented by words $xyxy^{-1}$, $x$ and $y$. 
Here $U$, $V_1$ and $V_2$ are glued with $X_9$ 
along the boundary components of $X_9$ corresponding to $x$, $y$, $xyxy^{-1}$, respectively. 
Then the fundamental group of $X'$ and its abelianization are obtained as follows: 
\begin{align*}
\pi_1(X')
&\cong
\lrangle
{x,y,\gamma,\gamma_0,\gamma_1, \gamma_2}
{\begin{array}{c}
(\gamma^{2^m}\gamma_0^{2^l})^{2^{k_0}}, \gamma_1^{2^{k_1}}, \gamma_2^{2^{k_2}},\\
\gamma_0=x, \gamma_1=y, \gamma_2=xyxy^{-1}
 \end{array}}\\
&\cong
\langle
x,y,\gamma
\mid
(\gamma^{2^m}x^{2^l})^{2^{k_0}}, y^{2^{k_1}}, (xyxy^{-1})^{2^{k_2}}
\rangle\\
&\overset{\mathfrak{a}}{\longrightarrow}
\Z\langle x\rangle\oplus\Z\langle y\rangle\oplus\Z\langle \gamma \rangle/
\Z\langle 2^{m+k_0}\gamma+2^{l+k_0}x\rangle\oplus\Z\langle 2^{k_1}y\rangle\oplus\Z\langle 2^{k_2+1}x\rangle\\
&\cong H_1(X'). 
\end{align*}
It must be an infinite cyclic group generated by $[\gamma]$, 
which is impossible. 

\vspace{2mm}
\noindent (4)
The polyhedron $X'$ is decomposed as $U\cup X_9\cup V_1\cup V_2$, 
and $U$, $V_1$ and $V_2$ are glued with $X_9$ 
along the boundary components of $X_9$ corresponding to $y$, $x$, $xyxy^{-1}$, respectively. 
Then the fundamental group of $X'$ and its abelianization are obtained as follows: 
\begin{align*}
\pi_1(X')
&\cong
\lrangle
{x,y,\gamma,\gamma_0,\gamma_1, \gamma_2}
{\begin{array}{c}
(\gamma^{2^m}\gamma_0^{2^l})^{2^{k_0}}, \gamma_1^{2^{k_1}}, \gamma_2^{2^{k_2}},\\
\gamma_0=y, \gamma_1=x, \gamma_2=xyxy^{-1}
 \end{array}}\\
&\cong
\langle
x,y,\gamma
\mid
(\gamma^{2^m}y^{2^l})^{2^{k_0}}, x^{2^{k_1}}, (xyxy^{-1})^{2^{k_2}}
\rangle\\
&\overset{\mathfrak{a}}{\longrightarrow}
\Z\langle x\rangle\oplus\Z\langle y\rangle\oplus\Z\langle \gamma \rangle/
\Z\langle 2^{m+k_0}\gamma+2^{l+k_0}y\rangle\oplus\Z\langle 2^{k_1}x\rangle\oplus\Z\langle 2^{k_2+1}x\rangle\\
&\cong H_1(X'). 
\end{align*}
It must be an infinite cyclic group generated by $[\gamma]$. 
Hence $k_0=k_1=l=0$. 
Thus we have the following: 
\begin{align*}
G(K)
&\cong
\langle
x,y,\gamma,\mu
\mid
(\gamma^{2^m}y^{2^l})^{2^{k_0}}, x^{2^{k_1}}, (xyxy^{-1})^{2^{k_2}}, \gamma\mu^{-g}
\rangle\cong
\langle
\mu
\rangle. 
\end{align*}

\vspace{2mm}
\noindent (5)
Suppose $G$ is as shown in Figure~\ref{fig:c=1cases2}-(v). 
The polyhedron $X'$ is decomposed as $U\cup X_9\cup V_1\cup V_2$, 
and $U$, $V_1$ and $V_2$ are glued with $X_9$ 
along the boundary components of $X_9$ corresponding to $xyxy^{-1}$, $x$, $y$, respectively. 
Then the fundamental group of $X'$ and its abelianization are obtained as follows: 
\begin{align*}
\pi_1(X')
&\cong
\lrangle
{x,y,\gamma,\gamma_0,\gamma_1, \gamma_2}
{\begin{array}{c}
(\gamma^{2^m}\gamma_0^{2^l})^{2^{k_0}}, \gamma_1^{2^{k_1}}, \gamma_2^{2^{k_2}},\\
\gamma_0=xyxy^{-1}, \gamma_1=x, \gamma_2=y
 \end{array}}\\
&\cong
\langle
x,y,\gamma
\mid
(\gamma^{2^m}(xyxy^{-1})^{2^l})^{2^{k_0}}, x^{2^{k_1}}, y^{2^{k_2}}
\rangle\\
&\overset{\mathfrak{a}}{\longrightarrow}
\Z\langle x\rangle\oplus\Z\langle y\rangle\oplus\Z\langle \gamma \rangle/
\Z\langle 2^{m+k_0}\gamma+2^{l+k_0+1}x\rangle\oplus\Z\langle 2^{k_1}x\rangle\oplus\Z\langle 2^{k_2}y\rangle\\
&\cong H_1(X'). 
\end{align*}
It must be an infinite cyclic group generated by $[\gamma]$, 
which is impossible. 

\vspace{2mm}
\noindent (6)
Suppose $G$ is as shown in Figure~\ref{fig:c=1cases2}-(vi). 
The polyhedron $X'$ is decomposed as $U\cup X_{10}\cup V_1\cup V_2$, 
and $U$, $V_1$ and $V_2$ are glued with $X_{10}$ 
along the boundary components of $X_{10}$ with length 1, 2 and 3, respectively. 
Then the fundamental group of $X'$ and its abelianization are obtained as follows: 
\begin{align*}
\pi_1(X')
&\cong
\lrangle
{x,y,\gamma,\gamma_0,\gamma_1, \gamma_2}
{\begin{array}{c}
(\gamma^{2^m}\gamma_0^{2^l})^{2^{k_0}}, \gamma_1^{2^{k_1}}, \gamma_2^{2^{k_2}},\\
\gamma_0=x, \gamma_1=xy, \gamma_2=xy^{-2}
 \end{array}}\\
&\cong
\langle
x,y,\gamma
\mid
(\gamma^{2^m}x^{2^l})^{2^{k_0}}, (xy)^{2^{k_1}}, (xy^{-2})^{2^{k_2}}
\rangle\\
&\overset{\mathfrak{a}}{\longrightarrow}
\frac
{\Z\langle x\rangle\oplus\Z\langle y\rangle\oplus\Z\langle \gamma \rangle}
{\Z\langle 2^{m+k_0}\gamma+2^{l+k_0}x\rangle\oplus\Z\langle 2^{k_1}(x+y)\rangle\oplus\Z\langle 2^{k_2}(x-2y)\rangle}\\
&\cong H_1(X'). 
\end{align*}
It must be an infinite cyclic group generated by $[\gamma]$, 
which is impossible. 

\vspace{2mm}
\noindent (7)
Suppose $G$ is as shown in Figure~\ref{fig:c=1cases2}-(vii). 
The polyhedron $X'$ is decomposed as $U\cup X_{10}\cup V_1\cup V_2$, 
and $U$, $V_1$ and $V_2$ are glued with $X_{10}$ 
along the boundary components of $X_{10}$ with length 2, 1 and 3, respectively. 
Then the fundamental group of $X'$ and its abelianization are obtained as follows: 
\begin{align*}
\pi_1(X')
&\cong
\lrangle
{x,y,\gamma,\gamma_0,\gamma_1, \gamma_2}
{\begin{array}{c}
(\gamma^{2^m}\gamma_0^{2^l})^{2^{k_0}}, \gamma_1^{2^{k_1}}, \gamma_2^{2^{k_2}},\\
\gamma_0=xy, \gamma_1=x, \gamma_2=xy^{-2}
 \end{array}}\\
&\cong
\langle
x,y,\gamma
\mid
(\gamma^{2^m}(xy)^{2^l})^{2^{k_0}}, x^{2^{k_1}}, (xy^{-2})^{2^{k_2}}
\rangle\\
&\overset{\mathfrak{a}}{\longrightarrow}
\frac
{\Z\langle x\rangle\oplus\Z\langle y\rangle\oplus\Z\langle \gamma \rangle}
{\Z\langle 2^{m+k_0}\gamma+2^{l+k_0}(x+y)\rangle\oplus\Z\langle 2^{k_1}x\rangle\oplus\Z\langle 2^{k_2}(x-2y)\rangle}\\
&\cong H_1(X'). 
\end{align*}
It must be an infinite cyclic group generated by $[\gamma]$, 
which is impossible. 

\vspace{2mm}
\noindent (8)
Suppose $G$ is as shown in Figure~\ref{fig:c=1cases2}-(viii). 
The polyhedron $X'$ is decomposed as $U\cup X_{10}\cup V_1\cup V_2$, 
and $U$, $V_1$ and $V_2$ are glued with $X_{10}$ 
along the boundary components of $X_{10}$ with length 3, 1 and 2, respectively. 
Then the fundamental group of $X'$ and its abelianization are obtained as follows: 
\begin{align*}
\pi_1(X')
&\cong
\lrangle
{x,y,\gamma,\gamma_0,\gamma_1, \gamma_2}
{\begin{array}{c}
(\gamma^{2^m}\gamma_0^{2^l})^{2^{k_0}}, \gamma_1^{2^{k_1}}, \gamma_2^{2^{k_2}},\\
\gamma_0=xy^{-2}, \gamma_1=x, \gamma_2=xy
 \end{array}}\\
&\cong
\langle
x,y,\gamma
\mid
(\gamma^{2^m}(xy^{-2})^{2^l})^{2^{k_0}}, x^{2^{k_1}}, (xy)^{2^{k_2}}
\rangle\\
&\overset{\mathfrak{a}}{\longrightarrow}
\frac
{\Z\langle x\rangle\oplus\Z\langle y\rangle\oplus\Z\langle \gamma \rangle}
{\Z\langle 2^{m+k_0}\gamma+2^{l+k_0}(x-2y)\rangle\oplus\Z\langle 2^{k_1}x\rangle\oplus\Z\langle 2^{k_2}(x+y)\rangle}\\
&\cong H_1(X'). 
\end{align*}
It must be an infinite cyclic group generated by $[\gamma]$, 
which is impossible. 
\end{proof}

\end{case}
\begin{figure}[tbp]
\includegraphics[width=.85\hsize]{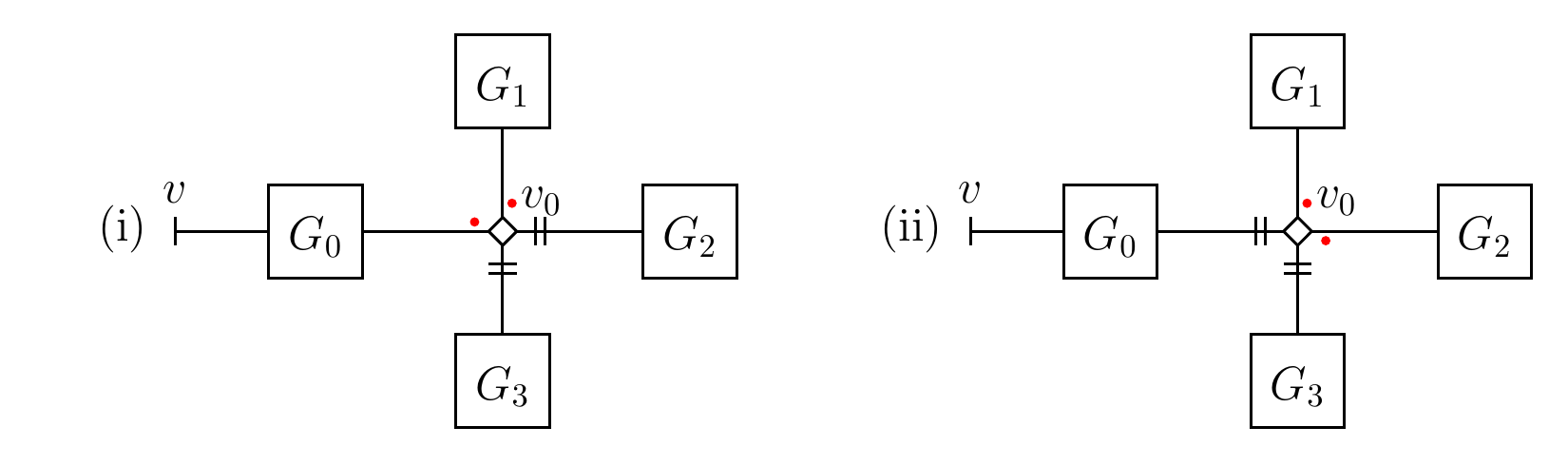}
\caption{
The possible cases of encoding graph $G$ of $X$ 
such that $v_0$ is of type $\vt{X_{11}}$. }
\label{fig:c=1cases3}
\end{figure}
\begin{case}
[$v_0$ is of type $\vt{X_{11}}$]
The graph $G$ is one of those shown in Figures~\ref{fig:c=1cases3}-(i) or -(ii). 
Let $G_0$, $G_1$, $G_2$ and $G_3$ be subgraphs of $G$ as indicated in the figure. 
The subpolyhedron $X'$ is decomposed into $U$, $X_{v_0}$, $V_1$, $V_3$ and $V_3$, 
where $U$, $V_1$, $V_3$ and $V_3$ are the subpolyhedron corresponding to $G_0$, $G_1$, $G_2$ and $G_3$. 
Note that $U$ has two cut ends $\gamma\sqcup\gamma_0$ as a subpolyhedron of $X$, 
and also note that $V_i$ has one cut end $\gamma_i$ for $i\in\{1,2,3\}$. 
Then we have $\pi_1(V_i)\cong\langle \gamma_1\mid\gamma_1^{2^{k_i}} \rangle$ 
by Lemma~\ref{lem:pi_1_of_subpolyhedron1} for some $k_i\in\Z_{\geq0}$. 
We also have $H_1(X')\cong\Z\langle\gamma\rangle$ by Lemma~\ref{lem:H_1_of_subpolyhedron}. 
Then we can apply Lemma~\ref{lem:pi_1_of_subpolyhedron2} to $U$, 
and we have $\pi_1(U)\cong\langle \gamma,\gamma_0\mid (\gamma^{2^m}\gamma_0^{2^l})^{2^{k_0}} \rangle$ 
for some $k_0,l,m\in\Z_{\geq0}$. 
\begin{lemma}
\label{lem:X11}
The following hold. 
\begin{enumerate}
\item
Figure~\ref{fig:c=1cases3}-(i) does not encodes a shadow of any $2$-knot. 
\item
Figure~\ref{fig:c=1cases3}-(ii) does not encodes a shadow of any $2$-knot. 
\end{enumerate}
\end{lemma}
\begin{proof}
(1) 
Suppose $G$ is as shown in Figure~\ref{fig:c=1cases3}-(xiii). 
The polyhedron $X'$ is decomposed as $U\cup X_{11}\cup V_1\cup V_2\cup V_3$, 
and $U$, $V_1$, $V_2$ and $V_3$ are glued with $X_{11}$ 
along the boundary components of $X_{11}$ with length 1, 1, 2 and 2, respectively. 
The fundamental group of $X'$ and its abelianization are obtained as follows: 
\begin{align*}
\pi_1(X')
&\cong
\lrangle
{x,y,\gamma,\gamma_0,\gamma_1, \gamma_2, \gamma_3}
{\begin{array}{c}
(\gamma^{2^m}\gamma_0^{2^l})^{2^{k_0}}, \gamma_1^{2^{k_1}}, \gamma_2^{2^{k_2}}, \gamma_3^{2^{k_3}},\\
\gamma_0=x, \gamma_1=y, \gamma_2=xy, \gamma_3=xy^{-1}
 \end{array}}\\
&\cong
\langle
x,y,\gamma
\mid
(\gamma^{2^m}x^{2^l})^{2^{k_0}}, y^{2^{k_1}}, (xy)^{2^{k_2}}, (xy^{-1})^{2^{k_3}}
\rangle\\
&\overset{\mathfrak{a}}{\longrightarrow}
\frac
{\Z\langle x\rangle\oplus\Z\langle y\rangle\oplus\Z\langle \gamma \rangle}
{\Z\langle 2^{m+k_0}\gamma+2^{l+k_0}x\rangle\oplus\Z\langle 2^{k_1}y\rangle
\oplus\Z\langle 2^{k_2}(x+y)\rangle\oplus\Z\langle 2^{k_3}(x-y)\rangle}\\
&\cong H_1(X'). 
\end{align*}
It must be an infinite cyclic group generated by $[\gamma]$, 
which is impossible. 

\vspace{2mm}
\noindent (2) 
Suppose $G$ is as shown in Figure~\ref{fig:c=1cases3}-(xiv). 
The polyhedron $X'$ is decomposed as $U\cup X_{11}\cup V_1\cup V_2\cup V_3$, 
and $U$, $V_1$, $V_2$ and $V_3$ are glued with $X_{11}$ 
along the boundary components of $X_{11}$ with length 2, 1, 1 and 2, respectively. 
Then the fundamental group of $X'$ and its abelianization are obtained as follows: 
\begin{align*}
\pi_1(X')
&\cong
\lrangle
{x,y,\gamma,\gamma_0,\gamma_1, \gamma_2, \gamma_3}
{\begin{array}{c}
(\gamma^{2^m}\gamma_0^{2^l})^{2^{k_0}}, \gamma_1^{2^{k_1}}, \gamma_2^{2^{k_2}}, \gamma_3^{2^{k_3}},\\
\gamma_0=xy, \gamma_1=x, \gamma_2=y, \gamma_3=xy^{-1}
 \end{array}}\\
&\cong
\langle
x,y,\gamma
\mid
(\gamma^{2^m}(xy)^{2^l})^{2^{k_0}}, x^{2^{k_1}}, y^{2^{k_2}}, (xy^{-1})^{2^{k_3}}
\rangle\\
&\overset{\mathfrak{a}}{\longrightarrow}
\frac
{\Z\langle x\rangle\oplus\Z\langle y\rangle\oplus\Z\langle \gamma \rangle}
{\Z\langle 2^{m+k_0}\gamma+2^{l+k_0}(x+y)\rangle\oplus\Z\langle 2^{k_1}x\rangle
\oplus\Z\langle 2^{k_2}y\rangle\oplus\Z\langle 2^{k_3}(x-y)\rangle}\\
&\cong H_1(X'). 
\end{align*}
It must be an infinite cyclic group generated by $[\gamma]$, 
which is impossible. 
\end{proof}
\end{case}

\section{Classification of $2$-knots with complexity one}
\label{sec:Classification of 2-knots with complexity one}
\subsection{Lemmas on decorated graphs}
We now define a {\it decoration} of an edge $e$ of an encoding graph $G$ as a half-integer 
such that it is an integer if and only if the number of red dots appended to $e$ is even (actually, zero or two). 
If every edge of $G$ is assigned with a decoration, $G$ is called a {\it decorated graph}. 
A decoration corresponds to a gleam, and a decorated tree encodes a shadowed polyhedron. 

We can easily describe how a decorated graph $G$ changes by 
adding a compressing disk and a connected-sum reduction. 
See Figure~\ref{fig:comp_disk_conn_sum}. 
\begin{figure}[tbp]
\includegraphics[width=110mm]{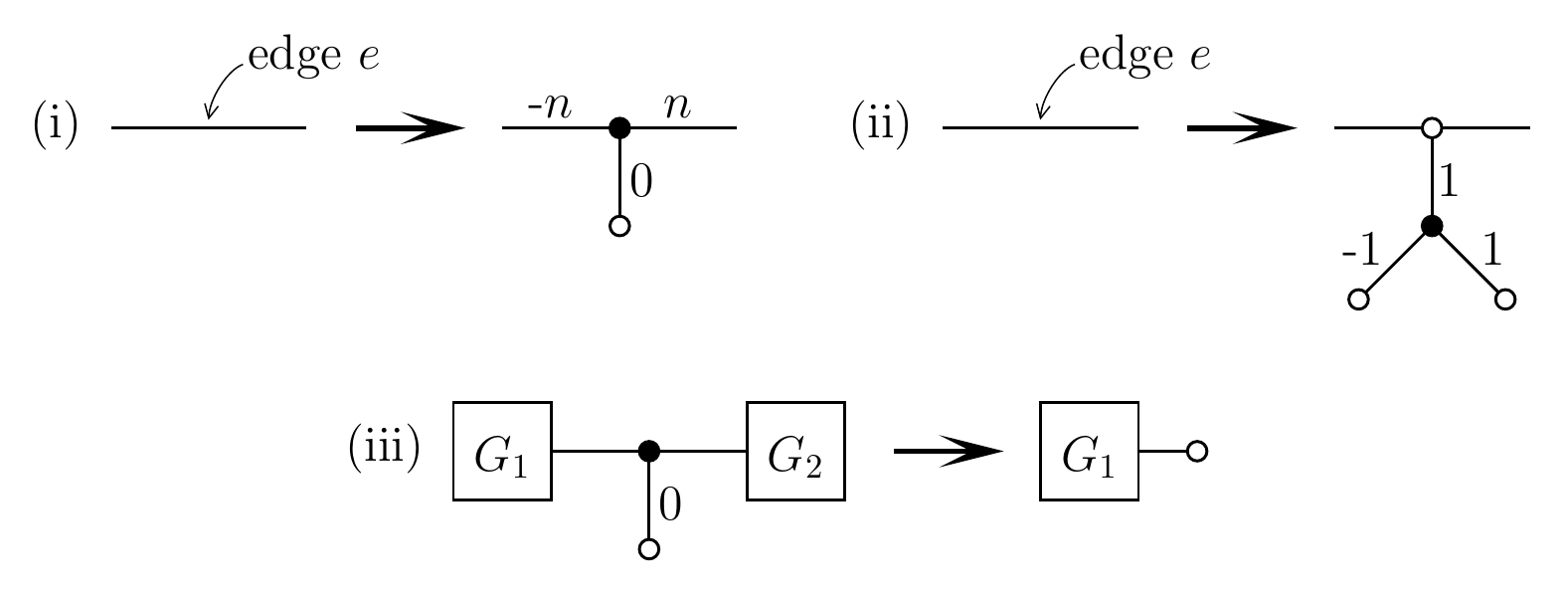}
\caption{(i) adding a horizontal compressing disk, (ii) adding a vertical compressing disk, and (iii) connected-sum reduction. }
\label{fig:comp_disk_conn_sum}
\end{figure}
If a lift of an edge $e$ of $G$ has a horizontal (resp. vertical) compressing disk, 
we can replace the edge $e$ as shown in Figure~\ref{fig:comp_disk_conn_sum}-(i) (resp. -(ii)). 
If a decorated graph is as shown in the left of Figure~\ref{fig:comp_disk_conn_sum}-(iii) 
and if the subpolyhedron corresponding to the subgraph $G_1$ contains $K$, 
we can adopt a decorated graph shown in the right of the figure. 

In this subsection, we provide some modifications of shadows and decorated graphs not changing a $2$-knot. 


\begin{figure}[tbp]
\includegraphics[width=80mm]{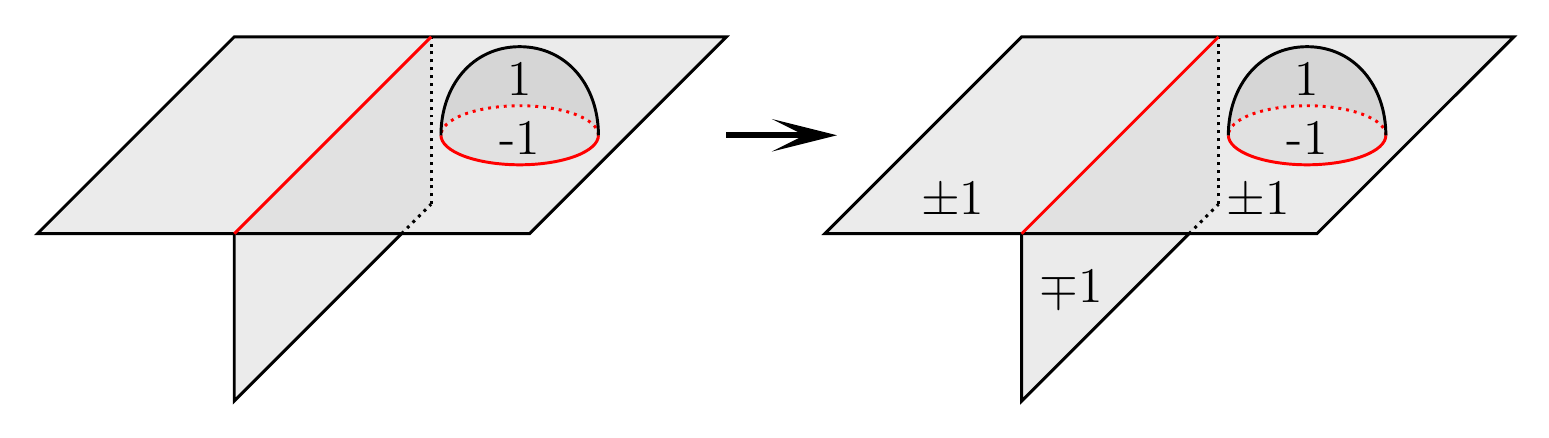}
\caption{A move on a shadow: if a region $R$ has a vertical compressing disk, 
then we can modify the gleams of $R,R'$ and $R''$ as in the figure, where $R'$ and $R''$ are regions adjacent to a common triple line with $R$.}
\label{fig:gleam_shift}
\end{figure}
\begin{lemma}
\label{lem:gleam_shift}
Suppose that $X'$ has a part as shown in the left of Figure~\ref{fig:gleam_shift}. 
Then the move shown in the figure and its inverse modify $X$ to another shadow of $K$. 
\end{lemma}
See \cite{KMN18} for the proof of the above. 
\begin{remark}
The regions appeared in Figure~\ref{fig:gleam_shift} must not a part of $K$. 
\end{remark}

We next introduce eight moves on decorated graphs as shown in Figure~\ref{fig:graph_moves}; 
moves-(a), -(b), -(c), -(d), -(e), -(f), -(g) and -(h). 
Note that the decoration $r$ in a move-(g) is not $\pm\frac12$. 
\begin{lemma}
\label{lem:graph_moves}
Let $G$ be a decorated graph of $X$ and $G'$ be the subgraph corresponding to $X'$. 
Then the moves shown in Figure~\ref{fig:graph_moves} that is performed on $G'$ 
modify $G$ to another decorated graph encoding a shadow of $K$. 
\end{lemma}
\begin{figure}[tbp]
\includegraphics[width=1\hsize]{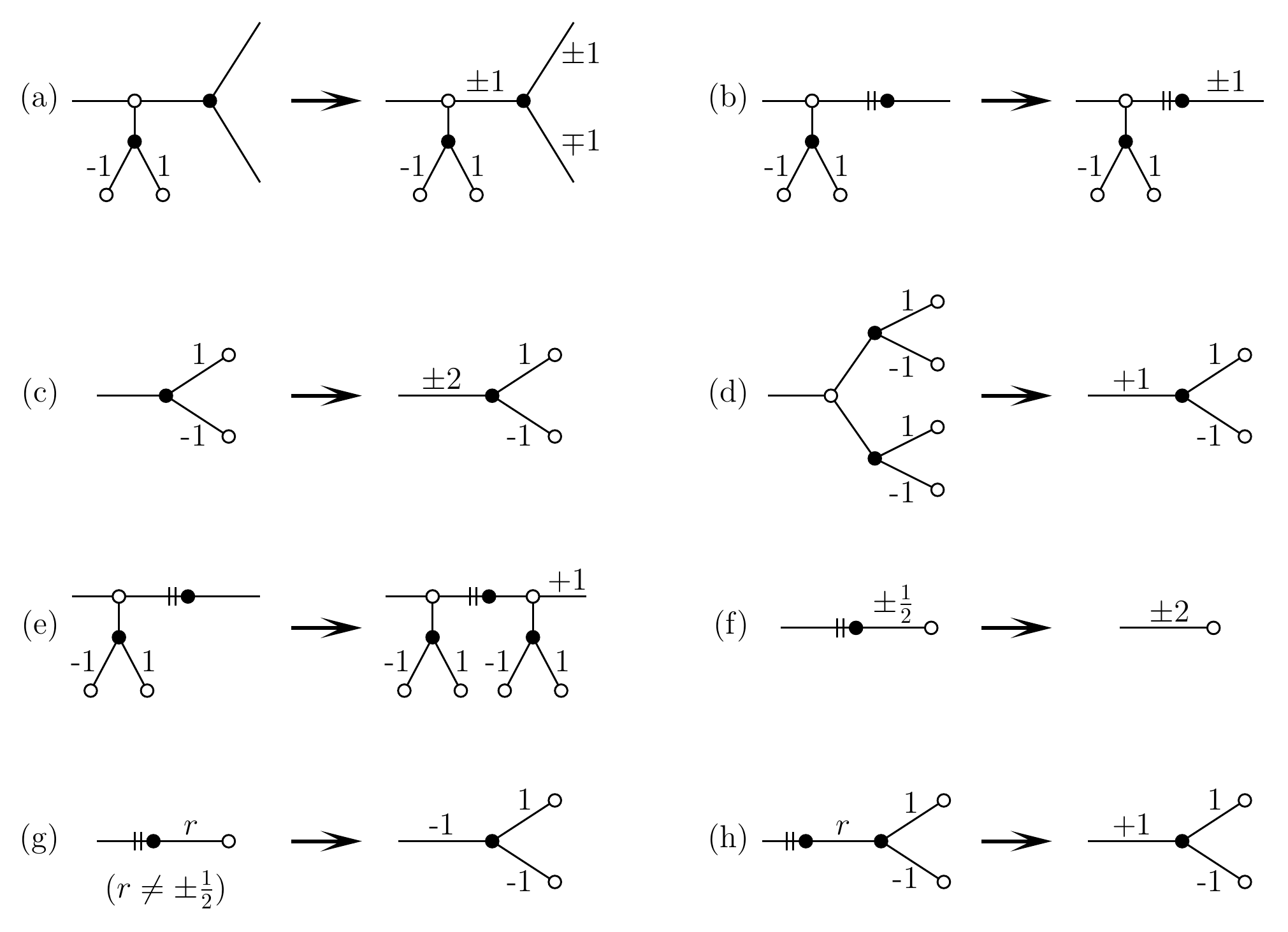}
\caption{Eight moves on decorated graphs. The decoration $r$ in (g) is not $\pm\frac12$. }
\label{fig:graph_moves}
\end{figure}
\begin{proof}
Moves-(a) and -(b) are obtained by a move in Figure~\ref{fig:gleam_shift}. 

A move-(c) is explained in \cite[Figure 34-(7)]{Mar11}. 

A move-(d) is a obtained by 
a move-(a), a connected-sum reduction, and a YV-move. 

A move-(e) is a kind of {\it propagation principle} \cite{KMN18}: 
if two of the three regions adjacent to a triple line have vertical compressing disks, 
then the other also has. 

A move-(f) is explained in \cite[Figure 34-(4)]{Mar11}. 

\begin{figure}[tbp]
\includegraphics[width=1\hsize]{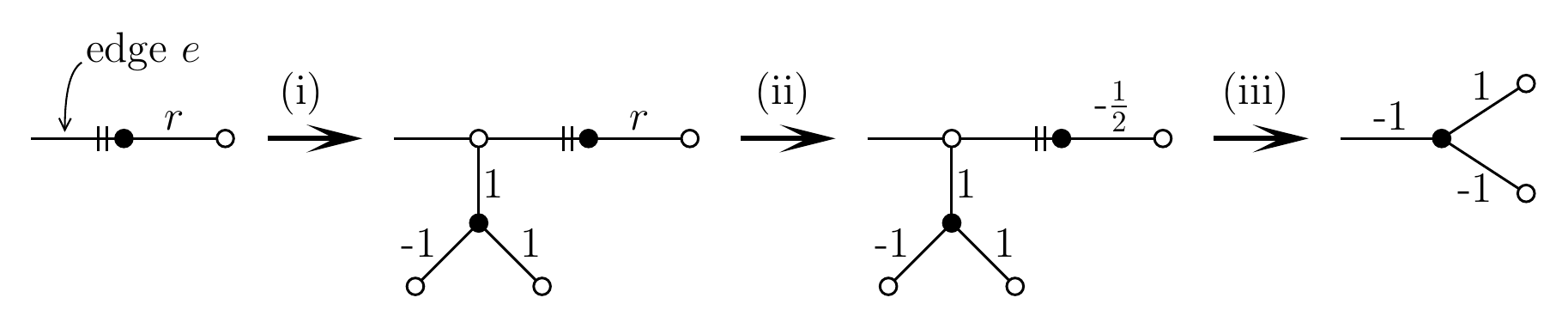}
\caption{The proof of a move-(g). }
\label{fig:graph_lemma_leaf2}
\end{figure}
A move-(g) is explained in Figure~\ref{fig:graph_lemma_leaf2}. 
Let $\pi:M_X\to X$ be a natural projection, where $M_X=\Nbd(X;S^4)$. 
Then the preimage of the subpolyhedron corresponding to the leftmost graph of Figure~\ref{fig:graph_lemma_leaf2} by 
$\pi|_{\partial M_X}$ is homeomorphic to the complement of the $(2,2r)$-torus knot in $S^3$, see \cite[Fig. 11]{IK17}. 
Recall that $r\ne\frac12$. 
Let $e$ be an edge as indicated in the figure, 
then a lift of $e$ has a vertical compressing disk by Property~P and Property~R. 
Hence we can add a vertical compressing disk as in Figure~\ref{fig:graph_lemma_leaf2}-(i). 
The move in Figure~\ref{fig:graph_lemma_leaf2}-(ii) 
is obtained by performing a move-(b) as many times as necessary. 
The move in Figure~\ref{fig:graph_lemma_leaf2}-(iii) 
is done by a move-(f) and a YV-move. 

\begin{figure}[tbp]
\includegraphics[width=1\hsize]{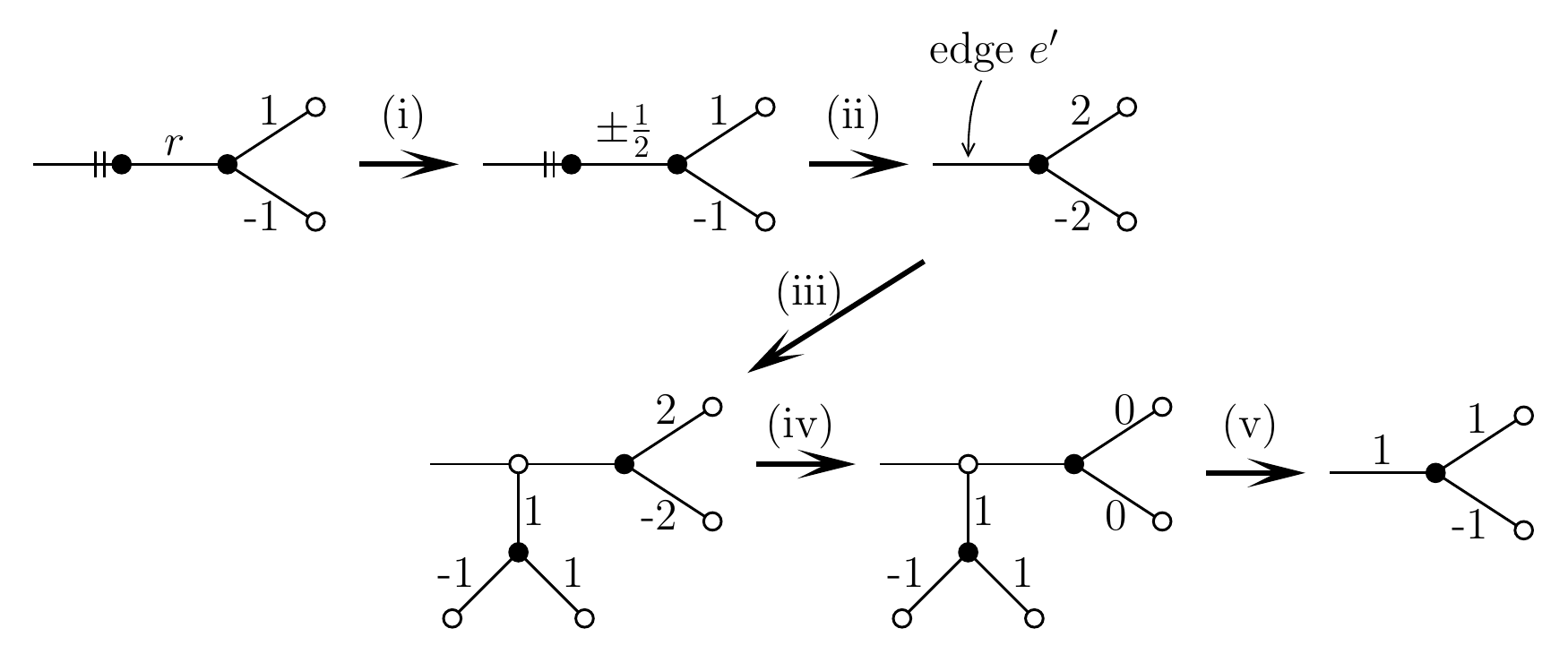}
\caption{The proof of a move-(h). }
\label{fig:graph_lemma_leaf3}
\end{figure}
A move-(h) is explained in Figure~\ref{fig:graph_lemma_leaf3}. 
The move in Figure~\ref{fig:graph_lemma_leaf3}-(i) is obtained by performing 
moves-(c) as many times as necessary. 
The move in Figure~\ref{fig:graph_lemma_leaf3}-(ii) can be done by using \cite[Figure 34-(3)]{Mar11}. 
Here we need the following claim:
\begin{claim}
\label{clm:comp_disk}
A lift of $e'$ has a vertical compressing disk, where $e'$ is an edge indicated in the figure. 
\end{claim}
\begin{proof}[Proof of Claim~\ref{clm:comp_disk}]
Let $\pi:M_X\to X$ be a natural projection. 
The preimage of a subpolyhedron homeomorphic to $Y_{111}$ by $\pi|_{\partial M_X}$ is homeomorphic to $P\times S^1$. 
Hence the subgraph after the move in Figure~\mbox{\ref{fig:graph_lemma_leaf2}-(ii)} corresponds to a $3$-manifold 
homeomorphic to the Seifert fibered space $(P;(2,1),(2,-1))$, which has one torus boundary. 
The Dehn filling of this manifold along the $(p,q)$-slope is $(P;(2,1),(2,-1),(p,q))$. 
Note that the slope with $(p,q)=(p,1)$ is sent to a lift of the edge $e'$ by $\pi$ injectively, 
and the slope with $(p,q)=(1,0)$ is sent to one point by $\pi$. 
The $3$-manifold $(P;(2,1),(2,-1),(p,q))$ is not homeomorphic to $\#_h(S^1\times S^2)$ for any $h\in\Z_{\geq0}$ unless $q=0$. 
If $q=0$, $(P;(2,1),(2,-1),(p,0))$ is homeomorphic to $S^1\times S^2$. 
It follows that a lift of $e'$ has a vertical compressing disk. 
\end{proof}
We then continue the proof for the move-(h). 
The move of Figure~\ref{fig:graph_lemma_leaf3}-(iii) is the addition of a compressing disk of a lift of $e'$. 
The move in Figure~\ref{fig:graph_lemma_leaf3}-(iv) is done by applying move-(a) twice. 
The move in Figure~\ref{fig:graph_lemma_leaf3}-(v) is 
is done by a connected-sum reduction and a YV-move. 
\end{proof}

\begin{lemma}
\label{lem:graph_lemma8}
Let $V$ be a subpolyhedron of $X$ with a single cut end $\gamma$ and $c(V)=0$. 
Let $G$ be a graph encoding $V$ and $v_\gamma$ be the vertex of type~$\vt{B}$ corresponding to $\gamma$. 
Suppose that $G$ has a vertex $v$ of type~$\vt{D}$ that is adjacent to a vertex $v'$ of type~$\vt{Y{}\!_{111}}$. 
Let $R$ be the disk region of $X$ corresponding to $v$. 
If $G$ has another vertex of type~$\vt{Y{}\!_{111}}$ between $v_\gamma$ and $v'$, that is, 
if $G$ is as shown in Figure~\ref{fig:graph_lemma8}, 
then $\gl(R)=0$. 
\end{lemma}
\begin{figure}[tbp]
\includegraphics[width=80mm]{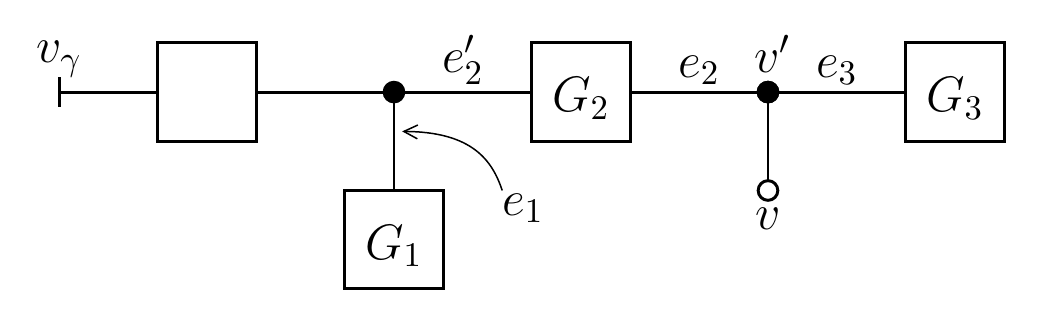}
\caption{An encoding graph restricting the gleam of a disk region. }
\label{fig:graph_lemma8}
\end{figure}
\begin{proof}
We give an orientation to $R$ arbitrarily, 
and we define edges $e_1$, $e_2$, $e'_2$, $e_3$ 
and subgraphs $G_1$, $G_2$, $G_3$ of $G$ as indicated in Figure~\ref{fig:graph_lemma8}. 
Let $V_1$, $U_2$ and $V_3$ be subpolyhedra of $V$ encoded by $G_1$, $G_2$ and $G_3$, respectively. 
Note that each $V_1$ and $V_3$ has one cut end and $U_2$ has two. 
Set $\gamma_i=\partial V_i$ for $i\in\{1,3\}$, 
and note that it is a lift of $e_i$. 
Let $\gamma_2$ and $\gamma'_2$ denote the cut ends of $U_2$, 
and also note that $\gamma_2$ and $\gamma'_2$ are lifts of $e_2$ and $e'_2$, respectively. 
By Lemma~\ref{lem:pi_1_of_subpolyhedron1}, 
we have $\pi_1(V_i)\cong \langle \gamma_i \mid \gamma_i^{2^{k_i}}\rangle$ 
for some $k_i\in\Z_{\geq0}$, 
and hence there is a $2$-chain $c_i$ in $V_i\subset V$ such that $\partial c_i=2^{k_i}[\gamma_i]$ for $i\in\{1,3\}$. 
We then define a $2$-chain $C_1$ according to the order of $[\gamma_2]$; 
\begin{itemize}
\item
in the case $p_2[\gamma_2]=0$ in $H_1(U_2)$ for some $p_2\in\Z_{>0}$, 
there exists a $2$-chain $c_2$ in $U_2$ with $\partial c_2=p_2[\gamma_2]$, 
and then set $C_1=c_2-p_2[R]$; 
\item
in the case where $[\gamma_2]$ is not a torsion element in $H_1(U_2)$, 
we have $p'_2[\gamma'_2]=p_2[\gamma_2]$ in $H_1(U_2)$ 
for some $p_2,p'_2\in\Z\setminus\{0\}$ by Lemma~\ref{lem:pi_1_of_subpolyhedron2}, 
and then set $C_1=p'_2c_1-p_2[R]$. 
\end{itemize}
Define another $2$-chain as $C_3=c_3-2^{k_3}[R]$. 
These $2$-chains $C_1$ and $C_3$ are homology cycles in $H_2(V)$ 
since $[\gamma_1]=[\gamma'_2]$ and $[\gamma_2]=[\gamma_3]=\partial[R]$. 
Then we have $Q(C_1,C_3)=p_22^{k_3}\gl(R)$, which must be $0$. 
Hence $\gl(R)=0$. 
\end{proof}

\begin{lemma}
\label{lem:leaf}
Suppose that $X'$ contains a simply-connected subpolyhedron $V$ 
with one cut end such that a lift of the cut end has a vertical compressing disk.
Then $K$ admits a shadow obtained from $X$ by replacing $V$ with a $2$-disk. 
\end{lemma}
\begin{figure}[tbp]
\includegraphics[width=60mm]{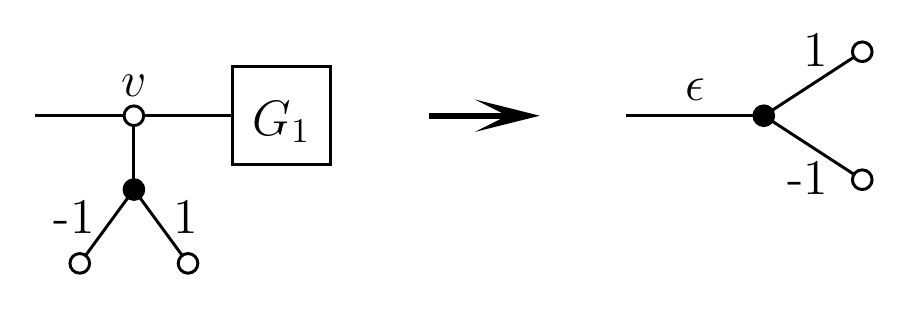}
\caption{The modification of a decorated graph as in Lemma~\ref{lem:leaf}. }
\label{fig:graph_lemma_leaf5}
\end{figure}
\begin{proof}
We give the proof by using decorated graphs. 
The assumption in the statement implies that 
a decorated graph of $X'$ has 
a subgraph $G$ as shown in the left of Figure~\ref{fig:graph_lemma_leaf5}, 
where the subgraph $G_1$ in the figure corresponds to $V$. 
It is enough to modify the graph as in the figure, 
where $\epsilon\in\{-\frac12,0,\frac12,1\}$. 

Let $v$ be the vertex of type~$\vt{P}$ shown in the left of Figure~\ref{fig:graph_lemma_leaf5}. 
By Lemma~\ref{lem:pi_1_of_subpolyhedron1-2}, 
there exists a leaf in $G_1$ such that all the vertices of type~$\vt{Y{}\!_{12}}$ contained 
in the geodesic from $v$ to the leaf are two-sided to $v$. 
Taking the union of all such geodesics $\ell_1,\ldots,\ell_m$, 
and it will be denoted by $T_\ell=\bigcup_{i=1}^m\ell_i$, 
which is a subtree of $G$ whose leaves are of type $\vt{D}$ execpt for $v$. 
Note that a vertex of type~$\vt{P}$ contained in $T_\ell\setminus v$ is 
a trivalent vertex even in $T_\ell$. 
We divide the proof into the following three cases: 
\begin{enumerate}
 \item
$T_\ell\setminus v$ does not contain a vertex of type~$\vt{Y{}\!_{111}}$ nor $\vt{P}$; 
 \item
$T_\ell\setminus v$ contains vertices of type~$\vt{Y{}\!_{111}}$ or $\vt{P}$, 
and farthest one from $v$ among them is of type~$\vt{Y{}\!_{111}}$; 
 \item
$T_\ell\setminus v$ contains vertices of type~$\vt{Y{}\!_{111}}$ or $\vt{P}$, 
and farthest one from $v$ among them is of type~$\vt{P}$. 
\end{enumerate}

\noindent (1) 
In this case, $T_\ell$ is a line, and 
all the vertices between $v$ and the leaf are of type~$\vt{Y{}\!_{12}}$. 
We can modify $G$ as in Figure~\ref{fig:graph_lemma_leaf5} by using moves-(f), -(g) and -(h). 

\begin{figure}[tbp]
\includegraphics[width=1\hsize]{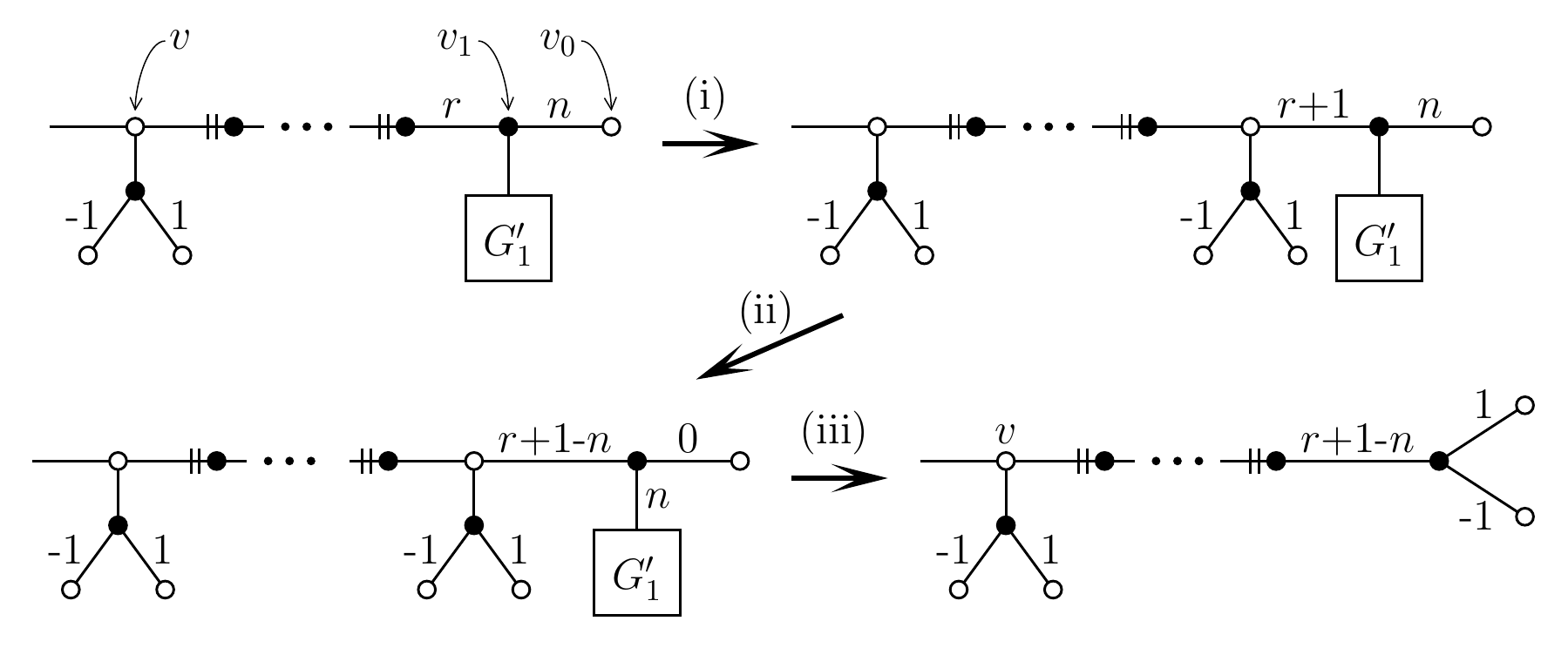}
\caption{The case (2) in the proof of Lemma~\ref{lem:leaf}. }
\label{fig:graph_lemma_leaf8}
\end{figure}
\noindent (2) 
Let $v_1$ be the vertex of type~$\vt{Y{}\!_{111}}$ farthest from $v$. 
Then there is a geodesic $\ell_i$ containing $v_1$, 
and let $v_0$ be the other endpoint than $v$. 
The vertices between $v_0$ and $v_1$ are of type~$\vt{Y{}\!_{12}}$. 
If the edge incident to $v_0$ is decorated by $r\ne\frac12$, 
a move-(g) can be applied, 
which is contrary to Lemma~\ref{lem:graph_lemma8}. 
Hence we can eliminate all the vertices of type~$\vt{Y{}\!_{12}}$ 
between $v_0$ and $v_1$ using only moves-(f), 
and then we can assume that $v_0$ and $v_1$ are connected by one edge. 
Let $e$ denote this edge. 

If $\ell_i$ has a vertex of type~$\vt{Y{}\!_{111}}$ other than $v_1$, 
the edge $e$ is decorated with $0$ by Lemma~\ref{lem:graph_lemma8}. 
Then the vertex $v_1$ can be eliminated by a connected-sum reduction. 

If $\ell_i$ has no vertices of type~$\vt{Y{}\!_{111}}$ other than $v_1$, 
then $G$ is as shown in the upper left of Figure~\ref{fig:graph_lemma_leaf8}. 
The modifications in Figure~\ref{fig:graph_lemma_leaf8}-(i) and -(ii) are 
done by moves-(e) and -(a), respectively. 
The move in Figure~\ref{fig:graph_lemma_leaf8}-(iii) is a connected-sum reduction and a YV-move. 
The lower right graph in Figure~\ref{fig:graph_lemma_leaf8} 
can be modified as we required 
by moves-(h), -(d) and -(c). 

\begin{figure}[tbp]
\includegraphics[width=1\hsize]{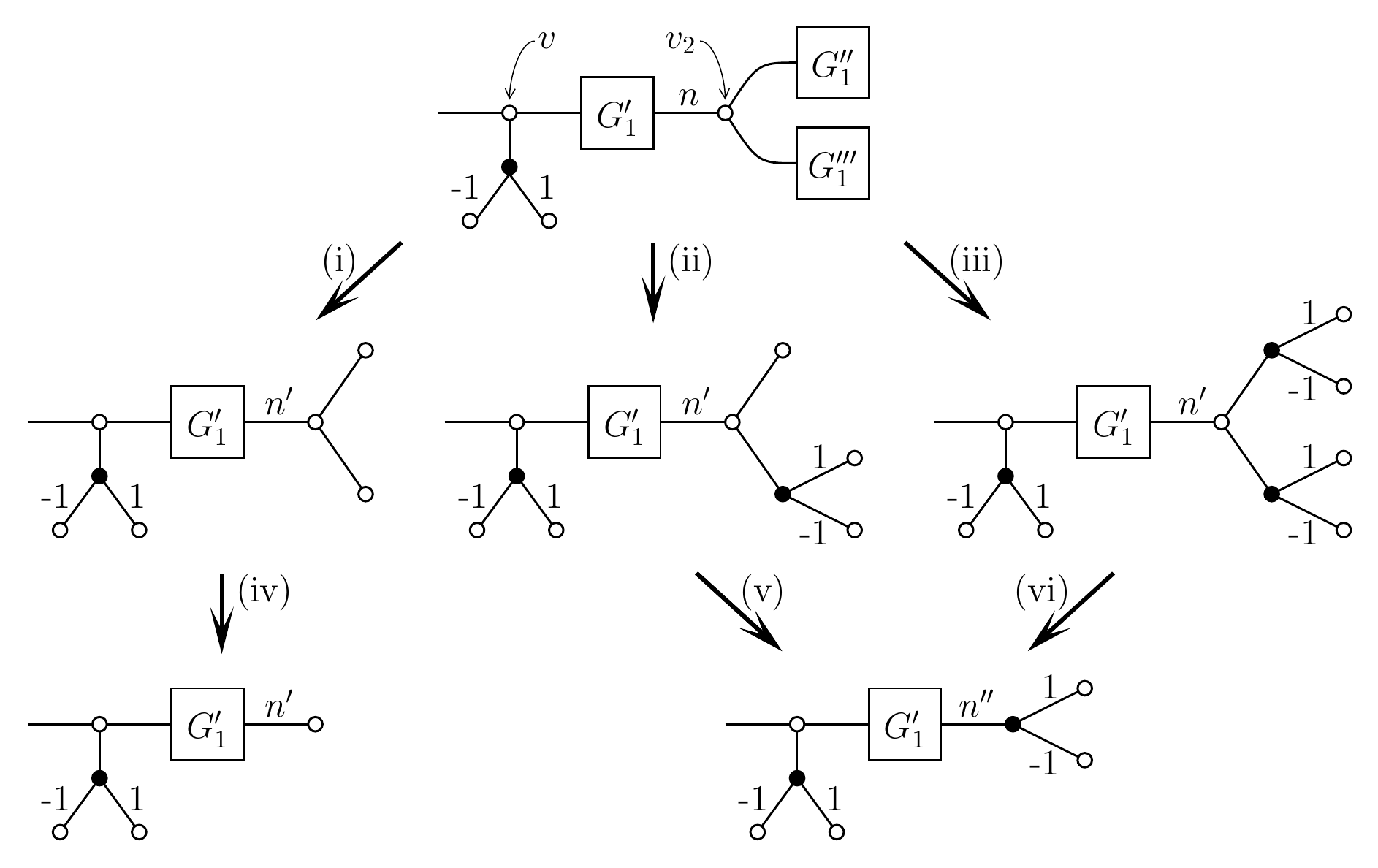}
\caption{The case (3) in the proof of Lemma~\ref{lem:leaf}. }
\label{fig:graph_lemma_leaf10}
\end{figure}
\noindent (3)
Let $v_2$ be the vertex of type~$\vt{P}$ farthest from $v$. 
Then $G$ is as the uppermost graph in Figure~\ref{fig:graph_lemma_leaf10}. 
Let $G'_1$, $G''_1$ and $G'''_1$ be subgraphs of $G_1$ as defined in the figure. 
The subgraphs $G_1''$ and $G_1'''$ do not contain vertices of type~$\vt{Y{}\!_{111}}$ nor $\vt{P}$. 
Then we can apply moves-(f), -(g) or -(h) as well as in (1) to these subgraphs, 
and $G$ is modified as shown in one of Figure~\ref{fig:graph_lemma_leaf10}-(i), -(ii) or -(iii). 
Moreover, 
the moves of Figure~\ref{fig:graph_lemma_leaf10}-(iv) and -(v) are obtained by YV-moves, 
and the move of Figure~\ref{fig:graph_lemma_leaf10}-(vi) 
is done by a move-(d). 
In either case, the vertex $v_2$ is eliminated, 
and we obtain the modification in Figure~\ref{fig:graph_lemma_leaf5} inductively. 
\end{proof}

\subsection{Existence of compressing disks}
For $i\in\{1,\ldots,11\}$, the polyhedron $X_i$ can be 
embedded in $\natural_2(S^1\times B^3)$ as a shadow, and 
the complement of $\partial X_i$ in $\partial (\natural_2(S^1\times B^3))$ ($=\#_2(S^1\times S^2)$) is a $3$-manifold with tori boundary. 
Note that this $3$-manifold $\#_2(S^1\times S^2)\setminus \partial X_i$ admits a complete hyperbolic structure with finite volume \cite{CT08}. 
In \cite{KMN18}, 
Dehn fillings on this $3$-manifold giving $\#_h(S^1 \times S^2)$ for some $h\in\Z_{\geq0}$ are 
strudied, and it leads to the following. 
\begin{lemma}
[\cite{KMN18}]
\label{lem:compressing_disks_X_3-4}
Suppose that $X$ contains a subpolyhedron $Y$ homeomorphic to $X_3$ or $X_4$. 
Then at least one of the following holds:
\begin{enumerate}
 \item
both of the components of $\partial Y$ have vertical compressing disks; or
 \item
the component of $\partial Y$ with length $1$ has a horizontal compressing disk. 
\end{enumerate}
\end{lemma}
\subsection{Banded unlink diagram of $2$-knot with complexity one}
%
Recall that, for $n\in\Z$, $K_n$ is a $2$-knot defined by the banded unlink diagram shown in Figure~\ref{fig:Kn}. 
We first prove the essential part of Theorem~\ref{thm:sc=1}: 
\begin{theorem}
\label{thm:sc=1_onlyifpart}
If a $2$-knot $K$ with $G(K)\not\cong\Z$ has shadow-complexity $1$, 
then $K$ is diffeomorphic to $K_n$ for some non-zero integer $n$. 
\end{theorem}
\begin{figure}[tbp]
\includegraphics[width=45mm]{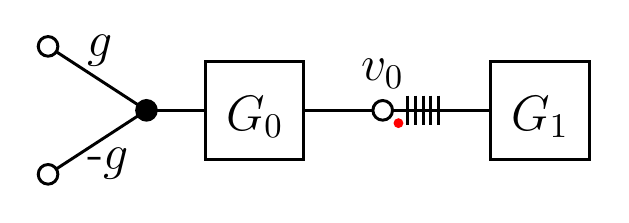}
\caption{A decorated graph encoding a shadow of $K$ having a subpolyhedron homeomorphic to $X_3$. }
\label{fig:graph_Kn1}
\end{figure}
\begin{figure}[tbp]
\includegraphics[width=90mm]{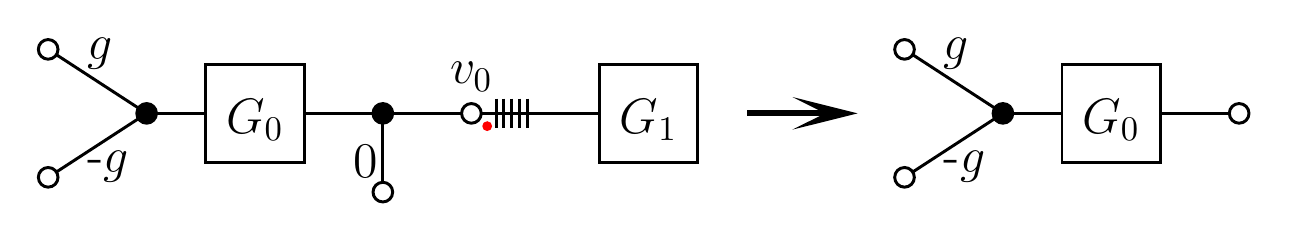}
\caption{A decorated graph encoding a shadow of $K$ having a subpolyhedron homeomorphic to $X_3$ 
such that the boundary component of the subpolyhedron with length $1$ has a horizontal compressing disk. }
\label{fig:graph_Kn2}
\end{figure}
\begin{figure}[tbp]
\includegraphics[width=90mm]{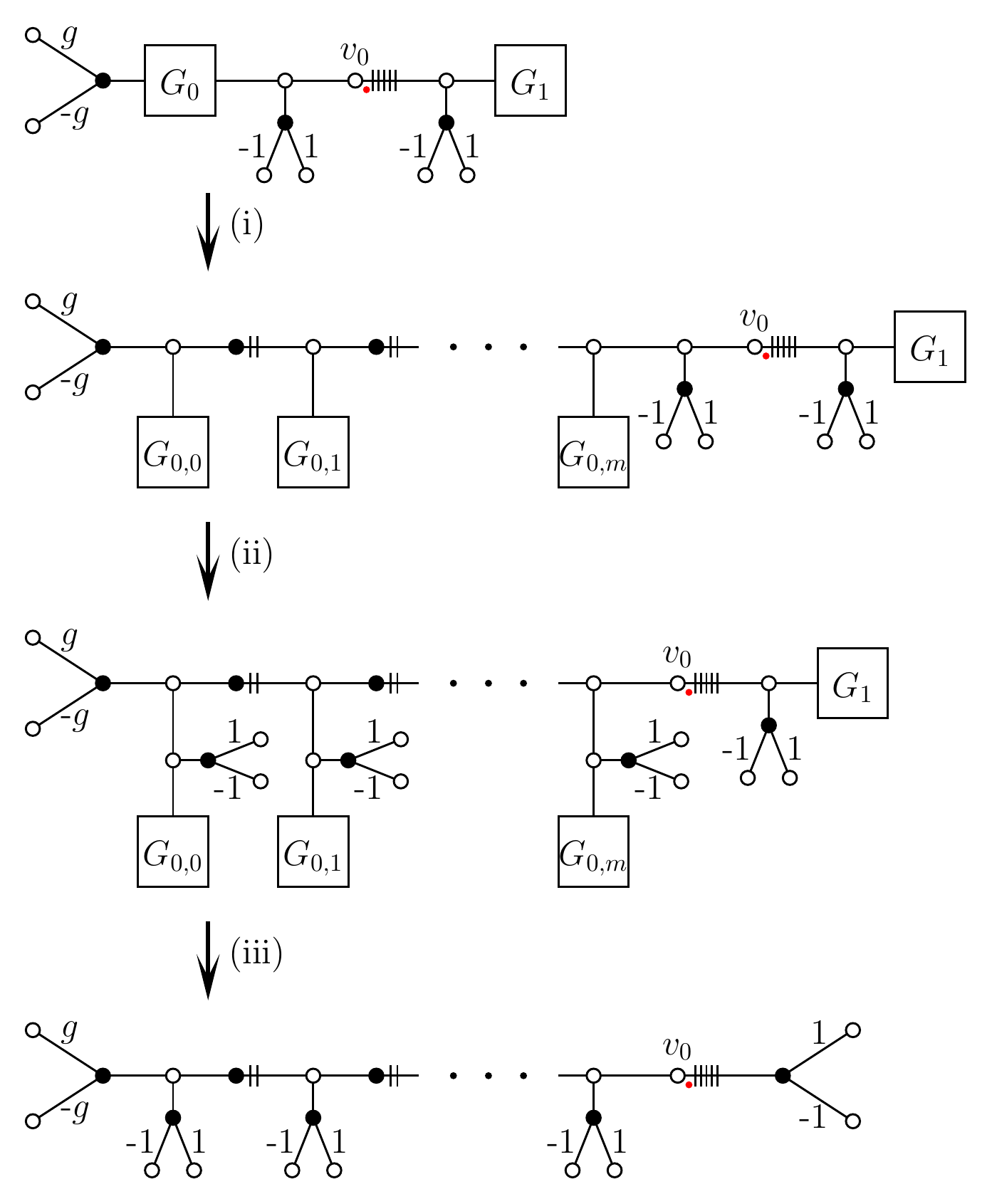}
\caption{A decorated graph encoding a shadow of $K$ having a subpolyhedron homeomorphic to $X_3$ 
such that the two boundary components of the subpolyhedron has vertical compressing disks. }
\label{fig:graph_Kn3}
\end{figure}
\begin{proof}
Let $X$ be a shadow of $K$ with $c(X)=1$, 
and let $G$ be a decorated tree graph for $X$. 
By Lemmas~\ref{lem:X34}, \ref{lem:X8910} and \ref{lem:X11}, 
$G$ has exactly one vertex $v_0$ of type~$\vt{X_3}$ or $\vt{X_4}$. 
Here we suppose that $v_0$ is of type~$\vt{X_3}$. 
Then $G$ is as shown in Figure~\ref{fig:graph_Kn1} by Lemma~\ref{lem:X34}. 

If Lemma~\ref{lem:compressing_disks_X_3-4}-(1) holds, 
then the graph $G$ can be assumed as shown in the left of Figure~\ref{fig:graph_Kn2}. 
Applying a connected-sum reduction, we obtain the right graph. 
This graph encodes a simple polyhedron without true vertex, 
which is not our focus here. 

We then assume that Lemma~\ref{lem:compressing_disks_X_3-4}-(2) holds, 
and the graph $G$ is as shown in the top of Figure~\ref{fig:graph_Kn3}. 
Let $U$ and $V$ be the subpolyhedra corresponding to the subgraphs $G_0$ and $G_1$, respectively. 
By Lemmas~\ref{lem:X34}, we have $\pi_1(U)\cong\langle \gamma,\gamma_1\mid \gamma^{2^m}\gamma_1\rangle$ and $V$ is simply-conneted. 
We apply Lemmas~\ref{lem:pi_1_of_subpolyhedron1-2} and \ref{lem:pi_1_of_subpolyhedron2-2} to $U$ and $V$, respectively, 
and then the graph $G$ can be assumed to be the second graph in Figure~\ref{fig:graph_Kn3}. 
Note that subgraphs $G_{1,0},G_{1,1}\ldots,G_{1,m}$ in the figure 
encode simply-connected subpolyhedra by Lemma~\ref{lem:pi_1_of_subpolyhedron2-2}. 
The move in Figure~\ref{fig:graph_Kn3}-(ii) is done by iterating moves-(e). 
and the move in Figure~\ref{fig:graph_Kn3}-(iii) is done by that in Figure~\ref{fig:graph_lemma_leaf5} 
(c.f. Lemma~\ref{lem:leaf}). 

From the bottom graph in Figure~\ref{fig:graph_Kn3}, 
we obtain a banded unlink diagram shown in the top of Figure~\ref{fig:shadow_Kn_cal2}. 
\begin{figure}[tbp]
\includegraphics[width=.95\hsize]{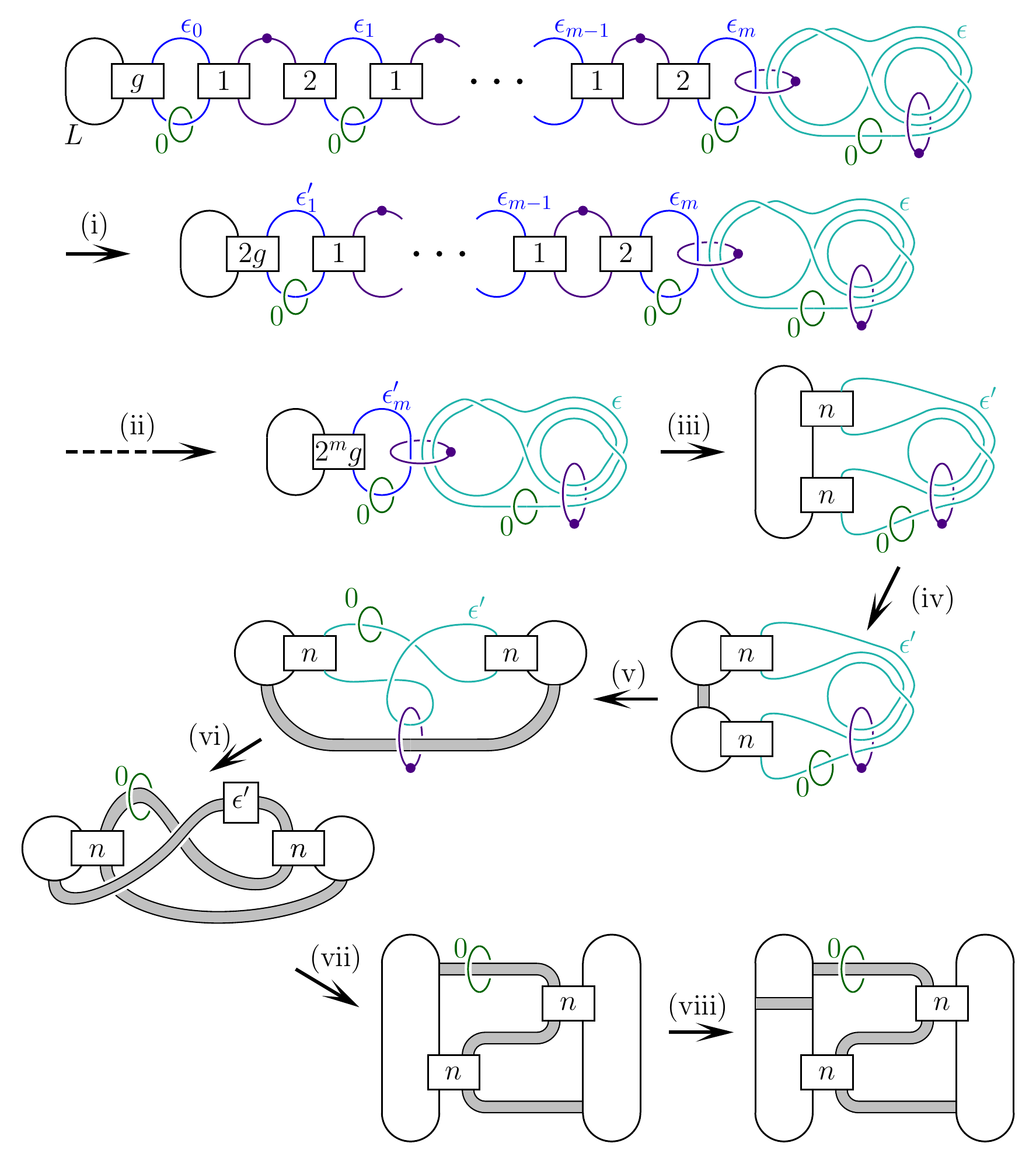}
\caption{A decorated graph of a shadow of $2$-knot with complexity one. }
\label{fig:shadow_Kn_cal2}
\end{figure}
We refer the reader to \cite{KN20} for a translation of a shadow into a Kirby diagram, and see also Remark~\ref{rmk:Kirby_diag} and \cite{CT08, Mar05}. 
Though the framings $\epsilon_0, \epsilon_1,\ldots,\epsilon_m$ and $\epsilon$ are determined from the gleams, 
each of them can be assumed to be $0$ or $1$ since there is a $0$-framed knot as a meridian. 
The first move of Figure~\ref{fig:shadow_Kn_cal2}-(i) is obtained by handle-slides and a cancellation of a 1-2 pair. 
We iterate the same process in Figure~\ref{fig:shadow_Kn_cal2}-(ii). 
The move~(iii) is obtained by handle-slides and a cancellation of a 1-2 pair, and we set $n=2^mg$ and $\epsilon'=0$ or $1$. 
The move~(iv) is done by a cup and 2-handle band swims. 
The move~(v) is an isotopy, and (iv) is obtained by a 2-handle band slide and a cancellation of a 1-2 pair. 
The move~(vii) is done by an isotopy if $\epsilon'=1$, and we also need 2-handle band slides if $\epsilon'=0$. 
The move~(viii) is obtained by a cap and 2-handle band slides. 
Finally, applying a 2-handle band swim and a cancellation of a 2-3 pair, 
we obtain the diagram shown in Figure~\ref{fig:Kn}. 

One can show the case where $v_0$ is of type~$\vt{X_4}$ in a similar way to the above, so we skip the details. 
\end{proof}

\begin{remark}
\label{rmk:Kirby_diag}
The method of a translation of a shadow only to a Kirby diagram is treated in \cite{KN20}, and that to a banded unlink diagram is actually not discussed. 
However, we can draw a diagram as shown in Figure~\ref{fig:shadow_Kn_cal2} 
by considering a decomposition $X=\Nbd(K;X)\cup X'$ and using \cite[Lemmas~1.1 and 1.2]{KN20}. 
Note that $\Nbd(K;X)$ is a shadow of $\Nbd(K;S^4)\cong S^2\times D^2$ and 
$\partial \Nbd(K;X)$ is a knot in $\partial \Nbd(K;S^4)\cong S^2\times S^1$ such that it winds $g$ times along $\{\mathrm{pt.}\}\times S^1$. 
\end{remark}
\begin{remark}
\label{rmk:Alex_poly}
If $v_0$ is of type~$\vt{X_3}$, 
the $2$-knot $K$ is diffeomorphic to $K_n$ with $n=2^mg>0$. 
On the other hand, if $v_0$ is of type~$\vt{X_4}$, 
the $2$-knot $K$ is diffeomorphic to $K_n$ with $n=-2^mg<0$.  
\end{remark}

The following implies that there exist infinitely many $2$-knots with shadow-complexity $1$. 
\begin{proposition}
\label{prop:Kn_diffeo}
The $2$-knots $K_n$ and $K_{n'}$ are not equivalent unless $n=n'$. 
\end{proposition}
\begin{proof}
From Lemma~\ref{lem:X34} and Remark~\ref{rmk:Alex_poly}, 
we have 
\[
G(K_n)\cong\langle x,y\mid x^2y^{|n|}x^{-1}y^{-n} \rangle, 
\]
and its Alexander polynomial is given as following: 
\begin{align*}
\varDelta_{K_n}(t)=
\begin{cases}
2-t^n          &\text{if }n\geq0,\\
1-t^{n}+t^{2n} &\text{if }n<0.
\end{cases}
\end{align*}
If $n\ne n'$, two polynomials $\varDelta_{K_n}(t)$ and $\varDelta_{K_{n'}}(t)$ are distinct. 
\end{proof}
At last, we give the proof of the complexity one case. 

\begin{theorem}
\label{thm:sc=1}
A $2$-knot $K$ with $G(K)\not\cong\Z$ has shadow-complexity $1$ if and only if 
$K$ is diffeomorphic to $K_n$ for some non-zero integer $n$. 
\end{theorem}
\begin{proof}
The only if part has been already discussed in Theorem~\ref{thm:sc=1_onlyifpart}. 

Let $n$ be an arbitrary non-zero integer. 
The banded unlink diagram of $K_n$ shown in Figure~\ref{fig:Kn} 
can be obtained from a shadow encoded in Figure~\mbox{\ref{fig:shadow_Kn}-(i)} or -(ii) 
in the same way as in the proof of Theorem~\ref{thm:sc=1_onlyifpart}. 
Therefore, $\shco(K_n)\leq 1$, and hence $\shco(K_n)=1$ by Proposition~\ref{prop:Kn_diffeo} and Theorem~\ref{thm:complexity0}. 
\end{proof}
\begin{figure}[tbp]
\includegraphics[width=.8\hsize]{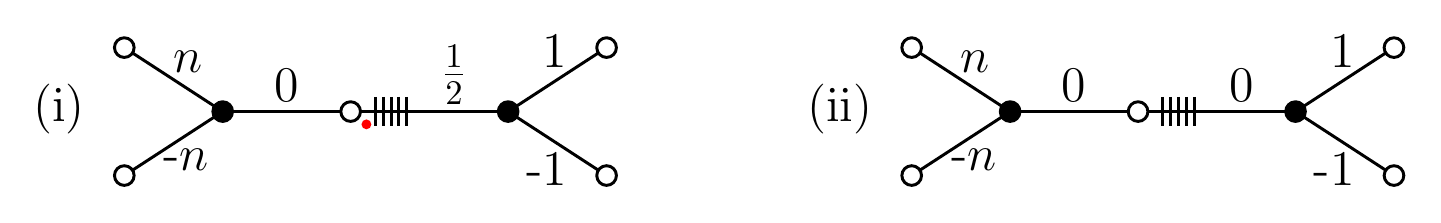}
\caption{A decorated graph of a shadow of $2$-knot with complexity one. }
\label{fig:shadow_Kn}
\end{figure}
\begin{remark}
\label{rmk:not_finite-to-one}
Let $X$ be a shadow of $K_n$ encoded by a decorated graph as shown in Figure~\mbox{\ref{fig:shadow_Kn}}. 
Its singular set $S(X)$ has 3 connected components: two circles and one 8-shaped graph. 
Then we can obtain a special shadow of $K_n$ from $X$ by applying $(0\to2)$-moves twice \cite{Cos05,Tur94}, and hence we have $\spshco(K_n)\leq5$. 
This implies that the special shadow-complexity for $2$-knots is not 
a finite-to-one invariant, 
while that for closed $4$-manifolds is finite-to-one \cite[Corollary~2.7]{Mar05}. 
\end{remark}



\begin{thebibliography}{99}
\bibitem{Art25}
E. Artin, 
{\it Zur Isotopie zweidimensionaler F\"achen im $R_4$}, 
Abh. Math. Sem. Univ. Hamburg {\bf 4} (1925) 174--177.

\bibitem{BCTT22}
R. Blair, M. Campisi, S. A. Taylor and M. Tomova, 
{\it Kirby–Thompson distance for trisections of knotted surfaces}
J. London Math. Soc. {\bf 105} (2022), no. 2, 765--793. 




\bibitem{CRS97}
J. S. Carter, J. H. Rieger and M. Saito, 
{\it A combinatorial description of knotted surfaces and their isotopies},
Adv. Math. {\bf 127} (1997), no. 1, 1--51.

\bibitem{Cos05}
F. Costantino, 
{\it Shadows and branched shadows of $3$ and $4$-manifolds},
Scuola Normale Superiore, Edizioni della Normale, Pisa, Italy, 2005.

\bibitem{Cos06}
F. Costantino,
{\it Stein domains and branched shadows of $4$-manifolds},
Geom. Dedicata {\bf 121} (2006), 89--111.

\bibitem{Cos06b}
F. Costantino,
{\it Complexity of $4$-manifolds}, 
Experiment. Math. {\bf 15} (2006), no. 2, 237--249.

\bibitem{Cos08}
F. Costantino,
{\it Branched shadows and complex structures on $4$-manifolds},
J. Knot Theory Ramifications {\bf 17} (2008), no. 11, 1429--1454.

\bibitem{CT08}
F. Costantino and D.~Thurston,
{\it $3$-manifolds efficiently bound $4$-manifolds},
J. Topol. {\bf 1} (2008), no. 3, 703--745. 

\bibitem{FQ90}
M. H. Freedman and F. Quinn, 
{\it Topology of 4-manifolds}, 
Princeton Mathematical Series {\bf 39}. Princeton University Press, Princeton, NJ, 1990. 

\bibitem{GK16}
D. Gay and R. Kirby, 
{\it Trisecting $4$-manifolds}, 
Geom. Topol. {\bf 20} (2016), no. 6, 3097--3132. 

\bibitem{GST10}
R.E. Gompf, M. Scharlemann and A. Thompson, 
{\it Fibered knots and potential counterexamples to the Property $2$R and Slice-Ribbon Conjectures},
Geom. Topol. {\bf 14} (2010) 2305--2347. 


\bibitem{HKM20}
M. C. Hughes, S. Kim and M. Miller, 
{\it Isotopies of surfaces in $4$-manifolds via banded unlink diagrams}, 
Geom. Topol. {\bf 24} (2020), no. 3, 1519--1569.

\bibitem{IK17}
M.~Ishikawa and Y.~Koda, 
{\it Stable maps and branched shadows of $3$-manifolds},
Math. Ann. {\bf 367} (2017), no. 3-4, 1819--1863.

\bibitem{IN20}
M. Ishikawa and H. Naoe, 
{\it Milnor fibration, A'Campo's divide and Turaev's shadow}, 
Singularities ----- Kagoshima 2017, 
Proceedings of the 5th Franco-Japanese-Vietnamese Symposium on Singularities, 
World Scientific Publishing, 2020, pp. 71--93. 

\bibitem{Jab16}
M. Jab\l onowski, 
{\it On a banded link presentation of knotted surfaces}, 
J. Knot Theory Ramifications {\bf 25} (2016), no. 3, 1640004, 11 pp. 

\bibitem{Kaw21}
A. Kawauchi, 
{\it Ribbonness of a stable-ribbon surface-link, I. A stably trivial surface-link}
Topology Appl. {\bf 301} (2021), Paper No. 107522, 16 pp. 

\bibitem{Kaw22}
A. Kawauchi, 
{\it Uniqueness of an orthogonal $2$-handle pair on a surface-link}, 
preprint (2022), added as a separated paper in arXiv:1804.02654v12. 

\bibitem{KSS82}
A. Kawauchi, T. Shibuya and S. Suzuki, 
{\it Descriptions on surfaces in four-space, I, Normal forms},
Math. Sem. Notes Kobe Univ. {\bf 10} (1982), 72--125. 

\bibitem{KK08}
C. Kearton and V. Kurlin, 
{\it All $2$–dimensional links in $4$–space live inside a universal $3$–dimensional polyhedron}, 
Algebr. Geom. Topol. {\bf 8} (2008), no. 3, 1223--1247. 

\bibitem{KMN18}
Y. Koda, B. Martelli and H. Naoe, 
{\it Four-manifolds with shadow-complexity one}, 
to appear in Ann. Fac. Sci. Toulouse. 

\bibitem{KN20}
Y. Koda and H. Naoe, 
{\it Shadows of acyclic $4$-manifolds with sphere boundary},
Algebr. Geom. Topol. {\bf 20} (2020), no. 7, 3707--3731

\bibitem{LP72}
F. Laudenbach, V. Po\'{e}naru, 
{\it A note on $4$-dimensional handlebodies}, 
Bull. Soc. Math. France {\bf 100} (1972), 337--344.

\bibitem{Lom81}
S. J. Lomonaco Jr., 
{\it The Homotopy Groups of Knots I. How to Compute the Algebraic $2$-Type},
Pacific J. Math. {\bf 95} (1981), no. 2, 349--390.

\bibitem{Mat03}
S Matveev, 
{\it Algorithmic topology and classification of $3$–manifolds}, 
Algorithms Comput. Math. 9, Springer (2003). 

\bibitem{Mar05} 
B. Martelli, 
{\it Links, two-handles, and four-manifolds}, 
Int. Math. Res. Not. IMRN  {\bf 2005},  no. 58, 3595--3623. 

\bibitem{Mar11} 
B. Martelli, 
{\it Four-manifolds with shadow-complexity zero}, 
Int. Math. Res. Not. IMRN  {\bf 2011},  no. 6, 1268--1351. 

\bibitem{MZ17}
J. Meier and A. Zupan, 
{\it Bridge trisections of knotted surfaces in $S^4$}, 
Trans. Amer. Math. Soc. {\bf 369} (2017), no. 10, 7343--7386. 

\bibitem{MZ18}
J. Meier and A. Zupan, 
{\it Bridge trisections of knotted surfaces in $4$-manifolds}, 
Proc. Natl. Acad. Sci. USA {\bf 115} (2018), no. 43, 10880--10886.


\bibitem{Nao17}
H. Naoe, 
{\it Shadows of $4$-manifolds with complexity zero and polyhedral collapsing}, 
Proc. Amer. Math. Soc. {\bf 145} (2017), no. 10, 4561--4572. 


\bibitem{Ros98}
D. Roseman, 
{\it Reidemeister-type moves for surfaces in four-dimensional space}, 
Knot theory (Warsaw, 1995), 347--380,
Banach Center Publ., 42, Polish Acad. Sci. Inst. Math., Warsaw, 1998.

\bibitem{SS05}
M. Saito and S. Satoh, 
{\it The spun trefoil needs four broken sheets},  
J. Knot Theory Ramifications {\bf 14} (2005), no. 7, 853--858.

\bibitem{Sat00}
S. Satoh, 
{\it On non-orientable surfaces in $4$-space which are projected with at most one triple point}, 
Proc. Amer. Math. Soc. {\bf 128} (2000), no. 9, 2789--2793.


\bibitem{Sat09}
S. Satoh,
{\it Triviality of a $2$-knot with one or two sheets}, 
Kyushu J. Math. {\bf 63} (2009), no. 2, 239--252. 


\bibitem{SS04}
S. Satoh and A. Shima, 
{\it The $2$-twist-spun trefoil has the triple point number four}, 
Trans. Amer. Math. Soc. {\bf 356} (2004), no. 3, 1007--1024.



\bibitem{Swe01}
F. J. Swenton, 
{\it On a calculus for $2$-knots and surfaces in $4$-space}, 
J. Knot Theory Ramifications {\bf 10} (2001), no. 8, 1133-1141. 

\bibitem{Tur94}
V.G. Turaev, 
{\it Quantum invariants of knots and $3$-manifolds}, 
De Gruyter Studies in Mathematics, vol 18, Walter de Gruyter \& Co., Berlin, 1994.

\bibitem{Yaj64}
T. Yajima, 
{\it On simply knotted spheres in $R^4$}, 
Osaka J. Math. {\bf 1} (1964), 133--152.

\bibitem{Yos94}
K. Yoshikawa, 
{\it An enumeration of surfaces in four-space}, 
Osaka J. Math. \textbf{31} (1994), no. 3, 497--522.

\bibitem{Zee65}
E.C. Zeeman, 
{\it Twisting spun knots}, 
Trans. Amer. Math. Soc. {\bf 115} (1965), 471--495. 

\end{thebibliography}
\end{document}